\documentclass[journal,onecolumn]{IEEEtran}

\usepackage{amsmath,amsfonts,bm}
\usepackage{amsthm,amssymb,mathrsfs}

\newtheorem{thm}{\textbf{Theorem}}
\newtheorem{coro}{\textbf{Corollary}}
\newtheorem{assum}{\textbf{Assumption}}
\newtheorem{defn}{\textbf{Definition}}
\newtheorem{props}{\textbf{Proposition}}
\newtheorem{lemma}{\textbf{Lemma}}
\newtheorem{remark}{\textbf{Remark}}
\newtheorem{problem}{\textbf{Problem}}

\newcommand{\df}{\mathrm{d}}
\newcommand{\ef}{\mathrm{e}}

\newcommand{\supp}{\mathrm{supp}}

\def\eqref#1{(\ref{#1})}

\def\1{\bm{1}}

\def\eps{{\epsilon}}

\def\rvxi{{\mathbf{\xi}}}

\def\rvg{{\mathbf{g}}}
\def\rvh{{\mathbf{h}}}
\def\rvu{{\mathbf{i}}}

\def\rvs{{\mathbf{s}}}

\def\rvu{{\mathbf{u}}}

\def\rvx{{\mathbf{x}}}
\def\rvy{{\mathbf{y}}}

\def\rmS{{\mathbf{S}}}

\def\rmX{{\mathbf{X}}}
\def\rmY{{\mathbf{Y}}}

\def\mA{{\bm{A}}}

\def\mI{{\bm{I}}}

\def\mS{{\bm{S}}}

\DeclareMathAlphabet{\mathsfit}{\encodingdefault}{\sfdefault}{m}{sl}
\SetMathAlphabet{\mathsfit}{bold}{\encodingdefault}{\sfdefault}{bx}{n}

\def\gA{{\mathcal{A}}}

\def\gD{{\mathcal{D}}}

\def\gH{{\mathcal{H}}}

\def\gP{{\mathcal{P}}}

\def\gS{{\mathcal{S}}}

\def\gX{{\mathcal{X}}}
\def\gY{{\mathcal{Y}}}

\def\sN{{\mathbb{N}}}

\def\sR{{\mathbb{R}}}

\def\sT{{\mathbb{T}}}

\def\cA{{\mathscr{A}}}
\def\cB{{\mathscr{B}}}

\def\cF{{\mathscr{F}}}

\newcommand{\E}{\mathbb{E}}

\newcommand{\R}{\mathbb{R}}

\newcommand{\diam}{\mathrm{diam}}

\DeclareMathOperator*{\argmin}{arg\,min}

\usepackage{amsthm,amsmath,amsfonts,amssymb}
\usepackage[numbers,sort&compress]{natbib}
\usepackage[hidelinks]{hyperref}
\usepackage{algorithmic}
\usepackage{algorithm}
\usepackage{array}
\usepackage[caption=false,font=normalsize,labelfont=sf,textfont=sf]{subfig}
\usepackage{textcomp}
\usepackage{stfloats}
\usepackage{url}
\usepackage{verbatim}
\usepackage{graphicx}
\usepackage{booktabs}
\usepackage{xcolor}
\begin{document}

\title{On the Nonasymptotic Scaling Guarantee of Hyperparameter Estimation in Inhomogeneous, Weakly-Dependent Complex Network Dynamical Systems}

\author{Yi Yu, Yubo Hou, Yinchong Wang, Nan Zhang, Jianfeng Feng,~\IEEEmembership{Senior Member,~IEEE}, Wenlian Lu,~\IEEEmembership{Senior Member,~IEEE}
\thanks{The authors were jointly supported by the STCSM (No. 23JC1400800), the JiHua Laboratory S\&T Program (No. X250881UG250), the Lingang Laboratory, Grant (No. LGL-1987), and the key projects of "Double First-Class" initiative of Fudan University (\textit{corresponding author: Wenlian Lu, Jianfeng Feng, Nan Zhang}).

Yi Yu is with the School of Mathematical Sciences, Fudan University, Shanghai 200433, China (e-mail: y\_yu21@m.fudan.edu.cn).

Yubo Hou is with the School of Mathematical Sciences, Fudan University, Shanghai 200433, China (e-mail: 22110180017@m.fudan.edu.cn).

Yinchong Wang is with the School of Mathematical Sciences, Fudan University, Shanghai 200433, China (e-mail: 20110180040@fudan.edu.cn).

Nan Zhang is with the Institute of Science and Technology for Brain-Inspired Intelligence, Fudan University, Shanghai
200433, China, and the School of Data Science, Fudan University, Shanghai 200433, China (e-mail: zhangnan@fudan.edu.cn).

Jianfeng Feng is with the Institute of Science and Technology for Brain-Inspired Intelligence, Fudan University, Shanghai
200433, China; Key Laboratory of Computational Neuroscience and Brain-Inspired Intelligence (Fudan University), Ministry of Education Fudan University, Shanghai 200433, China; Shanghai Center for Mathematical Sciences, Fudan University, Shanghai 200433, China; the Department of Computer Science, University of Warwick, Coventry CV47AL, UK; Zhang Jiang Fudan International Innovation Center, Fudan University, Shanghai 200433, China, and Ji Hua Laboratory, Foshan, China (e-mail: jianfeng64@gmail.com).

Wenlian Lu is with the Center for Applied Mathematics, Fudan University, Shanghai 200433, China; the School of Mathematical Sciences, Fudan University, Shanghai 200433, China; Shanghai Center for Mathematical Sciences, Fudan University, Shanghai 200433, China; the Institute of Science and Technology for Brain-Inspired Intelligence, Fudan University, Shanghai 200433, China; Ji Hua Laboratory, Foshan, China and Shanghai Key Laboratory of Contemporary Applied Mathematics, Fudan University, Shanghai 200433, China (e-mail: wenlian@fudan.edu.cn).}
}



\maketitle

\begin{abstract}
Hierarchical Bayesian models are increasingly used in large, inhomogeneous complex network dynamical systems by modeling parameters as draws from a hyperparameter-governed distribution. However, theoretical guarantees for these estimates as the network population size grows have been lacking. A critical concern is that hyperparameter estimation may diverge for larger networks, undermining the model's reliability. Formulating the system's evolution in a measure transport perspective, we propose a theoretical framework for estimating hyperparameters with mean-type observations, which are prevalent in many scientific applications. Our primary contribution is a nonasymptotic bound for the deviation of estimate of hyperparameters in inhomogeneous complex network dynamical systems with respect to network population size, which is established for a general family of optimization algorithms within a fixed observation duration. For systems with independent nodes, we first establish a fixed-accuracy probabilistic hyperparameter bound and then use it to derive an explicit nonasymptotic hyperparameter bound. We subsequently extend both bounds to the more challenging and practically relevant setting of systems with weakly-dependent nodes. We validate our theoretical findings with numerical experiments on two representative models: a Susceptible-Infected-Susceptible model and a Spiking Neuronal Network model. In both cases, the results confirm that the estimation error decreases as the network population size increases, aligning with our theoretical guarantees. This research proposes the foundational theory to ensure that hierarchical Bayesian methods are statistically consistent for large-scale inhomogeneous systems, filling a gap in this area of theoretical research and justifying their application in practice.
\end{abstract}




\section{Introduction} \label{sec:introduction}
Dynamical systems are widely used for modeling realistic phenomena across diverse fields including physics, biology and social science \cite{temam2012infinite, pastor2015epidemic, vega2007complex, castellano2009statistical, smith2003development, vogels2005neural, tyson2001network, snijders2001statistical}. An important class is the complex network dynamical system, which describes the evolution of interconnected nodes \cite{strogatz2001exploring, boccaletti2006complex, vogels2005neural, tyson2001network, snijders2001statistical, pastor2015epidemic, vega2007complex, keeling2005networks, albert2002statistical, newman2003structure}. The applications of this framework are vast. For example, the Susceptible-Infected-Resistant (SIR) model and its variants are used to study the spread of infectious diseases \cite{anderson1991infectious, daley1964epidemics, boccaletti2006complex, kermack1927contribution, cao2022mepognn, hethcote2000mathematics, keeling2005networks} such as Covid-19 \cite{he2020seir, basnarkov2021seair, ghostine2021extended}, and have even been adapted for cybersecurity \cite{lu2013optimizing, xu2019cybersecurity}. In neuroscience, brain activities are modeled as a complex network of evolving neuron states to analyze brain functions \cite{deary2010neuroscience, david2004evaluation, lu2024imitating, peyrache2012spatiotemporal, deco2012ongoing}. Obtaining accurate information from complex network dynamical systems is crucial for understanding their function, informing decision-making, and controlling their behavior \cite{morris2021optimal, eshghi2015optimal, van2011decreasing, anderson1991infectious}. For this purpose, a widely utilized method is to describe the dynamical system with a parameterized model and then estimate the system parameters from recorded observations. For a complex network dynamical system with network population size $n$, where $n$ is exactly the number of nodes, we denote the nodes' states at time $t$ as $\{\rvs_i(t)\}_{i=1}^n$, and the parameters associated with each node as $\theta_1, \dots, \theta_n$. The observations are typically gathered via a sensor function $f_t$ applied to all node states:
\begin{equation*}
    \rvy_t = f_t\left(\{\rvs_i(t)\}_{i=1}^n\right), \qquad 1\le t \le T.
\end{equation*}
There exist many research works focusing on estimating the system parameters via observations \cite{hogg1967some, winship1999estimation, neyman1948consistent, kar2012distributed, zhang2004impacts}. In data-rich scenarios, for instance, where each node can be observed individually or when multiple experimental replications are available, a variety of robust algorithms exist for parameter estimation, including classical methods like Maximum Likelihood estimation (MLE) \cite{neyman1948consistent, richards1961method, firth1993bias}, Least Square estimation (LSE) \cite{schmiester2021efficient, robins1994estimation, le2022linear}, and modern approaches based on neural networks \cite{nolan2021machine, chen2020machine, kutz2023machine, lober2024predictive, wei2025estimating, lakshminarayanan2024parameter, almanstotter2025pinnverse}. However, when handling real-world data modeled by a complex network dynamical system, the number of the parameters to be estimated ($n$) is often much larger than the size of the observable data. This high dimensionality of parameter space makes the simultaneous and accurate estimation of all individual parameters $\theta$ practically infeasible. One approach to overcome this challenge is to set the network system homogeneous, that is, to set all node parameters to be identical \cite{jiang2015defining, shi2019totally, bogacz2006homogeneous, deco2018perturbation}. However, the homogeneous setting could be far from the reality where the nodes are different from each other. The inhomogeneous setting is more reasonable in modeling, but estimating a huge amount of inhomogeneous parameters at once also faces many challenges, for example, over-fitting effect \cite{ruiz2013estimating}. However, in complex network dynamical system, the exact value of each parameter $\theta_i$ is frequently less interesting than the distribution of the parameters. A more powerful approach shifts the focus from estimating each parameter individually to estimating the distribution from which they are drawn. This is the central premise of the hierarchical Bayesian model. Hierarchical Bayesian model provide a framework to model parameters as random variables drawn from a probability distribution that is itself governed by a smaller set of hyperparameters \cite{allenby2006hierarchical, shiffrin2008survey, berliner1996hierarchical}. For the system with network population size $n$, the parameters $\{\theta_i\}_{i=1}^n$ are modeled as samples from a distribution parameterized by a hyperparameter $\rvh$, that is,
\begin{equation*}
    \theta_i \sim q(\cdot \mid \rvh), \qquad i=1,\dots,n.
\end{equation*}
By focusing on the hyperparameter $\rvh$, which governs the distribution of the node parameters and thereby the system's behavior, we can construct a reasonable and tractable model for a large, complex network dynamical system with limited observations. Owing to this advantage, hierarchical Bayesian model has wide applications among broad areas, such as inverse problems \cite{knapik2016bayes, stuart2010inverse, agapiou2013posterior, knapik2013bayesian, knapik2011bayesian, malinverno2004expanded}, biology \cite{gao2008gaussian, lesaffre2012bayesian}, neuroscience \cite{zhao2019bayesian, molloy2018hierarchical}, engineering and other areas \cite{lord2008effects, miranda2013bayesian, farid2017application, liang2020hierarchical, wikle2003hierarchical, rouder2005introduction, berliner1996hierarchical}.

As complex network dynamical systems scale, the network population size $n$ grows, and a critical concern arises that the hyperparameter estimations may diverge from their true values, rendering the underlying model unreliable. It is therefore significant to establish a theoretical guarantee for the consistency of hyperparameter estimation with respect to the growing network population size. While extensive theoretical work, such as \cite{neyman1948consistent, barndorff1983formula, fahrmeir1985consistency, robins1994estimation, zhan2024general, saumard2017concentration}, addresses consistency for classical parameter estimation methods like MLE and LSE in data-rich scenarios, corresponding theory for hierarchical Bayesian estimation are less developed. Existing studies have explored the consistency of hierarchical Bayesian methods in the context of inverse problems. For instance, \cite{knapik2016bayes} and \cite{dunlop2020hyperparameter} studied the consistency of hierarchical Bayesian estimation with respect to the order of function used for approximating the inverse problem, but the crucial question of consistency with respect to the network population size remains largely unaddressed. 

This question is also distinct from established consistency theory for Bayesian posterior distributions and contraction \cite{ghosal2000convergence}, empirical Bayes procedures, and hierarchical parameter learning \cite{chen2021consistency,rizzelli2025empirical}, because our asymptotic variable is the network population size under a fixed observation horizon.

Our work aims to fill this research gap by developing a rigorous theoretical framework to analyze the consistency of hyperparameter estimation in large-scale complex network dynamical systems. We considers hyperparameters governing the distributions of both the system's dynamic parameters and its initial states. By formulating the system's evolution from a measure transport perspective \cite{marzouk2016sampling, villani2008optimal}, we demonstrate how hyperparameter information is encoded into the nodes' states and, consequently, into the observations. We focus on mean-type observations, which appear as an average of certain nodes' features and are commonly used in many scientific domains \cite{peyrache2012spatiotemporal, buzsaki2012origin, linden2011modeling, david2004evaluation, da2013eeg,shaman2012forecasting, shaman2013real}. Our primary contribution is a nonasymptotic bound on the estimation error during a fixed observation window, which proves that the estimate converges to the true hyperparameter as the network population size grows. In this paper, we establish this consistency result not only for systems with independent nodes but also extend it to the more general and realistic case of weakly-dependent nodes. We validate our theoretical findings through numerical experiments on two representative models: a Susceptible-Infected-Susceptible (SIS) model \cite{shaman2012forecasting, shaman2013real, rasmussen2011inference, yang2014comparison} and a spiking neuronal network (SNN) model \cite{zhang2024framework, lu2024simulation, lu2024imitating}. The numerical results demonstrate that the empirical estimation error decreases as the network population size increases, in precise agreement with our theoretical guarantees. By providing this foundational support, our work ensures that practical methods like data assimilation \cite{zhang2024framework, law2015data, evensen2002sequential, lu2024simulation, ruiz2013estimating} remain statistically robust when applied to large-scale, inhomogeneous, weakly-dependent complex network dynamical systems. Furthermore, our framework is also applicable for Nonlinear Least Squares (NLS) estimation, as detailed in Appendix \ref{app:leastSquares}. The fundamental proof routine of the main results is inspired by concentration-of-measure argument \cite{devroye2013probabilistic, shalev2014understanding, vershynin2018high}. The work of \cite{escande2024concentration} derived similar theoretical results on independent sampling case, but the more challenging setting of large-scale, inhomogeneous, weakly-dependent network dynamics is not studied. Based on the general concentration results and an independent approximation lemma proposed by Rio \cite{rio2013inequalities}, our work extends the results to the inhomogeneous, weakly-dependent setting.

The remainder of this paper is organized as follows. In Section \ref{sec:problem_formulation}, we formally formulate the hyperparameter estimation problem for complex network dynamical systems in the perspective of hierarchical Bayesian modeling. Section \ref{sec:main_results} presents our main theoretical contributions. We begin by introducing essential preliminary concepts and technical lemmas in Section \ref{sec:preliminary}, then establish our first consistency guarantee for systems with independent nodes and heterogeneous finite-time marginals in Section \ref{sec:EstimationIID}. Following this, we present our core result in Section \ref{sec:EstimationWD}: the nonasymptotic guarantee for the setting of weakly-dependent nodes, for which we derive an explicit nonasymptotic error bound. In Section \ref{sec:Experiments}, we provide empirical validation for our theoretical claims by conducting numerical experiments on two representative systems: a Susceptible-Infectious-Susceptible (SIS) model and a Spiking Neuronal Network (SNN) model. Finally, Section \ref{sec:conclusion} concludes the paper by summarizing our findings and discussing their broader implications for the statistical modeling of large-scale systems. Appendix \ref{app:leastSquares} details the connection between our framework and classical Nonlinear Least Squares (NLS) estimation.

\section{Problem Formulation}\label{sec:problem_formulation}
In many complex network dynamical systems, the collective behavior is governed by node-specific parameters. Hierarchical Bayesian model provides a framework to model these inhomogeneous parameters as being drawn from a probability distribution that is itself governed by a set of unknown hyperparameters. Estimating these hyperparameters is key to understanding, simulating, and predicting the system's dynamics. A dynamical system can be formally described in either continuous or discrete time:
\begin{equation}\label{eq:dynSysGenODE}
    \begin{matrix}
            \left\{\begin{aligned}
                \frac{\df \rvs_i(t)}{\df t} &= g_i(\rvs_1(t), \dots, \rvs_n(t); \theta_1, \dots, \theta_n) \\
                \rvs_i(0) &\sim p_{0}(s\mid \rvh_{\mathrm{prob}}), \\
                \theta_i &\sim q(\cdot \mid \rvh_{\mathrm{dyn}}),
            \end{aligned}\right. &&\quad i = 1, \dots, n, \quad t \ge 0.
        \end{matrix} 
\end{equation}
In applications, the SNN model is often described in continuous time form \eqref{eq:dynSysGenODE}, while the SIS model is typically formulated in discrete time form \eqref{eq:dynSysGenDis}:
\begin{equation}\label{eq:dynSysGenDis}
    \begin{matrix}
            \left\{\begin{aligned}
                \rvs_i(t) &= g_i(\rvs_1(t-1), \dots, \rvs_n(t-1);  \theta_1, \dots, \theta_n), \text{ for } t = 1, 2, \dots\\
                \rvs_i(0) &\sim p_{0}(s\mid \rvh_{\mathrm{prob}}), \\
                \theta_i &\sim q(\theta \mid \rvh_{\mathrm{dyn}}), 
            \end{aligned}\right. &&\quad i = 1, \dots, n.
        \end{matrix} 
\end{equation}
In the above systems \eqref{eq:dynSysGenODE} and \eqref{eq:dynSysGenDis}, $\rvs_i(t)\in \gS \subset \R^k$ is the state of the $i$-th node at time $t$, and $g_i$ is the function governing the dynamics of the $i$-th node, which is determined by parameters $\theta = (\theta_1, \dots, \theta_n) \in \Theta$. Each parameter has marginal distribution $q(\theta \mid \rvh_{\mathrm{dyn}})$ determined by a hyperparameter $\rvh_{\mathrm{dyn}} \in \gH_{\mathrm{dyn}} \subset \R^{h_{\mathrm{dyn}}}$, and $p_{0}(s\mid \rvh_{\mathrm{prob}})$ is the common initial marginal distribution of the nodes' states, which is determined by a hyperparameter $\rvh_{\mathrm{prob}} \in \gH_{\mathrm{prob}} \subset \R^{h_{\mathrm{prob}}}$. Thus, the node states are identically distributed at the initial time, $\rvs_i(0)\sim p_0(\cdot\mid\rvh_{\mathrm{prob}})$ for $i=1,\dots,n$, but no identical-distribution condition is imposed at later times. The dynamics hyperparameter $\rvh_{\mathrm{dyn}}$ and the probability hyperparameter $\rvh_{\mathrm{prob}}$ are unknown and need to be estimated. Denote the system states at time $t$ as $\rmS_n(t) = (\rvs_1(t), \dots, \rvs_n(t))$. An $r_t$-dimensional observation of the system is often a result of an operator on the entire system state, that is, $\rvy_t = f_t(\rmS_n(t)) \in \R^{r_t}$ with $f_t:\gS^n \to \R^{r_t}$ being the operator at time $t$, where $r_t > 0$ represents the dimension of the observation space. In this work, we focus on the continuous time case \eqref{eq:dynSysGenODE} for clarity of presentation, while the discrete time case \eqref{eq:dynSysGenDis} can be treated similarly with proper continuous approximation schemes.

From the perspective of measure transport, the hyperparameters can be connected to the system's evolution through the finite-horizon ODE flow. Under Assumption \ref{assum:flowRegularity} in Appendix \ref{app:flowTransport}, the vector field in \eqref{eq:dynSysGenODE} generates a continuously differentiable flow map $G_t:\gS^n\to\gS^n$. The flow and its induced joint density satisfy
\begin{equation}\label{eq:dynTrans}
\left\{
\begin{aligned}
\left(\rmS_n(t)\mid\theta\right)
&=G_t(\rmS_n(0);\theta),\\
p_{t,n}(\mS,\theta\mid\rvh)
&=p_{0,n}\!\left(G_t^{-1}(\mS;\theta)\mid\rvh_{\mathrm{prob}}\right)
\left|\det\!\left(\nabla G_t^{-1}(\mS;\theta)\right)\right|
q_n(\theta\mid\rvh_{\mathrm{dyn}}).
\end{aligned}
\right.
\end{equation}
where $p_{0,n}$ and $q_n$ are the joint initial-state and parameter densities, and $p_{t,n}$ is the transported joint density of $\rmS_n(t)$ and $\theta$; in the independent specification, $p_{0,n}=p_0^n$ and $q_n=q^n$. The precise regularity conditions, the ODE derivation of $G_t$, and the proof of both relations in \eqref{eq:dynTrans} are given in Appendix \ref{app:flowTransport}. By integrating the second relation in \eqref{eq:dynTrans} over the other nodes' states and the parameters, we obtain the marginal distribution of the state of a single node at time $t$, that is,
\begin{equation}\label{eq:nodeMarginalDensity}
    p_{t,n,i}(s\mid\rvh) = \int_{\gS^{n-1} \times \Theta} p_{0,n}(G_t^{-1}(\mS; \theta) \mid \rvh_{\mathrm{prob}}) \left|\det(\nabla G_t^{-1}(\mS; \theta))\right| q_n(\theta \mid \rvh_{\mathrm{dyn}}) \df \mS_{-i} \df \theta, \quad i=1,\dots,n,
\end{equation}
where $\mS = (s_1, \dots, s_n)$ and $\mS_{-i} = (s_1, \dots, s_{i-1}, s_{i+1}, \dots, s_n)$. Notice that this marginal distribution $p_{t,n,i}(s\mid\rvh)$ is parameterized by the hyperparameters $\rvh_{\mathrm{prob}}$ and $\rvh_{\mathrm{dyn}}$, which means that the marginal distribution of each state at time $t$ is determined by the initial distribution and the dynamics. This explicit dependence is the foundation of our estimation framework. We will henceforth denote this hyperparameter-dependent marginal distribution of the $i$-th state at time $t$ as $\tilde{p}_{t,n,i}(s \mid \rvh_{\mathrm{prob}}, \rvh_{\mathrm{dyn}})$, with corresponding probability law $\widetilde P_{t,n,i}^{\rvh}$.

Since the distributions $\tilde{p}_{t,n,i}$ depend on the hyperparameters, the states $\rmS_n(t)$ for $t\ge 1$ contain the information of $\rvh_{\mathrm{prob}}$ and $\rvh_{\mathrm{dyn}}$. To formalize this connection, we first combine the dynamics and initial state hyperparameters into a single vector, $\rvh \triangleq (\rvh_{\mathrm{prob}}, \rvh_{\mathrm{dyn}})$, which lies in the hyperparameter space $\gH \triangleq \gH_{\mathrm{prob}}\times \gH_{\mathrm{dyn}} \subset \R^{h}$, where $h \triangleq h_{\mathrm{dyn}} + h_{\mathrm{prob}}$ is the dimension of $\gH$. Under the common-base construction and finite-horizon flow in Appendix \ref{app:flowTransport}, a hyperparameter-independent full simulator base random element $\rmX_n$ with fixed reference law and a measurable transformation $\psi_t:\gX^n\times\gH\to\gS^n$ represent the joint system state as
\begin{equation}\label{eq:Rosenblatt}
    \rmS_n(t;\rvh) = \psi_t(\rmX_n,\rvh).
\end{equation}
Equation \eqref{eq:Rosenblatt} is the joint generative representation and retains the cross-node dependence of the interacting system. For the proof-facing nodewise form, fix $n$ and an observation time $t$, and define the complete hyperparameter-indexed response of node $i$ by $\mathscr{S}_{t,n,i}(\rvh)=\rvs_i(t;\rvh)$. On an extension of the probability space, jointly distributed base variables $\rvx_i(t)\sim\mathrm{Uniform}(0,1)$ and Borel maps $\tilde\psi_{t,n,i}:[0,1]\times\gH\to\R^k$ can be chosen so that, simultaneously over the finite observation-time and node array,
\begin{equation}\label{eq:independentNodeProcessRepresentation}
\rvs_i(t;\rvh)
=\tilde\psi_{t,n,i}\{\rvx_i(t),\rvh\}
\quad\text{almost surely},
\qquad \rvh\in\gH.
\end{equation}
The same base variable $\rvx_i(t)$ represents the entire response curve over $\rvh$, and the reconstructed finite array has the same joint law as \eqref{eq:Rosenblatt}, including its cross-node and temporal dependence. The detailed conditions, construction, and proof of \eqref{eq:independentNodeProcessRepresentation} are given in Proposition \ref{prop:nodewiseBaseRepresentation} of Appendix \ref{app:flowTransport}.

Thus, the observation can be written as $\rvy_t(\rvh)=f_t\{\psi_t(\rmX_n,\rvh)\}$, where we write $\rvh$ explicitly to stress the dependence on the hyperparameters.

We concatenate the sequence of observations till time $t$ into a single observation vector:
\begin{equation*}
    \rmY(\rvh) \triangleq (\rvy_1(\rvh), \dots, \rvy_T(\rvh)) \in \R^{r_1}\times \cdots \times \R^{r_T}.
\end{equation*}
A standard approach for estimating the hyperparameter $\rvh$ is to minimize an empirical objective function, which measures the discrepancy between observations simulated with a candidate hyperparameter, $\rmY(\rvh)$, and the true observations recorded, $\rmY^\star = \left(\rvy_1^\star, \dots, \rvy_T^\star\right)$, which is assumed to be generated by $\rvy^\star_t = \hat{f_t}(\rmS^\star(t))$, where $\rmS^\star(t)$ is the system states evolved given true hyperparameter $\rvh^\star$ with noisy observation operator $\hat{f_t}:\gS^n \times \R^{n\times r_t} \to \R^{r_t}$. This objective function is typically defined using a distance metric $\gD$ in M-estimation \cite{tyler1987distribution, zhang2014novel, bellec2025error, maronna1976robust}:
\begin{equation*}
    L_n(\rvh) = \gD(\rmY(\rvh), \rmY^\star).
\end{equation*}
The expectation of the empirical objective function $L_n$ over sample distribution is denoted as its deterministic form $L^\star: \gH \to \R$. In data-rich scenario, the deterministic objective function $L^\star$ is minimized at the true hyperparameter $\rvh^\star$, that is,
\begin{equation*}
    L^\star(\rvh^\star) = \min_{\rvh \in \gH} L^\star(\rvh).
\end{equation*}

This M-estimator framework is general and encompasses many classical estimation problems. For example, in Least Squares Estimation (LSE), the empirical objective function is the average of squared residuals \cite{schmiester2021efficient, robins1994estimation, le2022linear}, while in Maximum Likelihood Estimation (MLE), it is the negative average log-likelihood \cite{neyman1948consistent, richards1961method, firth1993bias}. Such tasks can be viewed as a special case of our dynamical systems framework by considering a single-node system with \textit{resample and transform} dynamics:
\begin{equation}\label{eq:dynSysLS}
    \begin{matrix}
        \left\{\begin{aligned}
            \rvs(t) &= f(\rvx(t); \theta), \\
            \rvx(t) & \sim p(x), \\
            \theta &\sim \delta_{\rvh_{\mathrm{dyn}}}(\theta),
        \end{aligned}\right. \quad t = 1, \dots, T.
    \end{matrix}
\end{equation}
where $\delta_{\rvh_{\mathrm{dyn}}}(\cdot)$ is the Dirac-delta function at $\rvh_{\mathrm{dyn}}$. In this formulation, a new input $\rvx(t)$ is sampled and transformed by a time-invariant function $f$ at each time step, and the observation is mapped by an identity operator $\rvy_t = \rvs(t)$. The only hyperparameter to be estimated is the dynamics hyperparameter $\rvh = \rvh_{\mathrm{dyn}}$. In the single-node setting described above, an extensive body of literature, such as \cite{neyman1948consistent, barndorff1983formula, fahrmeir1985consistency, robins1994estimation, zhan2024general, saumard2017concentration}, has thoroughly studied the consistency of the estimator with respect to the observation duration $T$, which is referred to as the sample size. The Artificial Neural Networks (ANNs), such as CNN, RNN, transformer and so on, are widely used nowadays in the area of artificial intelligence \cite{rumelhart1986learning, lecun1989backpropagation, krizhevsky2017imagenet, jordan1997serial, vaswani2017attention}. Some works have framed the ANNs as a dynamical systems, where the forward propagation of information through layers is treated as a temporal evolution \cite{haber2017stable, chen2018neural}. From this point of view, the dynamics of ANNs can be describe as \eqref{eq:dynSysLS}. The central claim of our paper is that for the general multi-node complex network dynamical systems defined in \eqref{eq:dynSysGenODE}, the network population size $n$ plays a role analogous to the sample size $T$ in \eqref{eq:dynSysLS}. We prove that for a fixed observation duration, consistency is achieved as the network population size becomes sufficiently large. For the classical case, we provide a proof for the consistency of NLS estimation in Appendix \ref{app:leastSquares}.

In the following, we formulate the hyperparameter estimation problem of a general dynamical system \eqref{eq:dynSysGenODE} with $L^2$ loss as follow:

\begin{problem}[\textbf{Hyperparameter Estimation via Observation of Complex Network Dynamical System}]\label{defn:ProbGen}
    For a complex network dynamical system \eqref{eq:dynSysGenODE}, with the notations described above, an algorithm $\gA$ is performed to obtain estimation $\hat{\rvh}_n = \gA(L_n)$ for true hyperparameter $\rvh^\star$ by minimizing an empirical objective function $L_n$ merging the observation that
    \begin{equation*}
        L_n(\rvh) = \frac{1}{T}\|\rmY(\rvh) - \rmY^\star\|^2.
    \end{equation*}
    where the norm is a mixture of Euclidean norm in the product space $\R^{r_1} \times \cdots \times \R^{r_T}$ that $\|\rmY(\rvh) - \rmY^\star\|^2 = \|\rvy_1(\rvh) - \rvy_1^\star\|^2_{\R^{r_1}} + \cdots + \|\rvy_T(\rvh) - \rvy_T^\star\|^2_{\R^{r_T}}$. For simplicity, we omit the subscript $\R^{r_t}$ in the norm notation $\|\cdot\|_{\R^{r_t}}$ without ambiguity.
\end{problem}

\begin{remark}
In most applications, the observations at each time are of the same type that $r_1=\cdots=r_T$, $f_1 = \cdots = f_T$ and $\tilde{f}_1 = \cdots = \tilde{f}_T$. However, in some partially observable situations, the type of observations change at different time $t$, thereby they are described in different space $\R^{r_t}$ with different dimmension $r_t$. This setting enlarges the generality of observations' type used for the estimation problem.
\end{remark}

The observation duration $T$ and the network population size $n$ are two critical factors in the consistency theory of hyperparameter estimation problem. Consistency with respect to the observation duration $T$ can be formulated as a classical Nonlinear Least Squares Estimation problem, which is well-studied in the literature \cite{neyman1948consistent, barndorff1983formula, fahrmeir1985consistency, robins1994estimation, zhan2024general, saumard2017concentration}. In contrast, our analysis focuses on the less-explored problem of consistency with respect to the network population size $n$ for a fixed observation duration $T$, demonstrating how the convergence of the estimator $\hat{\rvh}_n$ to the true hyperparameter $\rvh^\star$ is controlled by the network population size $n$. 

Our primary contribution is to establish this consistency with respect to $n$ in the challenging and highly relevant setting of weakly-dependent nodes. This is an extension beyond existing works, which have typically been limited to either systems with independent and identically distributed (i.i.d.) nodes \cite{escande2024concentration} or classical setting where consistency is analyzed with respect to sample sizes as in \eqref{eq:dynSysLS}. Since networks in real world are defined by their interconnections, the assumption of nodes' states evolving independently is often unrealistic. Our analysis confronts this complexity by providing a theoretical guarantee for systems where dependencies exist but decay across the network \textit{distance}. The necessary preliminary definitions and lemmas are introduced in Section \ref{sec:preliminary}. The consistency of the estimate $\hat{\rvh}_n$ for Problem \ref{defn:ProbGen} is established first for independent nodes with heterogeneous finite-time marginals in Section \ref{sec:EstimationIID} and then for weakly-dependent nodes in Section \ref{sec:EstimationWD}.

\section{Main Results}\label{sec:main_results}
\subsection{Preliminary}\label{sec:preliminary}

Firstly, we make some assumptions on the dynamical system.
\begin{assum}\label{assum:dynSysLip}
    The dynamics $g_i$ in the dynamical system \eqref{eq:dynSysGenODE} is Lipschitz continuous with respect to all nodes.
\end{assum}

\begin{assum}\label{assum:dynSysCompactMeasure}
    In \eqref{eq:dynSysGenODE}, the initial distributions $p_{0}(s\mid \rvh_{\mathrm{prob}})$ and $q(\theta \mid \rvh_{\mathrm{dyn}})$ characterized by the hyperparameters satisfy the following conditions:
    \begin{enumerate}
        \item the initial state distribution $p_{0}(s\mid \rvh_{\mathrm{prob}})$ and the parameter distribution $q(\theta \mid \rvh_{\mathrm{dyn}})$ are supported on compact set, that is, $\supp\left(p_{0}\right) = \left\{s\mid p_{0}(s\mid \rvh_{\mathrm{prob}}) \neq 0\right\}$ and $\supp\left(q\right) =  \left\{\theta \mid q(\theta \mid \rvh_{\mathrm{dyn}}) \neq 0\right\}$ are compact sets in $\gS$ and $\Theta$, respectively,\\
        \item the probability density functions $p_{0}$ and $q$ are Lipschitz continuous with respect to the hyperparameters $\rvh_{\mathrm{prob}}$ and $\rvh_{\mathrm{dyn}}$, respectively.
    \end{enumerate}  
\end{assum}

Assumption \ref{assum:dynSysLip}, together with Assumptions \ref{assum:dynSysCompactMeasure} and \ref{assum:flowRegularity}, guarantees the solution's continuous dependence on the initial conditions and parameters \cite{arnold1992ordinary} and yields the full simulator representation in \eqref{eq:Rosenblatt}. The nodewise representation required by Theorems \ref{thm:ConsistencyIID} and \ref{thm:ConsistencyWD} is displayed in \eqref{eq:independentNodeProcessRepresentation}; its detailed conditions, construction, and proof are collected in Proposition \ref{prop:nodewiseBaseRepresentation} of Appendix \ref{app:flowTransport}.

In many practical applications, the observation operator $f$ takes the form of nonlinear average of the states, for example, in the Susceptible-Infectious-Susceptible (SIS) model, the observation, infection rate, is the average of the infected states; in the Spiking Neural Network (SNN) model, the observation, Local Field Potential (LFP), is the average of the neuron's membrane voltage \cite{kermack1927contribution}. Therefore, in this paper, we consider the observation operator $f_t$ of average form.
\begin{assum}[Observation operator]\label{assum:obsOperator}
    The observation operator $f_t$ is an operator, that is, for every $t\ge 1$, we have
    \begin{equation*}
        \rvy_t(\rvh) = f_t(\rmS_n(t)) = \frac{1}{n}\sum_{i=1}^n \phi_t(\rvs_i(t)) ,
    \end{equation*}
    where $\phi_t: \gS \to \R^{r_t}$.
\end{assum}
Using the nodewise representation \eqref{eq:independentNodeProcessRepresentation}, define the node-specific feature and rewrite the observation as
\begin{equation}\label{eq:nodeSpecificObservationFeature}
\begin{aligned}
\varphi_{t,i}(x,\rvh)
&=\phi_t\{\tilde\psi_{t,n,i}(x,\rvh)\},\\
\rvy_t(\rvh)
&=\frac1n\sum_{i=1}^n
\varphi_{t,i}\{\rvx_i(t),\rvh\}.
\end{aligned}
\end{equation}

Then, we make an assumption on the optimization algorithm $\gA$.

\begin{assum}[Optimization algorithm]\label{assum:minimizer}
    The optimization algorithm can find the minimizer of $L_n(\rvh)$ for every realization of the simulator and observation variables and every $n \in \sN$. That is, the algorithm $\gA$ takes the objective function $L_n(\rvh)$ as input and outputs an estimator $\hat{\rvh}_n$ satisfying
    \begin{equation*}
        L_n(\hat{\rvh}_n) = \min_{\rvh \in \gH} L_n(\rvh), \text{ where } \hat{\rvh}_n = \gA(L_n(\rvh)).
    \end{equation*}
\end{assum}

\begin{remark}
    The assumption for the algorithm can be relaxed to that the algorithm can find an approximate minimizer $\hat{\rvh}_n$ of $L_n$, that is,
    \begin{equation}\label{assum:approx_minimizer}
        L_n(\hat{\rvh}_n) \le L_n(\rvh^\star).
    \end{equation}
    In the following analyses, we take the exact minimizer condition Assumption \ref{assum:minimizer} for simplicity, while all these proofs can be conducted similarly under the approximate minimizer condition \eqref{assum:approx_minimizer}. 
\end{remark}

\begin{remark}
    The existence of the minimizer can be guaranteed by several mild conditions. For example, if $\rvh \mapsto L_n(\rvh)$ is lower semi-continuous for all $\rmS_n \in \gS^n$ and $\gH$ is compact, then a minimizer $\hat{\rvh}_n$ exists for all $\rmS_n \in \gS^n$ and at least one measurable minimizer exists under some additional conditions \cite{hess1996epi, pfanzagl1994parametric, vaart1997weak}.
\end{remark}

Only the ability of algorithm for finding the minimizer of objective function is not enough to guarantee the consistency of the estimator, because the objective function may have multiple minimizers. The deterministic objective function related to the estimation problem should be identifiable at the true hyperparameter $\rvh^\star$. We give the definition of identifiability and strictly identifiability as follows.

\begin{defn}\label{identifiable}
    An estimation problem or its corresponding deterministic objective function $L^\star: \gH \to \R$ is \textbf{identifiable}, if $L^\star$ is uniquely minimized at $\rvh = \rvh^\star$.

    Suppose $d:\gH \times \gH \to \R^+ \cup \{0\}$ is a metric on $\gH$. We say that an estimation problem or its corresponding deterministic objective function $L^\star: \gH \to \R$ is \textbf{strictly identifiable}, if $L^\star$ is identifiable at the unique minimizer $\rvh^\star$ and 
        \begin{equation*}
            \inf_{d(\rvh, \rvh^\star) \ge \eps} L^\star(\rvh) > L^\star(\rvh^\star), \qquad \forall \eps > 0.
        \end{equation*}  
\end{defn}

The identifiability condition is crucial and general for the consistency of the estimator \cite{white2012consistency}. Without this condition, it is not guaranteed for an estimation algorithm to find the true hyperparameter $\rvh^\star$ even if the deterministic objective function $L^\star$ is provided and the algorithm is effective as defined in Assumption \ref{assum:minimizer}. 

For the sequence of finite-$n$ network problems below, the notation $L^\star$ retains its original meaning $L^\star(\rvh)=\E L_n(\rvh)$, with the dependence on $n$ understood. We impose the identifiability condition uniformly in the network population size.

\begin{assum}[Identifiability]\label{assum:identifiability}
There exists $n_0\in\sN$ such that, for every $n\ge n_0$, $L^\star$ is uniquely minimized at $\rvh^\star$ and, for every $\eps>0$ for which $\{\rvh\in\gH:d(\rvh,\rvh^\star)\ge\eps\}$ is nonempty,
\begin{equation}\label{eq:uniformIdentifiability}
\eta(\eps)
\triangleq \frac12\inf_{n\ge n_0}
\left\{\inf_{d(\rvh,\rvh^\star)\ge\eps}L^\star(\rvh)-L^\star(\rvh^\star)\right\}>0.
\end{equation}
\end{assum}

The following condition provides a direct way to verify Assumption \ref{assum:identifiability} and also yields the nonasymptotic hyperparameter bounds below.

\begin{lemma}[Local growth and strict outer gap]\label{le:localGrowthIdentifiability}
Suppose that, uniformly for $n\ge n_0$, there exist constants $c>0$, $r_0>0$, $\Gamma>0$, and $\alpha\ge1$ such that
\begin{align}
L^\star(\rvh)-L^\star(\rvh^\star)
&\ge c\,d(\rvh,\rvh^\star)^\alpha,
&&0<d(\rvh,\rvh^\star)<r_0,\label{eq:localGrowthCondition}\\
L^\star(\rvh)-L^\star(\rvh^\star)
&\ge \Gamma,
&&d(\rvh,\rvh^\star)\ge r_0.\label{eq:strictOuterGap}
\end{align}
Then Assumption \ref{assum:identifiability} holds and
\begin{equation}\label{eq:localGrowthEta}
\eta(\eps)\ge
\frac12\min\left\{c\min(\eps,r_0)^\alpha,\Gamma\right\}.
\end{equation}
The proof is given in Appendix \ref{app:identifiabilityGrowth}.
\end{lemma}

The following lemma shows that an identifiable and continuous objective function within compact hyperparameter space is strictly identifiable. 

\begin{lemma}\label{le:identifiable}
    If the objective function $L^\star: \gH \to \sR$ is identifiable and continuous with respect to $\rvh$, and $\gH$ is compact, then $L^\star$ is strictly identifiable at the unique minimizer $\rvh^\star$.
\end{lemma}
The proof is given in Appendix \ref{app:preliminaryTools}.

To obtain the consistency of the estimator, we leverage the fact from M-estimator \cite{tyler1987distribution, zhang2014novel, bellec2025error, maronna1976robust} that the discrepancy between the true parameter $\rvh^\star$ and the estimator $\hat{\rvh}_n$ attained by an algorithm can be bounded by the difference between the empirical objective function $L_n$ and the deterministic objective function $L^\star$. To illustrate this statement, we have the following lemma.
\begin{lemma}\label{le:SIMarkov}
    Suppose the deterministic objective function $L^\star: \gH \to \R$ is strictly identifiable at the unique minimizer $\rvh^\star$, and $\hat{\rvh}_n$ is the estimator obtained by an algorithm satisfying the Assumption \ref{assum:minimizer}, then for any $\eps > 0$, we have 
    \begin{equation}\label{eq:SIMarkov}
        \gP\left(d(\hat{\rvh}_n, \rvh^\star) \ge \eps \right) \le \frac{2}{\eta(\eps)} \E \left[\sup_{\rvh \in \gH} \|L_n(\rvh) - L^\star(\rvh)\|\right],
    \end{equation}
    where $\eta(\eps)$ is a positive constant determined by $\eps$ and $L^\star$.
\end{lemma}
The proof is given in Appendix \ref{app:preliminaryTools}.

\setcounter{lemma}{5}

This stability reduction is standard in M-estimation
\cite{rogers1978finite,kearns1997algorithmic,bousquet2002stability}. The
sub-Gaussian, separability, covering-number, and entropy tools used below are
collected in Appendix \ref{app:preliminaryTools}.

\subsection{Estimation via Observations of Independent Complex Network Dynamical System}\label{sec:EstimationIID}

Equipped with these concepts and lemmas, we are now prepared to establish our main consistency results. We begin our analysis with node responses that are independent across nodes at each observation time and whose finite-time marginal distributions may differ across nodes. While many real-world networks exhibit dependencies among nodes, this case provides a clear and essential baseline for our framework. It allows us to demonstrate the core proof strategy and establish a consistency guarantee that will serve as a stepping stone to the more general, weakly-dependent case.

\begin{problem}[\textbf{Hyperparameter Estimation via Observation of Independent Complex Network Dynamical System}]\label{defn:ProbIID}
For every observation time $t$, assume that the node response processes
$\mathscr{S}_{t,n,i}=\{\rvs_i(t;\rvh):\rvh\in\gH\}$ represented in
\eqref{eq:independentNodeProcessRepresentation} are independent across $i$. Let
$\{\rvx_i(t):1\le i\le n\}$ and
$\{\rvx_i^\star(t):1\le i\le n\}$ be independent candidate and recorded
base arrays whose coordinates are $\mathrm{Uniform}(0,1)$ and independent
across $i$ for every fixed $t$. Their joint law is chosen so that the
reconstructed candidate and recorded arrays preserve their respective
temporal laws, and the two full arrays are independent of one another. For a simulated hyperparameter $\rvh$ and the true hyperparameter
$\rvh^\star$, define
\begin{equation}\label{eq:iidStateRealizations}
\rvs_i(t;\rvh)=\tilde\psi_{t,n,i}(\rvx_i(t),\rvh),
\qquad
\rvs_i^\star(t)=\tilde\psi_{t,n,i}(\rvx_i^\star(t),\rvh^\star).
\end{equation}
For fixed $t$, the same $\rvx_i(t)$ is used for every candidate hyperparameter,
so \eqref{eq:iidStateRealizations} represents the entire
hyperparameter-indexed response at that time rather than separate pointwise
marginals. The maps $\tilde\psi_{t,n,i}$ may depend on $i$ and $t$, and
arbitrary dependence across observation times is retained.
The candidate observations $\rmY_n(\rvh)=(\rvy_1(\rvh),\dots,\rvy_T(\rvh))$ are modeled under Assumption \ref{assum:obsOperator} as
\begin{equation*}
    \rvy_t(\rvh) = \frac{1}{n}\sum_{i=1}^n \phi_t(\rvs_i(t;\rvh)) \in \R^{r_t},
\end{equation*}
and the recorded observations $\rmY^\star=(\rvy_1^\star,\dots,\rvy_T^\star)$ are
\begin{equation*}
    \rvy_t^\star = \frac{1}{n}\sum_{i=1}^n \left\{\phi_t(\rvs_i^\star(t)) + \omega_{(t,i)}\right\}, \qquad t=1,\dots,T,
\end{equation*}
where the $\omega_{(t,i)}$ are centered Gaussian random vectors with covariance $\sigma_t^2\mI_{r_t}$, independent across $i$ and independent of both state arrays. Thus, for every fixed $t$, the blocks
\begin{equation}\label{eq:iidNodeBaseBlock}
Z_{t,i}=\left(\rvx_i(t),\rvx_i^\star(t),\omega_{(t,i)}\right),
\qquad i=1,\ldots,n,
\end{equation}
are independent across $i$, while dependence across the fixed observation times may be retained.

Thus, the empirical objective function is defined as
\begin{equation*}
    L_n(\rvh) = \frac{1}{T} \sum\limits_{t=1}^T \left\|\frac{1}{n} \sum\limits_{i=1}^n \left\{\phi_t(\rvs_i(t;\rvh)) - \phi_t(\rvs_i^\star(t)) - \omega_{(t,i)} \right\}\right\|^2.
\end{equation*}
The corresponding deterministic objective function is
\begin{equation*}
    L^\star(\rvh) = \E[L_n(\rvh)].
\end{equation*}
\end{problem}

\begin{lemma}[Cyclic node-disjoint decomposition]\label{le:cyclicDecomposition}
The unordered off-diagonal pairs of $\{1,\ldots,n\}$ can be partitioned
into at most $n+1$ matchings, each consisting of node-disjoint pairs. For odd
$n$, exactly $n$ matchings suffice. For even $n$, the partition follows by
adjoining one dummy vertex and deleting its incident edges.

If, in addition, $Z_1,\ldots,Z_n$ are independent random elements and, for
$i,k\le n$, $H_{i,k}(\rvh,\rvg)$ is a centered kernel measurable with respect
to $\sigma(Z_i,Z_k)$ satisfying, uniformly in $i,k$,
\begin{equation}\label{eq:cyclicKernelProxy}
\E\exp\{uH_{i,k}(\rvh,\rvg)\}
\le \exp\{u^2\nu^2d(\rvh,\rvg)^2/2\},\qquad u\in\R.
\end{equation}
then, for a universal constant $C_S$,
\begin{equation}\label{eq:cyclicFullProxy}
\E\exp\left\{u\sum_{i=1}^n\sum_{k=1}^nH_{i,k}(\rvh,\rvg)\right\}
\le \exp\left\{\frac{u^2C_Sn^3\nu^2}{2}d(\rvh,\rvg)^2\right\}.
\end{equation}
The proof and schematic are given in Appendix \ref{app:cyclicDecomposition}.
\end{lemma}

\begin{thm}[Probabilistic hyperparameter bound under independent condition]\label{thm:ConsistencyIID}
For the hyperparameter estimation Problem \ref{defn:ProbIID}, suppose that
the node response processes are independent across nodes at every fixed
observation time, without requiring independence across time or identical
finite-time marginal laws. Suppose further that the complex network
dynamical system satisfies Assumptions \ref{assum:dynSysLip},
\ref{assum:dynSysCompactMeasure}, and \ref{assum:flowRegularity}, the
algorithm satisfies Assumption \ref{assum:minimizer}, and the deterministic
objective satisfies Assumption \ref{assum:identifiability}. In addition,
suppose that the following conditions hold:
\begin{enumerate}
\item for every $t=1,\dots,T$, the observation map $\phi_t$ is uniformly
bounded on the finite-horizon reachable state set, with bounds independent
of $n$ and the node index;
\item for every $t=1,\dots,T$, the observation map $\phi_t$ is uniformly
Lipschitz continuous in its state argument on the finite-horizon reachable
state set, with Lipschitz constants independent of $n$ and the node index;
\item the hyperparameter space $\gH\subset\R^h$ is compact.
\end{enumerate}
Let $d$ be the Euclidean metric or a fixed norm-equivalent metric on $\gH$.
Then there is a constant $C_{\mathrm{ind}}$, independent of $n$, such that, for every $\eps>0$ and $n\ge n_0$,
\begin{equation}\label{eq:iidFixedAccuracy}
\gP\left(d(\hat{\rvh}_n,\rvh^\star)\ge\eps\right)
\le \frac{C_{\mathrm{ind}}}{\eta(\eps)\sqrt n}.
\end{equation}
Here $\eta(\eps)$ is the uniform identifiability margin in
\eqref{eq:uniformIdentifiability}.
\end{thm}

Theorem \ref{thm:ConsistencyIID} provides a formal guarantee that the probability of the estimator $\hat{\rvh}_n$ deviating from the true hyperparameter $\rvh^\star$ by more than a fixed accuracy $\eps$ diminishes as the network population size $n$ increases. The $n^{-1/2}$ factor is the independent-node objective-deviation scale.

\begin{proof}
We retain the node index in the feature and residual means and then follow the
Rademacher-symmetrization, sub-Gaussian-increment, and chaining argument.
This formulation requires uniform bounds across nodes, but not equality of
their finite-time marginal laws.
Use the nodewise realization in \eqref{eq:iidStateRealizations} and write
\begin{equation*}
\varphi_{t,i}(x,\rvh)
=\phi_t\{\tilde\psi_{t,n,i}(x,\rvh)\}.
\end{equation*}
The finite-horizon flow regularity in Assumptions
\ref{assum:dynSysLip}, \ref{assum:dynSysCompactMeasure}, and
\ref{assum:flowRegularity} supplies a nodewise hyperparameter-Lipschitz
constant $C_T$, uniform in $n$, $i$, and the observation time. The first two
conditions of the theorem therefore give constants
$B_t,L_{\phi,t}<\infty$ such that, almost surely,
\begin{equation}\label{eq:iidFeatureBounds}
\|\varphi_{t,i}(x,\rvh)\|\le B_t,
\qquad
\|\varphi_{t,i}(x,\rvh)-\varphi_{t,i}(x,\rvg)\|
\le M_t d(\rvh,\rvg).
\end{equation}
Here $M_t=L_{\phi,t}C_T$, the bounds hold uniformly in $n$ and $i$, and
$x$ denotes a realization of one node base random variable. Thus, the feature
regularity follows from the flow, nodewise representation, and observation
regularity.
The residual, its mean, and its centered part are
\begin{align}
l_{t,i}(\rvh)
&=\varphi_{t,i}(\rvx_i(t),\rvh)
-\varphi_{t,i}(\rvx_i^\star(t),\rvh^\star)-\omega_{(t,i)},\label{eq:iidResidual}\\
a_{t,i}(\rvh)&=\E l_{t,i}(\rvh),
\qquad
\zeta_{t,i}(\rvh)=l_{t,i}(\rvh)-a_{t,i}(\rvh).\notag
\end{align}
The candidate and recorded arrays in \eqref{eq:iidStateRealizations} are independent copies. For every $i$, their node-specific distributions agree at $\rvh^\star$, and hence $a_{t,i}(\rvh^\star)=0$, although the two feature realizations do not cancel sample by sample. For every fixed $t$, the process
$\{\zeta_{t,i}(\rvh):\rvh\in\gH\}$ is measurable with respect to
$\sigma(Z_{t,i})$, where $Z_{t,i}$ is defined in
\eqref{eq:iidNodeBaseBlock}; these sigma-fields are independent across $i$.
Define
\begin{equation}\label{eq:iidResidualAverages}
\bar a_{t,n}(\rvh)=\frac1n\sum_{i=1}^n a_{t,i}(\rvh),
\qquad
\overline\zeta_{t,n}(\rvh)=\frac1n\sum_{i=1}^n\zeta_{t,i}(\rvh).
\end{equation}
Direct expansion gives
\begin{equation}\label{eq:iidExactDecomposition}
\begin{aligned}
L_n(\rvh)-L^\star(\rvh)
&=\frac2T\sum_{t=1}^T
 \langle \bar a_{t,n}(\rvh),\overline\zeta_{t,n}(\rvh)\rangle\\
&\quad+\frac1T\sum_{t=1}^T
 \left\{\|\overline\zeta_{t,n}(\rvh)\|^2
 -\E\|\overline\zeta_{t,n}(\rvh)\|^2\right\}.
\end{aligned}
\end{equation}
Define the linear and centered quadratic fluctuation terms in
\eqref{eq:iidExactDecomposition}, respectively, by
\begin{align}
\mathcal L_n(\rvh)
&\triangleq \frac2T\sum_{t=1}^T
\langle \bar a_{t,n}(\rvh),\overline\zeta_{t,n}(\rvh)\rangle,
\label{eq:iidLinearFluctuation}\\
\mathcal Q_n(\rvh)
&\triangleq \frac1T\sum_{t=1}^T
\left\{\|\overline\zeta_{t,n}(\rvh)\|^2
-\E\|\overline\zeta_{t,n}(\rvh)\|^2\right\}.
\label{eq:iidQuadraticFluctuation}
\end{align}
Also, let
$\Delta_n\triangleq\sup_{\rvh\in\gH}|L_n(\rvh)-L^\star(\rvh)|$.

We first treat the linear process by Rademacher symmetrization at each
observation time. Put
\begin{equation}\label{eq:iidLinearKernelDefinition}
F_{t,i}(\rvh)=2\langle \bar a_{t,n}(\rvh),
\zeta_{t,i}(\rvh)\rangle,
\qquad
\mathcal L_{t,n}(\rvh)=\frac1n\sum_{i=1}^nF_{t,i}(\rvh),
\qquad
\mathcal L_n(\rvh)=\frac1T\sum_{t=1}^T\mathcal L_{t,n}(\rvh).
\end{equation}
For every fixed $t$, the processes $F_{t,i}$ are
$\sigma(Z_{t,i})$-measurable, independent across $i$, and centered, with
$F_{t,i}(\rvh^\star)=0$. If $F_{t,i}'$ is an independent copy and
$\epsilon_{t,i}$ are independent Rademacher variables, Lemma
\ref{le:Symmetry} gives
\begin{align}
\E\sup_{\rvh\in\gH}|\mathcal L_{t,n}(\rvh)|
&\le\E\sup_{\rvh\in\gH}
 \left|\frac1n\sum_i\{F_{t,i}(\rvh)-F_{t,i}'(\rvh)\}\right|\notag\\
&=\E\sup_{\rvh\in\gH}
 \left|\frac1n\sum_i\epsilon_{t,i}
 \{F_{t,i}(\rvh)-F_{t,i}'(\rvh)\}\right|\notag\\
&\le2\E\sup_{\rvh\in\gH}
 \left|\frac1n\sum_i\epsilon_{t,i}F_{t,i}(\rvh)\right|.
\label{eq:iidLinearSymmetrization}
\end{align}
The triangle inequality then combines these $T$ bounds without imposing
independence across observation times.
To compute the increment proxy, define
\begin{equation*}
\begin{aligned}
b_{t,i}(\rvh)
&=\varphi_{t,i}(\rvx_i(t),\rvh)
-\E\varphi_{t,i}(\rvx_i(t),\rvh),\\
c_{t,i}^\star
&=\varphi_{t,i}(\rvx_i^\star(t),\rvh^\star)
-\E\varphi_{t,i}(\rvx_i^\star(t),\rvh^\star),\\
\delta_{t,i}(\rvh,\rvg)&=b_{t,i}(\rvh)-b_{t,i}(\rvg).
\end{aligned}
\end{equation*}
Then $\zeta_{t,i}(\rvh)=b_{t,i}(\rvh)-c_{t,i}^\star-\omega_{(t,i)}$,
and \eqref{eq:iidFeatureBounds} yields
\begin{equation}\label{eq:iidCenteredBounds}
\|b_{t,i}(\rvh)\|,\|c_{t,i}^\star\|\le2B_t,
\quad
\|\delta_{t,i}(\rvh,\rvg)\|\le2M_td(\rvh,\rvg),
\end{equation}
together with
$\|\bar a_{t,n}(\rvh)-\bar a_{t,n}(\rvg)\|\le M_td(\rvh,\rvg)$ and
$\|\bar a_{t,n}(\rvh)\|\le2B_t$. Expanding an increment gives
\begin{equation}\label{eq:iidLinearIncrement}
\begin{aligned}
F_{t,i}(\rvh)-F_{t,i}(\rvg)
=2\{&\langle \bar a_{t,n}(\rvh)-\bar a_{t,n}(\rvg),
                    \zeta_{t,i}(\rvh)\rangle\\
&+\langle \bar a_{t,n}(\rvg),\delta_{t,i}(\rvh,\rvg)\rangle\}.
\end{aligned}
\end{equation}
The part of \eqref{eq:iidLinearIncrement} not containing
$\omega_{(t,i)}$ is centered
and bounded in absolute value by
$16B_tM_t d(\rvh,\rvg)$. Hoeffding's lemma therefore makes it
sub-Gaussian with that variance-proxy scale. Conditional on the base
variables, the remaining part is a linear form in the Gaussian noises and
has variance-proxy scale at most
$2\sigma_tM_t d(\rvh,\rvg)$. Proposition
\ref{prop:subGaussianSum} thus gives
\begin{equation}\label{eq:iidLinearProxy}
\E\exp\left[u\{F_{t,i}(\rvh)-F_{t,i}(\rvg)\}\right]
\le\exp\left\{\frac{u^2v_{\mathcal L,t}^2
d(\rvh,\rvg)^2}{2}\right\},
\end{equation}
where one may take
\begin{equation*}
v_{\mathcal L,t}=2(8B_t+\sigma_t)M_t,
\qquad
v_{\mathcal L}=\frac1T\sum_{t=1}^Tv_{\mathcal L,t}.
\end{equation*}
The same bound holds after multiplication by an independent Rademacher sign.
For each fixed $t$, independence across $i$ and normalization by $n$ show
that both $\mathcal L_{t,n}$ and the symmetrized process in
\eqref{eq:iidLinearSymmetrization} have increment proxy
$v_{\mathcal L,t}^2d(\rvh,\rvg)^2/n$. Only after this
node summation is complete do we average over $t$. The general, rather than
independent, part of Proposition \ref{prop:subGaussianSum} then gives the
combined bound
\begin{equation}\label{eq:iidLinearCombinedProxy}
\E\exp\left[u\{\mathcal L_n(\rvh)-\mathcal L_n(\rvg)\}\right]
\le\exp\left\{\frac{u^2v_{\mathcal L}^2
d(\rvh,\rvg)^2}{2n}\right\}.
\end{equation}
No temporal independence is used.

We next perform the quadratic calculation. Define
\begin{align}
K_{t,i,k}(\rvh)
&=\langle\zeta_{t,i}(\rvh),\zeta_{t,k}(\rvh)\rangle
-\E\langle\zeta_{t,i}(\rvh),\zeta_{t,k}(\rvh)\rangle,
\label{eq:iidCenteredQuadraticKernel}\\
H_{t,i,k}(\rvh,\rvg)&=K_{t,i,k}(\rvh)-K_{t,i,k}(\rvg).
\label{eq:iidQuadraticKernelIncrement}
\end{align}
By \eqref{eq:iidNodeBaseBlock}, $K_{t,i,k}$ and $H_{t,i,k}$ are measurable
with respect to $\sigma(Z_{t,i},Z_{t,k})$. This is the endpoint measurability required
by Lemma \ref{le:cyclicDecomposition}.
Then
\begin{equation}\label{eq:iidQuadraticKernel}
\mathcal Q_n(\rvh)=\frac1{Tn^2}
\sum_{t=1}^T\sum_{i=1}^n\sum_{k=1}^nK_{t,i,k}(\rvh).
\end{equation}
For $i\ne k$, before centering the kernel increment equals
\begin{equation}\label{eq:iidPairIncrement}
\begin{aligned}
&\langle\zeta_{t,i}(\rvh),\zeta_{t,k}(\rvh)\rangle
-\langle\zeta_{t,i}(\rvg),\zeta_{t,k}(\rvg)\rangle\\
&\quad=\langle\delta_{t,i}(\rvh,\rvg),\zeta_{t,k}(\rvh)\rangle
+\langle\zeta_{t,i}(\rvg),\delta_{t,k}(\rvh,\rvg)\rangle.
\end{aligned}
\end{equation}
Conditional on the base variables, the noise-dependent part of
\eqref{eq:iidPairIncrement} is
\begin{equation*}
-\langle\delta_{t,i},\omega_{(t,k)}\rangle
-\langle\omega_{(t,i)},\delta_{t,k}\rangle.
\end{equation*}
The two terms are independent Gaussian linear forms. By
\eqref{eq:iidCenteredBounds}, their joint variance proxy is at most
$8\sigma_t^2M_t^2d(\rvh,\rvg)^2$. The remaining centered part is bounded
by $32B_tM_td(\rvh,\rvg)$, so Hoeffding's lemma and Proposition
\ref{prop:subGaussianSum} imply
\begin{equation}\label{eq:iidPairProxy}
\E e^{uH_{t,i,k}(\rvh,\rvg)}
\le e^{u^2\nu_t^2d(\rvh,\rvg)^2/2},
\qquad
\nu_t=C_Q(B_t+\sigma_t)M_t,
\end{equation}
for a numerical constant $C_Q$. The same estimate holds when $i=k$ because
\begin{equation*}
\|\zeta_{t,i}(\rvh)\|^2-\|\zeta_{t,i}(\rvg)\|^2
=\langle\delta_{t,i}(\rvh,\rvg),
\zeta_{t,i}(\rvh)+\zeta_{t,i}(\rvg)\rangle,
\end{equation*}
whose noise contribution is again a Gaussian linear form and whose remaining
part is bounded using \eqref{eq:iidCenteredBounds}.

For fixed $t$, apply the probabilistic conclusion of Lemma
\ref{le:cyclicDecomposition} with the independent node blocks $Z_{t,i}$ and
the centered kernels $H_{t,i,k}$. The node-disjoint matchings give
\begin{equation}\label{eq:iidFixedTimeQuadraticProxy}
\E\exp\left[u\sum_{i=1}^n\sum_{k=1}^n
H_{t,i,k}(\rvh,\rvg)\right]
\le\exp\left\{\frac{u^2C_Sn^3\nu_t^2
d(\rvh,\rvg)^2}{2}\right\}.
\end{equation}
Thus the fixed-time node sum is completed before time aggregation. Define
$\mathcal Q_{t,n}(\rvh)=n^{-2}\sum_{i,k}K_{t,i,k}(\rvh)$, so that
$\mathcal Q_n=T^{-1}\sum_t\mathcal Q_{t,n}$. After the $n^{-2}$
normalization, \eqref{eq:iidFixedTimeQuadraticProxy} gives the fixed-time
increment proxy $C_S\nu_t^2d(\rvh,\rvg)^2/n$. The general sum rule then
combines the final average over $t$, giving the explicit bound
\begin{equation}\label{eq:iidQuadraticProxy}
\E\exp\left[u\{\mathcal Q_n(\rvh)-\mathcal Q_n(\rvg)\}\right]
\le\exp\left\{\frac{u^2v_{\mathcal Q}^2
d(\rvh,\rvg)^2}{2n}\right\},
\qquad
v_{\mathcal Q}=\frac{\sqrt{C_S}}{T}\sum_{t=1}^T\nu_t.
\end{equation}
No independence across $t$ is invoked. Moreover, the fourth moments of the
Gaussian noises and the bounded features give
$\{\E|\mathcal Q_{t,n}(\rvh^\star)|^2\}^{1/2}\le C_{\mathrm{anc},t}/n$
at each fixed $t$. Minkowski's inequality then combines the time average
without temporal independence:
\begin{equation*}
\E|\mathcal Q_n(\rvh^\star)|
\le\{\E\mathcal Q_n(\rvh^\star)^2\}^{1/2}
\le\frac1T\sum_{t=1}^T\frac{C_{\mathrm{anc},t}}{n}
=\frac{C_{\mathrm{anc}}}{n}.
\end{equation*}

The processes controlled by \eqref{eq:iidLinearCombinedProxy} and
\eqref{eq:iidQuadraticProxy} have continuous sample paths on compact
$\gH$, hence are separable. Dudley's theorem and Lemma
\ref{le:BoundedEntropy}, applied with the rescaled metrics
$v_{\mathcal L}d/\sqrt n$ and $v_{\mathcal Q}d/\sqrt n$, yield
\begin{equation}\label{eq:iidObjectiveDeviation}
\E\Delta_n
\le\frac{C_{\mathrm{ch}}\sqrt h\,\diam(\gH)
(v_{\mathcal L}+v_{\mathcal Q})}{\sqrt n}
+\frac{C_{\mathrm{anc}}}{n}
\le\frac{C_0}{\sqrt n}.
\end{equation}
Here and below $C_{\mathrm{ch}}$ is numerical, and $C_0$ is independent of
$n$. Finally, Lemma \ref{le:SIMarkov}, Assumption
\ref{assum:identifiability}, and \eqref{eq:iidObjectiveDeviation} prove
\eqref{eq:iidFixedAccuracy}, with $C_{\mathrm{ind}}=2C_0$.
\end{proof}

Theorem \ref{thm:ConsistencyIID} is a fixed-accuracy result: for each
prescribed $\eps>0$, it bounds the probability that the estimator leaves the
$\eps$-ball, and at the level of objective geometry it requires only the
uniform identifiability margin $\eta(\eps)$. This broad formulation leaves
the dependence of the parameter error on a chosen confidence level implicit.
Corollary \ref{coro:iidRadius} replaces Assumption
\ref{assum:identifiability} by the stronger quantitative conditions
\eqref{eq:localGrowthCondition}--\eqref{eq:strictOuterGap}: the outer gap
first localizes the estimator, and the local growth then converts the
$O(n^{-1/2})$ objective-deviation scale into an explicit $1-\rho$ parameter
radius of order $(\rho\sqrt n)^{-1/\alpha}$. Thus the corollary gives a more
explicit finite-$n$ conclusion for a narrower class of deterministic
objectives.

\begin{coro}[Nonasymptotic hyperparameter bound under independent condition]\label{coro:iidRadius}
Suppose that all the conditions of Theorem \ref{thm:ConsistencyIID} other
than Assumption \ref{assum:identifiability} hold, and suppose instead that
the uniform local-growth and strict outer-gap conditions
\eqref{eq:localGrowthCondition}--\eqref{eq:strictOuterGap} hold. Let $C_0$
be the constant in \eqref{eq:iidObjectiveDeviation}. Then, for every
$n\ge n_0$ and $0<\rho<1$ satisfying
\begin{equation}\label{eq:iidOuterGapThreshold}
\frac{2C_0}{\rho\sqrt n}<\Gamma,
\end{equation}
the estimator satisfies, with probability at least $1-\rho$,
\begin{equation}\label{eq:iidRadius}
d(\hat{\rvh}_n,\rvh^\star)
\le \left(\frac{2C_0}{c\rho\sqrt n}\right)^{1/\alpha}.
\end{equation}
\end{coro}

\begin{proof}[Proof of Corollary \ref{coro:iidRadius}]
Lemma \ref{le:localGrowthIdentifiability} shows that
\eqref{eq:localGrowthCondition}--\eqref{eq:strictOuterGap} imply Assumption
\ref{assum:identifiability}. Hence all the conditions of Theorem
\ref{thm:ConsistencyIID} are satisfied. Define
\begin{equation*}
\Delta_n=\sup_{\rvh\in\gH}|L_n(\rvh)-L^\star(\rvh)|
\end{equation*}
and the event
\begin{equation*}
\mathcal E_{n,\rho}^{\mathrm{ind}}
=\left\{\Delta_n\le\frac{C_0}{\rho\sqrt n}\right\}.
\end{equation*}
By Markov's inequality and \eqref{eq:iidObjectiveDeviation},
\begin{equation*}
\begin{aligned}
\gP\bigl((\mathcal E_{n,\rho}^{\mathrm{ind}})^c\bigr)
&=\gP\left(\Delta_n>\frac{C_0}{\rho\sqrt n}\right)\\
&\le\frac{\rho\sqrt n}{C_0}\E\Delta_n
\le\rho.
\end{aligned}
\end{equation*}
Therefore, $\gP(\mathcal E_{n,\rho}^{\mathrm{ind}})\ge1-\rho$.

On $\mathcal E_{n,\rho}^{\mathrm{ind}}$, Assumption
\ref{assum:minimizer} gives
\begin{equation}\label{eq:iidCorollaryExcess}
\begin{aligned}
L^\star(\hat{\rvh}_n)-L^\star(\rvh^\star)
&=L^\star(\hat{\rvh}_n)-L_n(\hat{\rvh}_n)
  +L_n(\hat{\rvh}_n)-L_n(\rvh^\star)\\
&\quad+L_n(\rvh^\star)-L^\star(\rvh^\star)\\
&\le
\bigl|L^\star(\hat{\rvh}_n)-L_n(\hat{\rvh}_n)\bigr|
+\bigl|L_n(\rvh^\star)-L^\star(\rvh^\star)\bigr|\\
&\le2\Delta_n
\le\frac{2C_0}{\rho\sqrt n}.
\end{aligned}
\end{equation}
If $d(\hat{\rvh}_n,\rvh^\star)\ge r_0$, then
\eqref{eq:strictOuterGap} gives
\begin{equation*}
L^\star(\hat{\rvh}_n)-L^\star(\rvh^\star)\ge\Gamma,
\end{equation*}
which contradicts \eqref{eq:iidOuterGapThreshold} and
\eqref{eq:iidCorollaryExcess}. Hence
$d(\hat{\rvh}_n,\rvh^\star)<r_0$. If
$\hat{\rvh}_n=\rvh^\star$, \eqref{eq:iidRadius} is immediate. Otherwise,
\eqref{eq:localGrowthCondition} and \eqref{eq:iidCorollaryExcess} yield
\begin{equation*}
c\,d(\hat{\rvh}_n,\rvh^\star)^\alpha
\le L^\star(\hat{\rvh}_n)-L^\star(\rvh^\star)
\le\frac{2C_0}{\rho\sqrt n}.
\end{equation*}
Taking the $1/\alpha$ power proves \eqref{eq:iidRadius} on an event of
probability at least $1-\rho$.
\end{proof}

The two geometric conditions \eqref{eq:localGrowthCondition} and
\eqref{eq:strictOuterGap} in the corollary hold under a checkable
mean-separation and controlled-variance regime. In particular, the mean
observation map should be locally nondegenerate at $\rvh^\star$ and remain
uniformly separated from its value at $\rvh^\star$ outside a neighborhood of
that point, while the hyperparameter-dependent simulation variance is
stationary at $\rvh^\star$ and has insufficient negative curvature or outer
drift to cancel the mean separation. Uniform $C^2$ regularity then gives the
local condition with $\alpha=2$. For independent nodes, uniformly bounded
first and second hyperparameter derivatives make the relevant variance and
curvature corrections $O(n^{-1})$, so the required domination holds for all
sufficiently large $n$. The exact-target stationarity is essential; without
it, the objective must instead be recentered at its finite-$n$ population
minimizer. The complete sufficient-condition derivation is given in Appendix
\ref{app:iidRadiusConditions}.

\subsection{Estimation via Observations of Weakly-Dependent Complex Network Dynamical System}\label{sec:EstimationWD}

Section \ref{sec:EstimationIID} established a nonasymptotic scaling guarantee for hyperparameter estimation in complex network dynamical systems whose node responses are independent across nodes at each observation time and have heterogeneous finite-time marginals. However, the core premise of complex networks is interconnectivity, making the assumption of independence a significant limitation for real-world applications. The proof strategy in Theorem \ref{thm:ConsistencyIID}, which leverages the properties of independent sums of random variables, is no longer applicable in the presence of dependencies among nodes.

We now turn to the central contribution of this work that extends the consistency results to the more realistic and challenging setting of weakly-dependent complex network dynamical systems. To overcome the technical challenges posed by dependencies, we employ the concept of $\beta$-mixing to formally quantify the degree of dependence among the nodes' states. We then develop a proof based on an independent approximation technique, which allows us to control the deviation introduced by dependencies and establish consistency results analogous to those in the independent case. 

The coefficient of absolute regularity $\beta(\cA,\cB)$ is defined in
Appendix \ref{app:preliminaryTools}. We impose the following finite-block
condition directly on the node states at every time.

\begin{assum}[Uniform-in-time finite-block absolute regularity]
\label{assum:betaMixing}
For $m\in\{1,2\}$, define
\begin{equation}\label{eq:timeBlockBeta}
\begin{aligned}
\beta_{t,n}(m,k)=\sup\ &
\beta\left(
\sigma\{(\rvs_a(t;\rvh),\rvs_a^\star(t)):
              \rvh\in\gH,\ a\in A\},
\sigma\{(\rvs_b(t;\rvh),\rvs_b^\star(t)):
              \rvh\in\gH,\ b\in B\}\right),\\
&\varnothing\ne A,B\subseteq\{1,\ldots,n\},\quad
|A|\le m,\quad
\min_{a\in A,b\in B}|a-b|\ge k .
\end{aligned}
\end{equation}
There exist constants $C_\beta<\infty$, $\xi\ge0$, and $\lambda>0$,
independent of $n$ and $t$, such that
\begin{equation}\label{eq:betaMixingDecay}
\beta_{t,n}(m,k)\le C_\beta m^\xi k^{-\lambda},
\qquad
t\in[0,T],\quad m\in\{1,2\},\quad 1\le k\le n-1.
\end{equation}
\end{assum}

\setcounter{remark}{4}
\begin{remark}[Verification from initial-time mixing]\label{rem:betaVerification}
Assumption \ref{assum:betaMixing} is imposed directly at every
$t\in[0,T]$. Under the finite-horizon ancestor or localization conditions in
Appendix \ref{app:betaVerification}, it follows from finite-block mixing of
the initial primitive variables. Appendix \ref{app:betaVerification} also verifies the condition
for a distance-preference network with an ordered, fixed-range Gaussian
kernel and gives the corresponding quenched growing-separation statement for
the untruncated Gaussian-kernel graph.
\end{remark}

\begin{problem}[\textbf{Hyperparameter Estimation via Observation of Weakly-Dependent Complex Network Dynamical System}]\label{defn:ProbWD}
Consider the observations and objective in Problem \ref{defn:ProbIID}, using
the nodewise response representation in
\eqref{eq:independentNodeProcessRepresentation}. For the candidate and
recorded runs, write
\begin{equation}\label{eq:wdStateRealizations}
\rvs_i(t;\rvh)=\tilde\psi_{t,n,i}\{\rvx_i(t),\rvh\},
\qquad
\rvs_i^\star(t)=\tilde\psi_{t,n,i}\{\rvx_i^\star(t),\rvh^\star\},
\end{equation}
and define the node-specific base pair
\begin{equation}\label{eq:wdNodeBaseBlock}
B_{t,i}=(\rvx_i(t),\rvx_i^\star(t)).
\end{equation}
Replace the fixed-time independence condition in Problem
\ref{defn:ProbIID} by Assumption \ref{assum:betaMixing}. The base pairs in
\eqref{eq:wdNodeBaseBlock} are generally dependent across $i$. The
construction in Proposition \ref{prop:nodewiseBaseRepresentation} ensures
that their finite-block coefficients are at most twice those in Assumption
\ref{assum:betaMixing}. The candidate and
recorded arrays are independent realizations, and
the observation noises are independent of both arrays. The joint
reconstruction in \eqref{eq:wdStateRealizations} has the same law as the
full simulator representation in \eqref{eq:Rosenblatt}. The remaining
settings are unchanged.
\end{problem}

A core technical tool in our proof is the independent approximation lemma proposed by Rio \cite{rio2013inequalities}, which allows us to approximate dependent random variables with independent ones while controlling the approximation error using the $\beta$-mixing coefficient.

\begin{lemma}[Independent Approximation \cite{rio2013inequalities}]\label{le:indepApprox}
Let $\cA$ be a $\sigma$-field on the probability space $(\Omega, \cF, \gP)$, and $\rvx$ be a random variable defined on this probability space with values in some Polish space. Then there exists a random variable $\rvx^\prime$, with the same law as $\rvx$, independent of $\cA$, such that $\gP(\rvx^\prime \ne \rvx) = \beta(\sigma(\rvx), \cA)$.
\end{lemma}

The finite-sequence coupling used below follows by recursively applying
Lemma \ref{le:indepApprox}.

\begin{lemma}[Finite-Sequence Independent Approximation]
\label{le:finiteSequenceApprox}
Let $B_1,\ldots,B_q$ be Polish-valued random elements and
suppose
\begin{equation}\label{eq:blockBerbeeCondition}
\beta\!\left(\sigma(B_r),
\sigma(B_{r+1},\ldots,B_q)\right)\le b_r,
\qquad r=1,\ldots,q-1.
\end{equation}
On an extension of the probability space, there are mutually independent
copies $B_1',\ldots,B_q'$ with $B_r'\stackrel{d}{=}B_r$ and
$\gP(B_r'\ne B_r)\le b_r$ for $r<q$. The coupling auxiliaries may be chosen
independently of any prescribed random element that is independent of
$(B_1,\ldots,B_q)$.
\end{lemma}

\begin{proof}
We proceed by induction from the first block to the last. For
$r=1,\ldots,q-1$, suppose that $B_1',\ldots,B_{r-1}'$ have been constructed
so that they are mutually independent and jointly independent of
$(B_r,\ldots,B_q)$; for $r=1$, the assertion is vacuous. Define
\begin{equation*}
\cA_r=\sigma(B_1',\ldots,B_{r-1}',B_{r+1},\ldots,B_q).
\end{equation*}
By the induction hypothesis, the already constructed copies are independent
of the untouched tail $(B_r,\ldots,B_q)$. Adjoining their $\sigma$-field to
the future blocks therefore does not change the absolute-regularity
coefficient, and
\begin{equation*}
\beta\!\left(\sigma(B_r),\cA_r\right)
=\beta\!\left(\sigma(B_r),
\sigma(B_{r+1},\ldots,B_q)\right)
\le b_r.
\end{equation*}
Applying Lemma \ref{le:indepApprox} on an extension with fresh auxiliary
randomness gives a copy $B_r'$ that is independent of $\cA_r$ and satisfies
\begin{equation*}
\gP(B_r'\ne B_r)=\beta\!\left(\sigma(B_r),\cA_r\right)
\le b_r.
\end{equation*}
Consequently, $B_1',\ldots,B_r'$ are mutually independent and jointly
independent of $(B_{r+1},\ldots,B_q)$, which closes the induction. Setting
$B_q'=B_q$ gives the asserted mutually independent sequence. If a prescribed
random element is independent of $(B_1,\ldots,B_q)$, take the product
extension and all fresh coupling auxiliaries independently of that element;
the resulting coupled sequence preserves the stated independence.
\end{proof}

\begin{thm}[Probabilistic hyperparameter bound under weakly-dependent condition]\label{thm:ConsistencyWD}
For Problem \ref{defn:ProbWD}, using the nodewise base representation
\eqref{eq:wdStateRealizations}, suppose that the dynamical system satisfies
Assumptions \ref{assum:dynSysLip}, \ref{assum:dynSysCompactMeasure},
\ref{assum:flowRegularity}, and \ref{assum:betaMixing}, the algorithm satisfies Assumption
\ref{assum:minimizer}, and the deterministic objective satisfies Assumption
\ref{assum:identifiability}. In addition, suppose that the following
conditions hold:
\begin{enumerate}
\item for every $t=1,\dots,T$, the observation map $\phi_t$ is uniformly
bounded on the finite-horizon reachable state set, with bounds independent
of $n$ and the node index;
\item for every $t=1,\dots,T$, the observation map $\phi_t$ is uniformly
Lipschitz continuous in its state argument on the finite-horizon reachable
state set, with Lipschitz constants independent of $n$ and the node index;
\item the hyperparameter space $\gH\subset\R^h$ is compact.
\end{enumerate}
Let $d$ be the Euclidean metric or a fixed norm-equivalent metric on $\gH$.
Then there is a constant $C_{\mathrm{WD}}$, independent
of $n$, such that, for every $\eps>0$ and $n\ge n_0$,
\begin{equation}\label{eq:wdFixedAccuracy}
\gP\left(d(\hat{\rvh}_n,\rvh^\star)\ge\eps\right)
\le\frac{C_{\mathrm{WD}}}{\eta(\eps)}
n^{-\lambda/(1+2\lambda)} .
\end{equation}
\end{thm}

Analogous to Theorem \ref{thm:ConsistencyIID}, Theorem
\ref{thm:ConsistencyWD} gives a fixed-accuracy bound for weakly-dependent
systems. The exponent records the cost of dependence and approaches the
independent-node exponent $1/2$ as $\lambda$ increases.

\begin{proof}
Retain $\Delta_n$ and the common fluctuation notation introduced in the
proof of Theorem \ref{thm:ConsistencyIID}. Under the weakly-dependent
nodewise realization \eqref{eq:wdStateRealizations} and the feature map
\eqref{eq:nodeSpecificObservationFeature}, the residual terms and their
averages are those in \eqref{eq:iidResidual}--\eqref{eq:iidResidualAverages},
the linear and centered quadratic fluctuations are those in
\eqref{eq:iidLinearFluctuation}--\eqref{eq:iidQuadraticFluctuation}, and the
corresponding kernels are those in \eqref{eq:iidLinearKernelDefinition},
\eqref{eq:iidCenteredQuadraticKernel}, and
\eqref{eq:iidQuadraticKernelIncrement}. The candidate and recorded arrays
are independent realizations with the same node-specific marginal laws at
$\rvh^\star$, so $a_{t,i}(\rvh^\star)=0$ for every $t,i$. Consequently,
the exact decomposition \eqref{eq:iidExactDecomposition} and the kernel form
\eqref{eq:iidQuadraticKernel} apply without change; in particular,
$\mathcal L_n(\rvh^\star)=0$. The feature bounds
\eqref{eq:iidFeatureBounds}--\eqref{eq:iidCenteredBounds}, and consequently
the increment estimates \eqref{eq:iidLinearProxy} and
\eqref{eq:iidPairProxy}, hold uniformly in $n$ and the node indices without
requiring identical node marginals.

Each residual process is measurable with respect to
$\sigma\{B_{t,i},\omega_{(t,i)}\}$. By Problem \ref{defn:ProbWD}, the
finite-block coefficient of the base pairs is bounded by twice the
corresponding coefficient in \eqref{eq:timeBlockBeta}. All independence used
below is therefore introduced by coupling. Fix an integer separation gap
$1\le\mathfrak g\le n-1$ for $n\ge2$; the case $n=1$ is absorbed by enlarging
the final constant.

For the linear process, partition the node indices into the $\mathfrak g$
residue classes
\begin{equation}\label{eq:wdResidueClasses}
I_s=\{s+q\mathfrak g:q\ge0,\ s+q\mathfrak g\le n\},
\qquad s=1,\ldots,\mathfrak g.
\end{equation}
Let $q_s=|I_s|$. Then $\sum_sq_s=n$ and, by Cauchy--Schwarz,
\begin{equation}\label{eq:wdResidueCardinality}
\left(\sum_{s=1}^{\mathfrak g}\sqrt{q_s}\right)^2
\le\mathfrak g\sum_{s=1}^{\mathfrak g}q_s
=n\mathfrak g.
\end{equation}
Nodes in a fixed $I_s$ are separated by at least $\mathfrak g$. For every
$t,s$, Lemma \ref{le:finiteSequenceApprox} therefore gives copies
$B_{t,i}'=(\rvx_i'(t),\rvx_i^{\star\prime}(t))$, $i\in I_s$, that are
mutually independent within $I_s$ and satisfy
\begin{equation*}
\gP\{B_{t,i}'\ne B_{t,i}\}
\le2\beta_{t,n}(1,\mathfrak g),
\qquad i\in I_s.
\end{equation*}
All classwise couplings may be placed on one product extension; independence
between different classes is neither asserted nor needed. Define
\begin{equation*}
\rvs_i'(t;\rvh)=\tilde\psi_{t,n,i}\{\rvx_i'(t),\rvh\},
\qquad
\rvs_i^{\star\prime}(t)
=\tilde\psi_{t,n,i}\{\rvx_i^{\star\prime}(t),\rvh^\star\}.
\end{equation*}
Using the unchanged observation noise, set
\begin{equation*}
\begin{aligned}
l_{t,i}'(\rvh)
&=\varphi_{t,i}\{\rvx_i'(t),\rvh\}
-\varphi_{t,i}\{\rvx_i^{\star\prime}(t),\rvh^\star\}
-\omega_{(t,i)},\\
\zeta_{t,i}'(\rvh)&=l_{t,i}'(\rvh)-a_{t,i}(\rvh),\\
F_{t,i}'(\rvh)
&=2\langle\bar a_{t,n}(\rvh),\zeta_{t,i}'(\rvh)\rangle.
\end{aligned}
\end{equation*}
The coupling auxiliaries are independent of the noise. Consequently,
$F_{t,i}'$ has the same process law as $F_{t,i}$ and is centered. Moreover,
for a deterministic constant $C_{L,t}$,
\begin{align}
\sup_{\rvh\in\gH}|F_{t,i}(\rvh)-F_{t,i}'(\rvh)|
&\le C_{L,t}\mathbf 1_{\{B_{t,i}'\ne B_{t,i}\}},\notag\\
\E\sup_{\rvh\in\gH}|F_{t,i}(\rvh)-F_{t,i}'(\rvh)|
&\le 2C_{L,t}\beta_{t,n}(1,\mathfrak g).
\label{eq:wdLinearCouplingError}
\end{align}
The indicator bound follows because the two processes agree on the coupling
event, while the common noise cancels from their difference on its
complement.

Define the coupled linear processes
\begin{equation}\label{eq:wdCoupledLinearProcess}
\widetilde{\mathcal L}_{t,n,\mathfrak g}(\rvh)
=\frac1n\sum_{s=1}^{\mathfrak g}\sum_{i\in I_s}F_{t,i}'(\rvh),
\qquad
\widetilde{\mathcal L}_{n,\mathfrak g}(\rvh)
=\frac1T\sum_{t=1}^T
\widetilde{\mathcal L}_{t,n,\mathfrak g}(\rvh).
\end{equation}
Equations \eqref{eq:wdLinearCouplingError} and
\eqref{eq:betaMixingDecay} give the explicit approximation bound
\begin{align}
\E\sup_{\rvh\in\gH}
|\mathcal L_n(\rvh)-\widetilde{\mathcal L}_{n,\mathfrak g}(\rvh)|
&\le\frac1{Tn}\sum_{t=1}^T\sum_{s=1}^{\mathfrak g}
\sum_{i\in I_s}
\E\sup_{\rvh\in\gH}|F_{t,i}(\rvh)-F_{t,i}'(\rvh)|\notag\\
&\le C_{\mathcal L,\beta}\mathfrak g^{-\lambda}.
\label{eq:wdTotalLinearCouplingError}
\end{align}

It remains to concentrate the coupled process. For each fixed $t,s$, the
increment bound \eqref{eq:iidLinearProxy} and independence within $I_s$
imply
\begin{equation*}
\E\exp\left[
u\sum_{i\in I_s}\{F_{t,i}'(\rvh)-F_{t,i}'(\rvg)\}\right]
\le
\exp\left\{\frac{u^2q_sv_{\mathcal L,t}^2
d(\rvh,\rvg)^2}{2}\right\}.
\end{equation*}
To make the aggregation of possibly dependent class sums explicit, suppose
that $X_1,\ldots,X_m$ satisfy
$\E\exp(uX_j)\le\exp(u^2\sigma_j^2/2)$ for all $u\in\R$, and put
$V=\sum_{j=1}^m\sigma_j>0$. H\"older's inequality with exponents
$V/\sigma_j$ yields
\begin{align}
\E\exp\left(u\sum_{j=1}^mX_j\right)
&=\E\prod_{j=1}^m\exp(uX_j)\notag\\
&\le\prod_{j=1}^m
\left\{\E\exp\left(\frac{uV}{\sigma_j}X_j\right)\right\}^{\sigma_j/V}
\le\exp\left(\frac{u^2V^2}{2}\right),
\qquad u\in\R.
\label{eq:wdGeneralSubGaussianAggregation}
\end{align}
Terms with $\sigma_j=0$ are omitted. The class sums need not be mutually
independent. Applying \eqref{eq:wdGeneralSubGaussianAggregation} with
$\sigma_s=n^{-1}v_{\mathcal L,t}\sqrt{q_s}$, followed by
\eqref{eq:wdResidueCardinality}, gives
\begin{equation}\label{eq:wdLinearFixedTimeMGF}
\E\exp\left[u\{\widetilde{\mathcal L}_{t,n,\mathfrak g}(\rvh)
-\widetilde{\mathcal L}_{t,n,\mathfrak g}(\rvg)\}\right]
\le
\exp\left\{\frac{u^2v_{\mathcal L,t}^2}{2n^2}
\left(\sum_{s=1}^{\mathfrak g}\sqrt{q_s}\right)^2
d(\rvh,\rvg)^2\right\}
\le
\exp\left\{\frac{u^2v_{\mathcal L,t}^2\mathfrak g
d(\rvh,\rvg)^2}{2n}\right\}.
\end{equation}
Finally, apply \eqref{eq:wdGeneralSubGaussianAggregation} to the time
average in \eqref{eq:wdCoupledLinearProcess} with
$\sigma_t=T^{-1}v_{\mathcal L,t}\sqrt{\mathfrak g/n}\,d(\rvh,\rvg)$.
Writing $v_{\mathcal L}=T^{-1}\sum_tv_{\mathcal L,t}$ yields the formal
sub-Gaussian increment bound
\begin{equation}\label{eq:wdLinearProxy}
\E\exp\left[u\{\widetilde{\mathcal L}_{n,\mathfrak g}(\rvh)
-\widetilde{\mathcal L}_{n,\mathfrak g}(\rvg)\}\right]
\le
\exp\left\{\frac{u^2v_{\mathcal L}^2\mathfrak g
d(\rvh,\rvg)^2}{2n}\right\},
\qquad u\in\R.
\end{equation}
No independence across residue classes or observation times is used in
\eqref{eq:wdLinearFixedTimeMGF}--\eqref{eq:wdLinearProxy}.

For the quadratic process, use the deterministic matching partition
$\{\mathcal M_j:1\le j\le J_n\}$ from Lemma
\ref{le:cyclicDecomposition}, where $J_n\le n+1$. Within each matching,
join two edges in a conflict graph when the minimum ordering distance between
their endpoint sets is less than $\mathfrak g$. A fixed edge has at most
$4\mathfrak g$ neighbors, so a greedy coloring produces classes
$\mathcal C_{j,c}$ satisfying
\begin{equation}\label{eq:wdColorClassCounts}
C_j\le4\mathfrak g+1,
\qquad
q_{j,c}\triangleq|\mathcal C_{j,c}|,
\qquad
\sum_{c=1}^{C_j}q_{j,c}=|\mathcal M_j|\le\frac n2.
\end{equation}
Every $\mathcal C_{j,c}$ consists of node-disjoint edges, and distinct
two-node endpoint sets in the class are separated by at least
$\mathfrak g$.

For an edge $e=\{a,b\}$, let
\begin{equation*}
B_{t,e}=(B_{t,a},B_{t,b}),
\qquad
K_{t,e}=K_{t,a,b}+K_{t,b,a},
\qquad
H_{t,e}=H_{t,a,b}+H_{t,b,a}.
\end{equation*}
Order the edges of a fixed $\mathcal C_{j,c}$ as $e_1,\ldots,e_q$. The
endpoint set of $e_r$ is separated by at least $\mathfrak g$ from the union
of the later endpoint sets. Thus \eqref{eq:timeBlockBeta} verifies
\eqref{eq:blockBerbeeCondition} with
$b_r=2\beta_{t,n}(2,\mathfrak g)$, and Lemma
\ref{le:finiteSequenceApprox} yields mutually independent copies
$\{B_{t,e}':e\in\mathcal C_{j,c}\}$ such that
\begin{equation*}
\gP\{B_{t,e}'\ne B_{t,e}\}
\le2\beta_{t,n}(2,\mathfrak g).
\end{equation*}
Define $K_{t,e}'$ and
$H_{t,e}'(\rvh,\rvg)=K_{t,e}'(\rvh)-K_{t,e}'(\rvg)$ from the copied
endpoint block and the unchanged endpoint noises. Construct the diagonal
copies $K_{t,i,i}'$ by the residue-class coupling in
\eqref{eq:wdResidueClasses}, and put
$H_{t,i,i}'(\rvh,\rvg)=K_{t,i,i}'(\rvh)-K_{t,i,i}'(\rvg)$.
The copies retain the corresponding kernel laws and are centered.

Define the off-diagonal and diagonal coupled components, the fixed-time
process, and its time average by
\begin{align}
\widetilde{\mathcal Q}_{t,n,\mathfrak g}^{\mathrm{off}}(\rvh)
&\triangleq\frac1{n^2}
\sum_{j=1}^{J_n}\sum_{c=1}^{C_j}
\sum_{e\in\mathcal C_{j,c}}K_{t,e}'(\rvh),
\label{eq:wdCoupledQuadraticOffDiagonal}\\
\widetilde{\mathcal Q}_{t,n,\mathfrak g}^{\mathrm{diag}}(\rvh)
&\triangleq\frac1{n^2}\sum_{i=1}^nK_{t,i,i}'(\rvh),
\label{eq:wdCoupledQuadraticDiagonal}\\
\widetilde{\mathcal Q}_{t,n,\mathfrak g}(\rvh)
&\triangleq
\widetilde{\mathcal Q}_{t,n,\mathfrak g}^{\mathrm{off}}(\rvh)
+\widetilde{\mathcal Q}_{t,n,\mathfrak g}^{\mathrm{diag}}(\rvh),
\label{eq:wdCoupledQuadraticFixedTime}\\
\widetilde{\mathcal Q}_{n,\mathfrak g}(\rvh)
&\triangleq\frac1T\sum_{t=1}^T
\widetilde{\mathcal Q}_{t,n,\mathfrak g}(\rvh).
\label{eq:wdCoupledQuadraticProcess}
\end{align}
Couplings for different $(t,j,c)$ may be dependent. Only the copies inside
one matching-color class must be independent.

On the event $B_{t,e}'=B_{t,e}$, the original and copied kernel increments
agree. On the complementary event, the feature factors are bounded and the
Gaussian noise factors have finite first moments. Independence of the
coupling auxiliary and the noises therefore gives
\begin{equation}\label{eq:wdCouplingError}
\E\sup_{\rvh,\rvg\in\gH}
|H_{t,e}(\rvh,\rvg)-H_{t,e}'(\rvh,\rvg)|
\le C_{Q,t}B_t(B_t+\sigma_t\sqrt{r_t})
\beta_{t,n}(2,\mathfrak g).
\end{equation}
The diagonal analogue has the same form with
$\beta_{t,n}(1,\mathfrak g)$. Summing over the unordered edges and diagonal
indices, dividing by $Tn^2$, and using \eqref{eq:betaMixingDecay} gives
\begin{align}
&\E\sup_{\rvh,\rvg\in\gH}
\left|\{\mathcal Q_n(\rvh)-\mathcal Q_n(\rvg)\}
-\{\widetilde{\mathcal Q}_{n,\mathfrak g}(\rvh)
-\widetilde{\mathcal Q}_{n,\mathfrak g}(\rvg)\}\right|\notag\\
&\quad\le\frac1{Tn^2}\sum_{t=1}^T
\left\{\sum_{j=1}^{J_n}\sum_{c=1}^{C_j}
\sum_{e\in\mathcal C_{j,c}}
\E\sup_{\rvh,\rvg\in\gH}|H_{t,e}-H_{t,e}'|
+\sum_{i=1}^n\E\sup_{\rvh,\rvg\in\gH}
|H_{t,i,i}-H_{t,i,i}'|\right\}
\le C_{\mathcal Q,\beta}\mathfrak g^{-\lambda}.
\label{eq:wdTotalCouplingError}
\end{align}

We next derive the increment bound for the coupled quadratic process.
Equation \eqref{eq:iidPairProxy} gives, for every $u\in\R$,
\begin{equation*}
\E\exp\{uH_{t,e}'(\rvh,\rvg)\}
\le\exp\left\{\frac{u^2C_E\nu_t^2d(\rvh,\rvg)^2}{2}\right\}.
\end{equation*}
Independence of the edge blocks inside $\mathcal C_{j,c}$ then gives
\begin{equation*}
\E\exp\left[u\sum_{e\in\mathcal C_{j,c}}
H_{t,e}'(\rvh,\rvg)\right]
\le\exp\left\{\frac{u^2C_Eq_{j,c}\nu_t^2
d(\rvh,\rvg)^2}{2}\right\}.
\end{equation*}
For each matching, \eqref{eq:wdColorClassCounts} and Cauchy--Schwarz yield
\begin{align}
\sum_{c=1}^{C_j}\sqrt{q_{j,c}}
&\le\sqrt{C_j\sum_{c=1}^{C_j}q_{j,c}}
\le\sqrt{(4\mathfrak g+1)|\mathcal M_j|}
\le C\sqrt{n\mathfrak g},\notag\\
\sum_{j=1}^{J_n}\sum_{c=1}^{C_j}\sqrt{q_{j,c}}
&\le Cn\sqrt{n\mathfrak g}.
\label{eq:wdMatchingAggregation}
\end{align}
Define the unnormalized off-diagonal increment
\begin{equation*}
S_{t,\mathrm{off}}(\rvh,\rvg)
\triangleq\sum_{j=1}^{J_n}\sum_{c=1}^{C_j}
\sum_{e\in\mathcal C_{j,c}}H_{t,e}'(\rvh,\rvg).
\end{equation*}
Applying \eqref{eq:wdGeneralSubGaussianAggregation} across all
matching-color classes and then using
\eqref{eq:wdMatchingAggregation} gives, for every $u\in\R$,
\begin{equation*}
\begin{aligned}
\E\exp\{uS_{t,\mathrm{off}}(\rvh,\rvg)\}
&\le\exp\left\{\frac{u^2C_E\nu_t^2}{2}
\left(\sum_{j=1}^{J_n}\sum_{c=1}^{C_j}\sqrt{q_{j,c}}\right)^2
d(\rvh,\rvg)^2\right\}\\
&\le\exp\left\{\frac{u^2Cn^3\mathfrak g\,\nu_t^2
d(\rvh,\rvg)^2}{2}\right\}.
\end{aligned}
\end{equation*}
Because the increment of \eqref{eq:wdCoupledQuadraticOffDiagonal} equals
$n^{-2}S_{t,\mathrm{off}}(\rvh,\rvg)$, the off-diagonal component satisfies
\begin{equation}\label{eq:wdOffDiagonalFixedTimeMGF}
\E\exp\left[u\{\widetilde{\mathcal Q}_{t,n,\mathfrak g}^{\mathrm{off}}(\rvh)
-\widetilde{\mathcal Q}_{t,n,\mathfrak g}^{\mathrm{off}}(\rvg)\}\right]
\le\exp\left\{\frac{u^2C\mathfrak g\nu_t^2
d(\rvh,\rvg)^2}{2n}\right\}.
\end{equation}

For the diagonal component, define
\begin{equation*}
S_{t,\mathrm{diag}}(\rvh,\rvg)
\triangleq\sum_{i=1}^nH_{t,i,i}'(\rvh,\rvg).
\end{equation*}
Independence within each residue class, \eqref{eq:iidPairProxy},
\eqref{eq:wdGeneralSubGaussianAggregation}, and
\eqref{eq:wdResidueCardinality} give
\begin{equation*}
\begin{aligned}
\E\exp\{uS_{t,\mathrm{diag}}(\rvh,\rvg)\}
&\le\exp\left\{\frac{u^2C\nu_t^2}{2}
\left(\sum_{s=1}^{\mathfrak g}\sqrt{q_s}\right)^2
d(\rvh,\rvg)^2\right\}\\
&\le\exp\left\{\frac{u^2Cn\mathfrak g\,\nu_t^2
d(\rvh,\rvg)^2}{2}\right\}.
\end{aligned}
\end{equation*}
Since the increment of \eqref{eq:wdCoupledQuadraticDiagonal} is
$n^{-2}S_{t,\mathrm{diag}}(\rvh,\rvg)$,
\begin{equation*}
\E\exp\left[u\{\widetilde{\mathcal Q}_{t,n,\mathfrak g}^{\mathrm{diag}}(\rvh)
-\widetilde{\mathcal Q}_{t,n,\mathfrak g}^{\mathrm{diag}}(\rvg)\}\right]
\le\exp\left\{\frac{u^2C\mathfrak g\nu_t^2
d(\rvh,\rvg)^2}{2n^3}\right\}.
\end{equation*}
The diagonal bound is dominated by
\eqref{eq:wdOffDiagonalFixedTimeMGF}. Applying
\eqref{eq:wdGeneralSubGaussianAggregation} first to the two components in
\eqref{eq:wdCoupledQuadraticFixedTime} gives
\begin{equation*}
\E\exp\left[u\{\widetilde{\mathcal Q}_{t,n,\mathfrak g}(\rvh)
-\widetilde{\mathcal Q}_{t,n,\mathfrak g}(\rvg)\}\right]
\le\exp\left\{\frac{u^2C_{Q,0}^2\mathfrak g\nu_t^2
d(\rvh,\rvg)^2}{2n}\right\}.
\end{equation*}
Applying it again to the time average in
\eqref{eq:wdCoupledQuadraticProcess} gives
\begin{equation}\label{eq:wdQuadraticProxy}
\E\exp\left[u\{\widetilde{\mathcal Q}_{n,\mathfrak g}(\rvh)
-\widetilde{\mathcal Q}_{n,\mathfrak g}(\rvg)\}\right]
\le\exp\left\{\frac{u^2v_{\mathcal Q,\mathrm{WD}}^2\mathfrak g
d(\rvh,\rvg)^2}{2n}\right\},
\qquad
v_{\mathcal Q,\mathrm{WD}}
=\frac{C_{Q,0}}T\sum_{t=1}^T\nu_t.
\end{equation}
Here $C_{Q,0}$ is independent of $n$ and $\mathfrak g$.
This calculation uses no independence across matching-color classes or
observation times.

It remains to control the quadratic anchor at $\rvh^\star$. With
$K_{t,e}$ and $K_{t,e}'$ as above, the mismatch-event calculation in
\eqref{eq:wdCouplingError} gives
\begin{equation*}
\E|K_{t,e}(\rvh^\star)-K_{t,e}'(\rvh^\star)|
\le C_{\mathrm{anc},t}\beta_{t,n}(2,\mathfrak g).
\end{equation*}
The diagonal analogue is bounded by
$C_{\mathrm{anc},t}\beta_{t,n}(1,\mathfrak g)$. Hence
\begin{equation}\label{eq:wdAnchorTotalCoupling}
\E|\mathcal Q_n(\rvh^\star)
-\widetilde{\mathcal Q}_{n,\mathfrak g}(\rvh^\star)|
\le C_{\mathrm{anc},\beta}\mathfrak g^{-\lambda}.
\end{equation}
Within one matching-color class, the copied centered kernels are independent
and have uniformly bounded second moments. Thus
\begin{equation*}
\E\left|\sum_{e\in\mathcal C_{j,c}}
K_{t,e}'(\rvh^\star)\right|
\le\left\{\E\left|\sum_{e\in\mathcal C_{j,c}}
K_{t,e}'(\rvh^\star)\right|^2\right\}^{1/2}
\le C\sqrt{q_{j,c}}.
\end{equation*}
Using \eqref{eq:wdMatchingAggregation}, the normalized off-diagonal anchor
satisfies
\begin{equation*}
\frac1{n^2}\sum_{j=1}^{J_n}\sum_{c=1}^{C_j}
\E\left|\sum_{e\in\mathcal C_{j,c}}
K_{t,e}'(\rvh^\star)\right|
\le\frac C{n^2}\sum_{j=1}^{J_n}\sum_{c=1}^{C_j}\sqrt{q_{j,c}}
\le C\sqrt{\frac{\mathfrak g}{n}}.
\end{equation*}
The same argument over the residue classes bounds the diagonal anchor by
$C\sqrt{\mathfrak g/n^3}$. Averaging over $t$ and applying the triangle
inequality therefore gives
\begin{equation}\label{eq:wdCoupledAnchorBound}
\E|\widetilde{\mathcal Q}_{n,\mathfrak g}(\rvh^\star)|
\le C\sqrt{\frac{\mathfrak g}{n}}.
\end{equation}
Combining \eqref{eq:wdAnchorTotalCoupling} and
\eqref{eq:wdCoupledAnchorBound} yields
\begin{equation}\label{eq:wdAnchorBound}
\E|\mathcal Q_n(\rvh^\star)|
\le C\left\{\mathfrak g^{-\lambda}
+\sqrt{\frac{\mathfrak g}{n}}\right\}.
\end{equation}

Equations \eqref{eq:wdLinearProxy} and \eqref{eq:wdQuadraticProxy} give the
respective sub-Gaussian metrics
\begin{equation*}
d_{\mathcal L,n,\mathfrak g}(\rvh,\rvg)
=v_{\mathcal L}\sqrt{\frac{\mathfrak g}{n}}d(\rvh,\rvg),
\qquad
d_{\mathcal Q,n,\mathfrak g}(\rvh,\rvg)
=v_{\mathcal Q,\mathrm{WD}}\sqrt{\frac{\mathfrak g}{n}}
d(\rvh,\rvg).
\end{equation*}
The coupled processes have continuous sample paths on compact $\gH$ and are
therefore separable. Since
$\widetilde{\mathcal L}_{n,\mathfrak g}(\rvh^\star)=0$, Dudley's theorem
and Lemma \ref{le:BoundedEntropy} imply
\begin{align}
\E\sup_{\rvh\in\gH}
|\widetilde{\mathcal L}_{n,\mathfrak g}(\rvh)|
&\le C_{\mathrm{ch}}\sqrt h\,\diam(\gH)v_{\mathcal L}
\sqrt{\frac{\mathfrak g}{n}},\label{eq:wdLinearChaining}\\
\E\sup_{\rvh\in\gH}
|\widetilde{\mathcal Q}_{n,\mathfrak g}(\rvh)
-\widetilde{\mathcal Q}_{n,\mathfrak g}(\rvh^\star)|
&\le C_{\mathrm{ch}}\sqrt h\,\diam(\gH)
v_{\mathcal Q,\mathrm{WD}}
\sqrt{\frac{\mathfrak g}{n}}.
\label{eq:wdQuadraticChaining}
\end{align}

The preceding estimates enter $\Delta_n$ through the deterministic
inequality
\begin{align}
\Delta_n
&\le
\sup_{\rvh\in\gH}|\mathcal L_n(\rvh)
-\widetilde{\mathcal L}_{n,\mathfrak g}(\rvh)|
+\sup_{\rvh\in\gH}|\widetilde{\mathcal L}_{n,\mathfrak g}(\rvh)|\notag\\
&\quad+
\sup_{\rvh\in\gH}\left|\{\mathcal Q_n(\rvh)-\mathcal Q_n(\rvh^\star)\}
-\{\widetilde{\mathcal Q}_{n,\mathfrak g}(\rvh)
-\widetilde{\mathcal Q}_{n,\mathfrak g}(\rvh^\star)\}\right|\notag\\
&\quad+
\sup_{\rvh\in\gH}|\widetilde{\mathcal Q}_{n,\mathfrak g}(\rvh)
-\widetilde{\mathcal Q}_{n,\mathfrak g}(\rvh^\star)|
+|\mathcal Q_n(\rvh^\star)|.
\label{eq:wdDeviationDecomposition}
\end{align}
Taking expectations in \eqref{eq:wdDeviationDecomposition} and applying
\eqref{eq:wdTotalLinearCouplingError},
\eqref{eq:wdTotalCouplingError}, \eqref{eq:wdAnchorBound},
\eqref{eq:wdLinearChaining}, and \eqref{eq:wdQuadraticChaining} gives
\begin{equation}\label{eq:wdObjectiveDeviation}
\E\Delta_n
\le C_1\left\{\mathfrak g^{-\lambda}
+\sqrt{\frac{\mathfrak g}{n}}\right\}.
\end{equation}
Choose
\begin{equation*}
\mathfrak g_n=\left\lfloor n^{1/(1+2\lambda)}\right\rfloor\vee1.
\end{equation*}
For $n\ge2$,
$\tfrac12n^{1/(1+2\lambda)}\le\mathfrak g_n
\le n^{1/(1+2\lambda)}$, so
\begin{equation*}
\mathfrak g_n^{-\lambda}
\le2^\lambda n^{-\lambda/(1+2\lambda)},
\qquad
\sqrt{\frac{\mathfrak g_n}{n}}
\le n^{-\lambda/(1+2\lambda)}.
\end{equation*}
Substitution into \eqref{eq:wdObjectiveDeviation} gives
\begin{equation}\label{eq:wdObjectiveRate}
\E\Delta_n
\le C_2 n^{-\lambda/(1+2\lambda)}.
\end{equation}
Finally, Lemma \ref{le:SIMarkov} and Assumption
\ref{assum:identifiability} yield
\begin{equation*}
\gP\{d(\hat\rvh_n,\rvh^\star)\ge\eps\}
\le\frac2{\eta(\eps)}\E\Delta_n
\le\frac{2C_2}{\eta(\eps)}
n^{-\lambda/(1+2\lambda)}.
\end{equation*}
This is \eqref{eq:wdFixedAccuracy} with $C_{\mathrm{WD}}=2C_2$.
\end{proof}

Relative to Theorem \ref{thm:ConsistencyIID}, Theorem
\ref{thm:ConsistencyWD} retains the same dynamical-system,
observation-regularity, compactness, minimizer, and objective-identifiability
conditions, but replaces fixed-time independence across nodes with the uniform
finite-block absolute-regularity condition in Assumption
\ref{assum:betaMixing}. Both theorems allow heterogeneous finite-time node
marginals and require no independence across observation times; the broader
node-dependence scope of Theorem \ref{thm:ConsistencyWD} changes the
fixed-accuracy scale from $n^{-1/2}$ in \eqref{eq:iidFixedAccuracy} to
$n^{-\lambda/(1+2\lambda)}$ in \eqref{eq:wdFixedAccuracy}, with the latter
approaching $n^{-1/2}$ as $\lambda$ increases. As in Theorem
\ref{thm:ConsistencyIID}, however, Theorem \ref{thm:ConsistencyWD} controls a
prescribed accuracy $\eps$ through the identifiability margin $\eta(\eps)$,
leaving the confidence-dependent parameter radius implicit. Corollary
\ref{coro:wdRadius} replaces Assumption \ref{assum:identifiability} by the
stronger local-growth and strict outer-gap conditions
\eqref{eq:localGrowthCondition}--\eqref{eq:strictOuterGap}: the outer gap first
localizes the estimator, and local growth converts the
$n^{-\lambda/(1+2\lambda)}$ objective-deviation scale into an explicit
$1-\rho$ parameter radius of order
$\{\rho^{-1}n^{-\lambda/(1+2\lambda)}\}^{1/\alpha}$. Thus, the corollary gives
a more explicit nonasymptotic hyperparameter bound for a narrower class of
deterministic objectives under the same weak-dependence setting.

\begin{coro}[Nonasymptotic hyperparameter bound under weakly-dependent condition]\label{coro:wdRadius}
Suppose that all the conditions of Theorem \ref{thm:ConsistencyWD} other
than Assumption \ref{assum:identifiability} hold, and suppose instead that
the uniform local-growth and strict outer-gap conditions
\eqref{eq:localGrowthCondition}--\eqref{eq:strictOuterGap} hold. Let $C_2$
be the constant in \eqref{eq:wdObjectiveRate}. Then, for every $n\ge n_0$
and $0<\rho<1$ satisfying
\begin{equation}\label{eq:wdOuterGapThreshold}
\frac{2C_2}{\rho}n^{-\lambda/(1+2\lambda)}<\Gamma,
\end{equation}
the estimator satisfies, with probability at least $1-\rho$,
\begin{equation}\label{eq:wdRadius}
d(\hat{\rvh}_n,\rvh^\star)
\le
\left(\frac{2C_2}{c\rho}
n^{-\lambda/(1+2\lambda)}\right)^{1/\alpha}.
\end{equation}
\end{coro}

\begin{proof}
Lemma \ref{le:localGrowthIdentifiability} shows that
\eqref{eq:localGrowthCondition}--\eqref{eq:strictOuterGap} imply Assumption
\ref{assum:identifiability}. Thus all the conditions of Theorem
\ref{thm:ConsistencyWD} are satisfied. Set
\begin{equation*}
\Delta_n=\sup_{\rvh\in\gH}|L_n(\rvh)-L^\star(\rvh)|
\end{equation*}
and
\begin{equation*}
\mathcal E_{n,\rho}^{\mathrm{WD}}
=\left\{\Delta_n\le
\frac{C_2}{\rho}n^{-\lambda/(1+2\lambda)}\right\}.
\end{equation*}
Using Markov's inequality together with \eqref{eq:wdObjectiveRate} gives
\begin{equation*}
\begin{aligned}
\gP\bigl((\mathcal E_{n,\rho}^{\mathrm{WD}})^c\bigr)
&=\gP\left(\Delta_n>
\frac{C_2}{\rho}n^{-\lambda/(1+2\lambda)}\right)\\
&\le\frac{\rho n^{\lambda/(1+2\lambda)}}{C_2}\E\Delta_n
\le\rho.
\end{aligned}
\end{equation*}
Consequently, $\gP(\mathcal E_{n,\rho}^{\mathrm{WD}})\ge1-\rho$.

On $\mathcal E_{n,\rho}^{\mathrm{WD}}$, the minimizing property of
$\hat{\rvh}_n$ yields
\begin{equation}\label{eq:wdCorollaryExcess}
\begin{aligned}
L^\star(\hat{\rvh}_n)-L^\star(\rvh^\star)
&=L^\star(\hat{\rvh}_n)-L_n(\hat{\rvh}_n)
  +L_n(\hat{\rvh}_n)-L_n(\rvh^\star)\\
&\quad+L_n(\rvh^\star)-L^\star(\rvh^\star)\\
&\le
\bigl|L^\star(\hat{\rvh}_n)-L_n(\hat{\rvh}_n)\bigr|
+\bigl|L_n(\rvh^\star)-L^\star(\rvh^\star)\bigr|\\
&\le2\Delta_n
\le\frac{2C_2}{\rho}n^{-\lambda/(1+2\lambda)}.
\end{aligned}
\end{equation}
If $d(\hat{\rvh}_n,\rvh^\star)\ge r_0$, then
\eqref{eq:strictOuterGap} and \eqref{eq:wdCorollaryExcess} give
\begin{equation*}
\Gamma\le L^\star(\hat{\rvh}_n)-L^\star(\rvh^\star)
\le\frac{2C_2}{\rho}n^{-\lambda/(1+2\lambda)},
\end{equation*}
contradicting \eqref{eq:wdOuterGapThreshold}. Hence
$d(\hat{\rvh}_n,\rvh^\star)<r_0$. The conclusion is immediate when
$\hat{\rvh}_n=\rvh^\star$. Otherwise,
\eqref{eq:localGrowthCondition} and \eqref{eq:wdCorollaryExcess} give
\begin{equation*}
c\,d(\hat{\rvh}_n,\rvh^\star)^\alpha
\le\frac{2C_2}{\rho}n^{-\lambda/(1+2\lambda)}.
\end{equation*}
Taking the $1/\alpha$ power proves \eqref{eq:wdRadius} on an event of
probability at least $1-\rho$.
\end{proof}

\section{Numerical Experiments}\label{sec:Experiments}

Having established a rigorous theoretical framework for the consistency of hyperparameter estimation in both independent and weakly-dependent systems, we now shift our focus to empirical validation. The convergence rates derived in Theorems \ref{thm:ConsistencyIID} and \ref{thm:ConsistencyWD} provide a formal guarantee, but their practical relevance hinges on whether these effects are observable in realistic simulations.

Section \ref{sec:Experiments} bridges theory and practice through two distinct and widely-used models: a Susceptible-Infectious-Susceptible (SIS) model in Section \ref{sec:SIS}, which is a variant of Susceptible-Infected-Resistant (SIR) model \cite{shaman2012forecasting, shaman2013real, rasmussen2011inference, yang2014comparison}, representing discrete-state dynamics common in epidemiology, and a spiking neuronal network (SNN) model in Section \ref{sec:SNN} \cite{zhang2024framework, lu2024simulation, lu2024imitating}, representing continuous-state dynamics in neuroscience. For both systems, we will show that the empirical estimation error decreases as the network population size grows, providing strong numerical evidence that aligns with our theoretical guarantees.

\subsection{Data Assimilation in SIS Models}\label{sec:SIS}
Susceptible-Infectious-Susceptible model is a classical framework for describing the spread of infectious diseases in which previously infected individuals become susceptible again after recovery. Examples of diseases that fit this pattern include the common cold, seasonal influenza, and COVID-19 \cite{anderson1991infectious, daley1964epidemics, boccaletti2006complex, kermack1927contribution, cao2022mepognn, hethcote2000mathematics, keeling2005networks, he2020seir, basnarkov2021seair, ghostine2021extended, lu2013optimizing, xu2019cybersecurity}. 

In the SIS framework, the population of interest is represented as a network, where each individual is modeled as a node $i$. For a network with network population size $n$, the connections between individuals are described by an adjacency matrix $\mA$, where $a_{i,j} = 1$ indicates that node $i$ is connected to node $j$, and $a_{i,j}=0$ if there is no connection between node $i$ and node $j$. Let $\gamma_i$ and $\lambda_i$ be the infection rate and recovery rate of the concerning diseases, respectively. Since individuals may have different levels of immunity, these rates are not necessarily identical across nodes. However, once the disease is specified, the rates are assumed to follow certain probability distributions, that is, $\gamma_i \sim p(\gamma \mid \rvh^\star_{\gamma, \text{dyn}})$ and similarly for the recovery rate $\lambda_i \sim p(\lambda \mid \rvh^\star_{\lambda, \text{dyn}})$, where $\rvh^\star_{\gamma, \text{dyn}}$ and $\rvh^\star_{\lambda, \text{dyn}}$ are the corresponding true hyperparameters. Let $\rvxi_i(t)$ be the state of node $i$ at time $t$, which is defined by $\rvxi_i(t) = 0$ representing that the individual is susceptible and $\rvxi_i(t) = 1$ representing that the individual is infectious. For $i = 1, \dots, n$ and  $t = 1, \dots, T$, the SIS dynamical system can then be formulated as:
\begin{equation*}
    \begin{matrix}\left\{
        \begin{aligned}
            \rvxi_i(t) &= \left(1-\rvxi_i(t-1)\right)\cdot \mathbb{I}\left[\rvu_i(t) \le I_i(t)\right] + \rvxi_i(t-1) \cdot \mathbb{I}\left[\rvu_i^\prime(t) \le S_i(t) \right], \\
            S_i(t) &= \lambda_i, \\
            I_i(t) &= 1 - \prod\limits_{1 \le k \le n, k \ne i} \left(1- \gamma_i a_{k,i} \rvxi_k(t-1)\right), \\
            \lambda_i &\sim p(\lambda \mid \rvh^\star_{\lambda, \text{dyn}}), \\
            \gamma_i &\sim p(\gamma \mid \rvh^\star_{\gamma, \text{dyn}}),\\
            \rvxi_i(0) &\sim Bernoulli(p^\star), \\
            \rvu_i(t)& ,\rvu_i^\prime(t) \sim U\left([0,1]\right),
        \end{aligned} \right.
    \end{matrix}
\end{equation*}
where $\{\rvu_i(t)\}_{1 \le i \le n, 1 \le t \le T}$ are independent and identically uniformly distribution random variables, and $\mathbb{I}$ refers to the indicator function, and $S_i(t)$ and $I_i(t)$ represent the probability that node $i$ transitions from the infectious state to the susceptible state and the probability that node $i$ transitions from the susceptible state to the infectious state at time $t$, respectively.

When a disease begins to spread, the public health organizations typically record the infection status of the population, which serves as the initial condition for the SIS model. Through the model, our goal is to obtain reliable estimates of the information of infection rate and recovery rate for the concerning disease to assess diseases severity and control the spreading. Thus, the hyperparameters $\rvh^\star_{\gamma, \text{dyn}}$ and $\rvh^\star_{\lambda, \text{dyn}}$ are the primary objects of interest, and we denote $\rvh^\star = (\rvh^\star_{\gamma, \text{dyn}}, \rvh^\star_{\lambda, \text{dyn}}) \in \sR_+^2$.

The observation for SIS model is commonly defined as the proportion of infected individuals at time $t$: 
\begin{equation*}
    \rvy_t = \frac{1}{n} \sum\limits_{i=1}^n \rvxi_i(t), \text{ for } t= 1,\dots, T.
\end{equation*} 
The form of this observation satisfies the observation operator Assumption \ref{assum:obsOperator} of our framework. Therefore, our theoretical results are directly applicable to this estimation problem when combined with an optimization algorithm that minimizes the corresponding objective function.

To demonstrate the theoretical guarantee for hyperparameter estimation problem in the SIS model, we employ the EnKF algorithm as the estimation algorithm. We construct SIS networks of different sizes, with the average in-degree of nodes being $10$. The true infection rates are drawn from an exponential distribution with hyperparameter $\rvh^\star_{\gamma, \text{dyn}} = 0.0015$ and the true recovery rates follow an exponential distribution with hyperparameter $\rvh^\star_{\lambda, \text{dyn}} = 0.0025$. Since the infection rate exhibits more complex dynamics than the recovery rate in the SIS dynamical systems, we focus on estimating the infection rate hyperparameter $\rvh^\star_{\gamma, \text{dyn}}$, and assume the recovery rate hyperparameter is known in the following experiments. Thus, the hyperparameter to be estimated reduces to $\rvh^\star = \rvh^\star_{\gamma, \text{dyn}} \in \sR_+$.

We first examine the identifiability of this estimation problem. After constructing an SIS model with $1000$ nodes, we perturb the infection rate hyperparameter around its true value $\rvh^\star$. For each perturbed hyperparameter, we simulate the SIS dynamics 10 times with different random seeds to obtain statistical information of observations. The deviation from the true observation is quantified using the RRMSE (Relative Root Mean Square Error) for every perturbed model, 
\begin{equation}\label{eq:rrmse}
    \mathrm{RRMSE}(\rmY(\tilde{\rvh})) = \sqrt{\frac{1}{T} \left\|\frac{\rmY(\tilde{\rvh}) - \rmY^\star}{\rmY^\star}\right\|^2},
\end{equation}
where $\rmY^\star$ denotes the observation generated using the true hyperparameter $\rvh^\star$, and $\rmY(\rvh)$ denotes the observation generated using the perturbed hyperparameter $\rvh$. This value indicates the deviation of observations generated with perturbed hyperparameter from the observation generated with the true hyperparameter. The mean and standard deviation of the RRMSE values are shown in Fig. \ref{fig:rmse_observation_deviation_SIS}. The perturbation level is measured by the relative noise amplitude : $ \left|\frac{\tilde{\rvh} - \rvh^\star}{\rvh^\star}\right|$. The results indicate that the observations with the perturbed hyperparameter $\tilde{\rvh}$ coincide with the true observation only when $\tilde{\rvh} = \rvh^\star$, which confirms that the hyperparameter estimation problem for the SIS model is identifiable.

\begin{figure}[!t]
    \centering
    \includegraphics[width=0.45\textwidth]{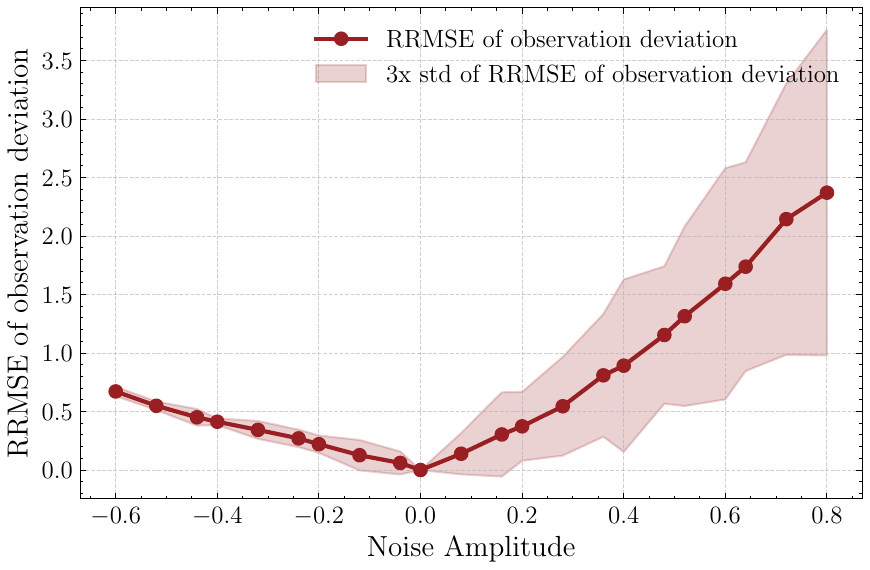}
    \caption{The RRMSE of the observation $\rmY(\tilde{\rvh})$ with perturbed hyperparameter $\tilde{\rvh}$ and the true observation $\rmY^\star$ with the preset true hyperparameter $\rvh^\star$ for different perturbation noise amplitude, for SIS model experiments.}
    \label{fig:rmse_observation_deviation_SIS}
\end{figure}

We constructed SIS models with 21 different network population sizes: 400, 600, 800, 1000, 1200, 1400, 1600, 1800, 2000, 2200, 2400, 2600, 2800, 3000, 3200, 3400, 5000, 6000, 7000, 8000, and 9000. The sampling of sizes is denser at the lower end of this range to closely observe the initial, rapid phase of convergence predicted by our theory. For larger networks, where the estimator saturates, a sparser sampling is sufficient.

\begin{figure}[H]
    \centering
    \includegraphics[width=0.45\textwidth]{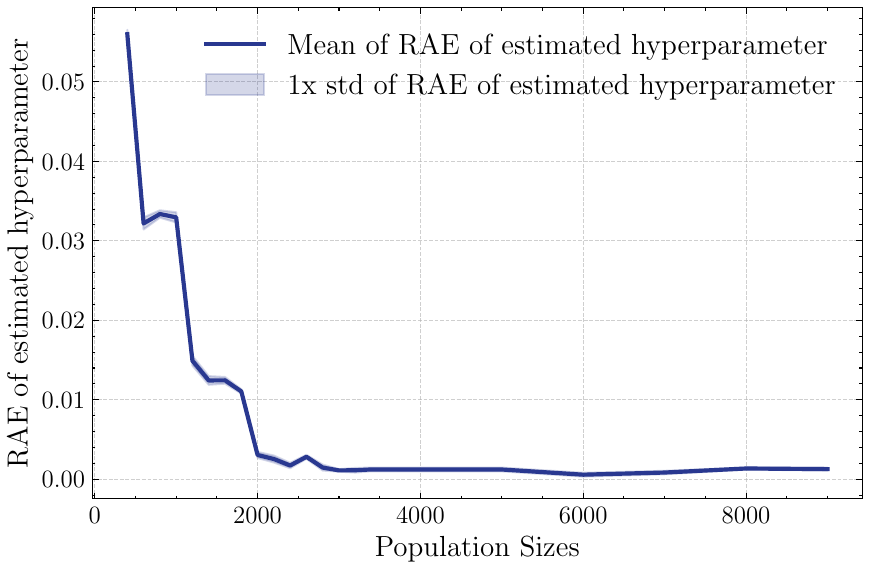}
    \caption{Hyperparameter estimation results for different network population sizes for SIS model experiments. The RAE is computed for the last 100 steps of the estimation process representing the final converged performance.}
    \label{fig:hp_estimation_size_SIS}
\end{figure}

Then, based on the observation sequence, we estimate the hyperparameter $\rvh^\star$ at each time $t$ using the EnKF \cite{evensen2002sequential, law2015data} algorithm with the generated observation $\rvy_t(\hat{\rvh}_n(t))$ with the estimation $\hat{\rvh}_n(t)$. This yields a sequence of hyperparameter estimates $\left\{\hat{\rvh}_n(1), \hat{\rvh}_n(2), \dots, \hat{\rvh}_n(t)\right\}$. To evaluate the performance of the estimation algorithm, the estimated hyperparameter $\hat{\rvh}_n(t)$ is compared with the preset hyperparameter $\rvh^\star$ with the RAE (Relative Absolute Error): 
\begin{equation}\label{eq:rae_curve}
    \mathrm{RAE}(\hat{\rvh}_n(t)) = \left|\frac{\hat{\rvh}_n(t) - \rvh^\star}{\rvh^\star}\right|.
\end{equation}

Each data assimilation task is run for 1500 time steps, with each step corresponding to a new observation. The RAE curves of the estimated hyperparameters for different network population sizes are shown in Fig. \ref{fig:hp_estimation_traces_SIS}. The results indicate that the estimates tend to converge after approximately 1200 steps. Moreover, the final RAE values decrease as the network population size increases, thereby validating our theoretical framework in a practical setting.

To provide a more intuitive measure of convergence results, we compute the RAE of the mean over the last 100 time steps for the estimated hyperparameter, which reflects the convergence performance of the estimation:
\begin{equation}\label{eq:rae_final}
    \mathrm{RAE}(\bar{\rvh}) = \left|\frac{\bar{\rvh} - \rvh^\star}{\rvh^\star}\right| \text{ with } \bar{\rvh} = \frac{1}{100} \sum\limits_{t=T-100}^{T} \hat{\rvh}_n(t).
\end{equation}
We show these RAE values for SIS models with different network population sizes in Fig. \ref{fig:hp_estimation_size_SIS}. Our theoretical results guarantee that the estimated hyperparameter $\hat{\rvh}$ will converge to the true hyperparameter $\rvh^\star$ with probability increasing to 1 as the network population size increasing. In Fig. \ref{fig:hp_estimation_size_SNN}, the RAE decreases with the network population size increasing, which coincides with our theoretical result.

\begin{figure}[H]
    \centering
    \subfloat[\label{fig:hp_estimation_traces_SIS:a1}]{\includegraphics[width=0.22\textwidth]{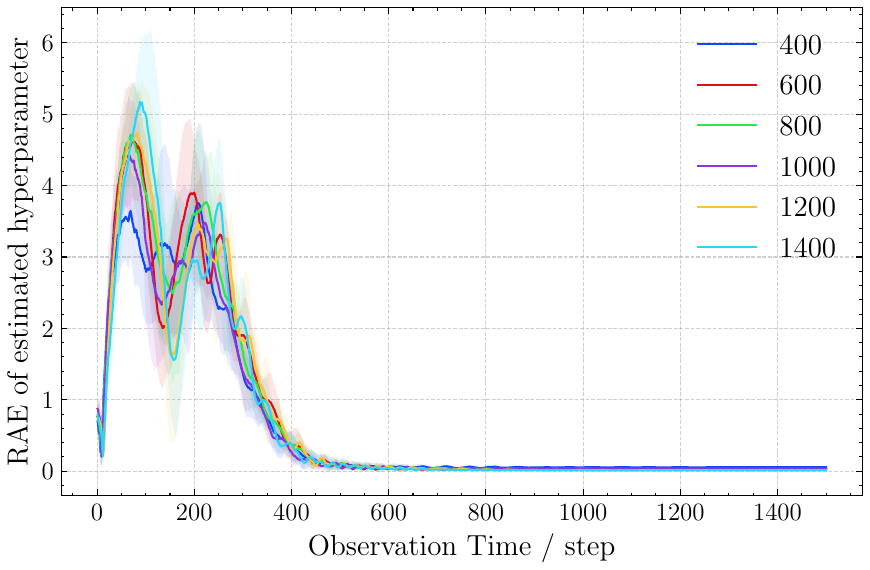}}
    \subfloat[\label{fig:hp_estimation_traces_SIS:b1}]{\includegraphics[width=0.22\textwidth]{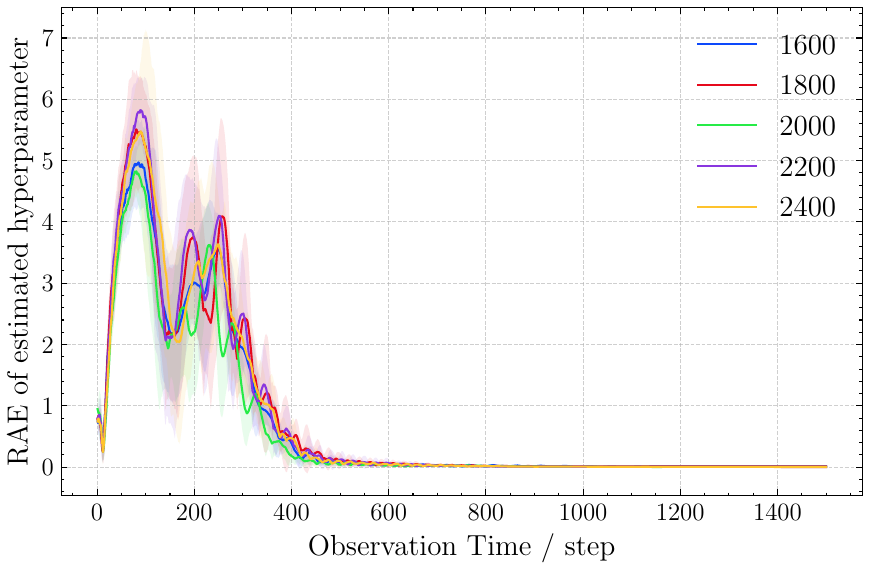}}
    \subfloat[\label{fig:hp_estimation_traces_SIS:c1}]{\includegraphics[width=0.22\textwidth]{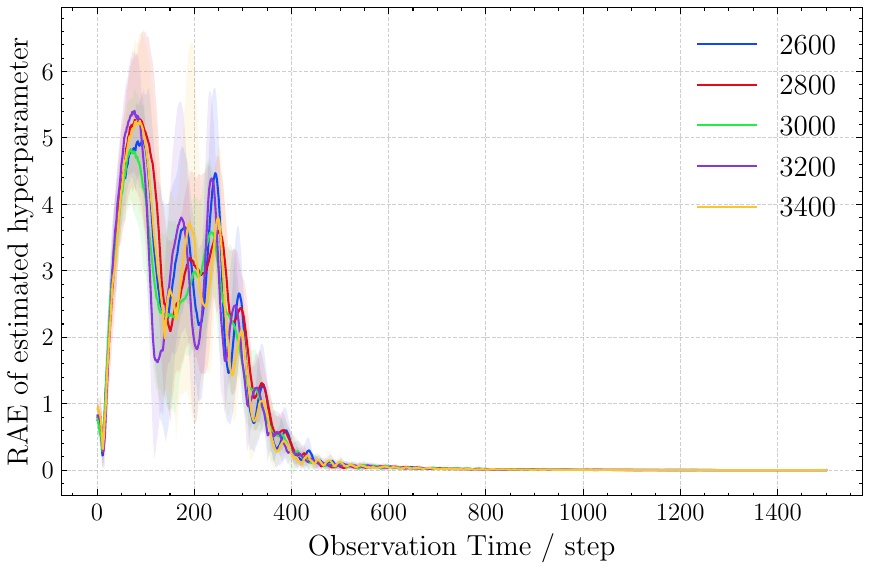}}
    \subfloat[\label{fig:hp_estimation_traces_SIS:d1}]{\includegraphics[width=0.22\textwidth]{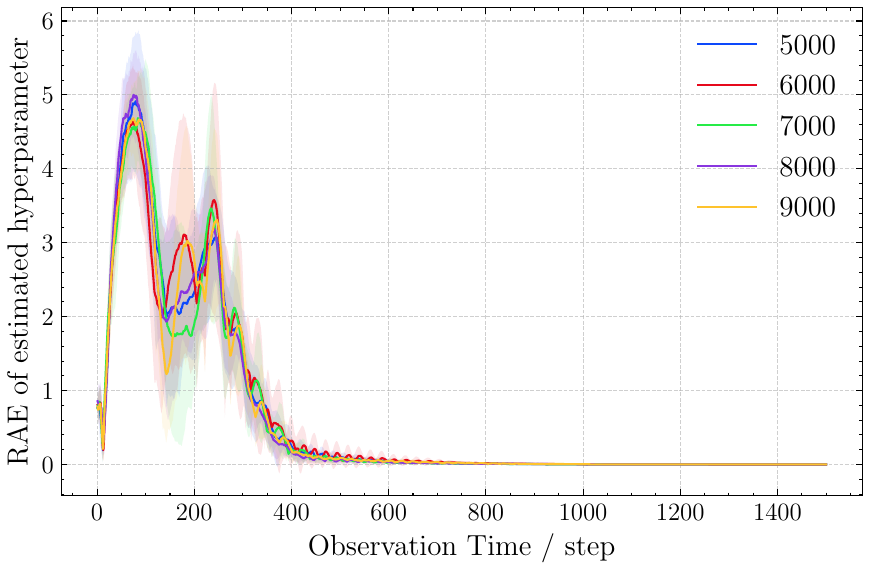}}\\
    \subfloat[\label{fig:hp_estimation_traces_SIS:a2}]{\includegraphics[width=0.22\textwidth]{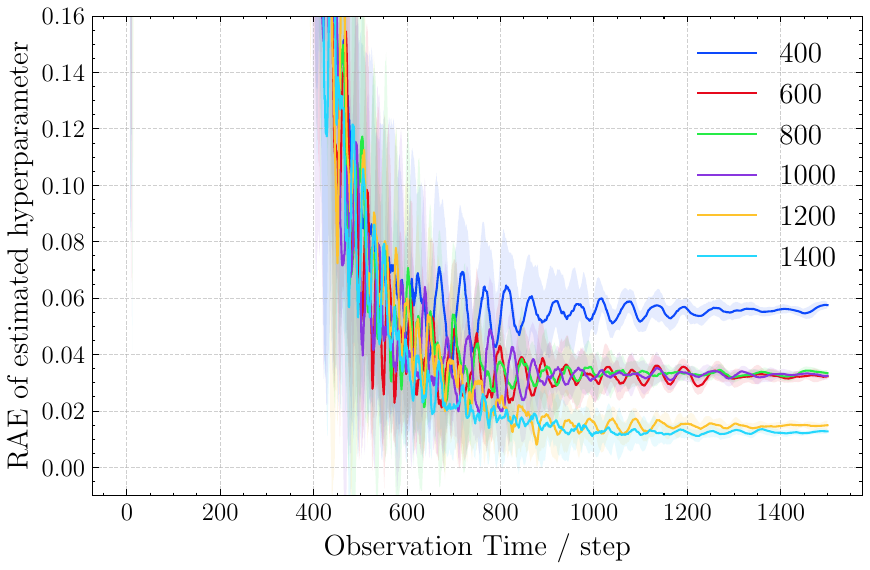}}
    \subfloat[\label{fig:hp_estimation_traces_SIS:b2}]{\includegraphics[width=0.22\textwidth]{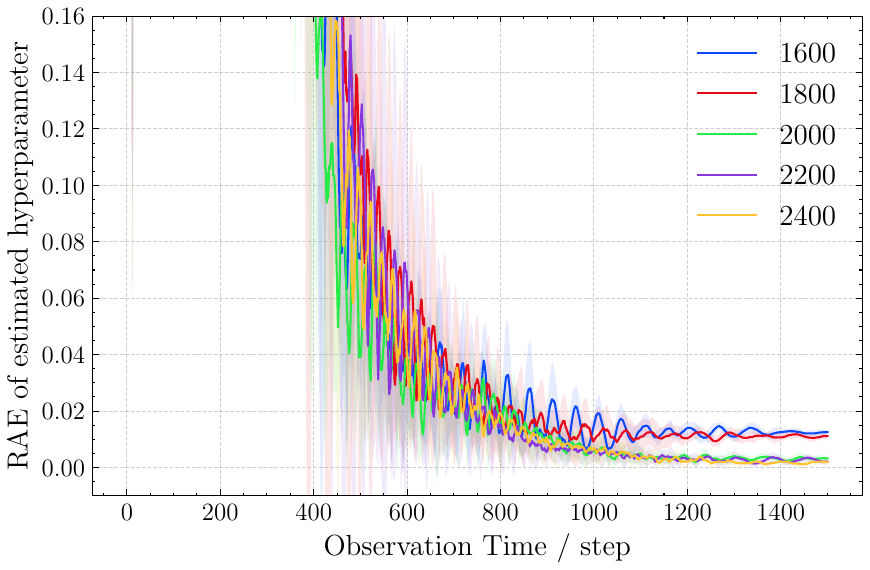}}
    \subfloat[\label{fig:hp_estimation_traces_SIS:c2}]{\includegraphics[width=0.22\textwidth]{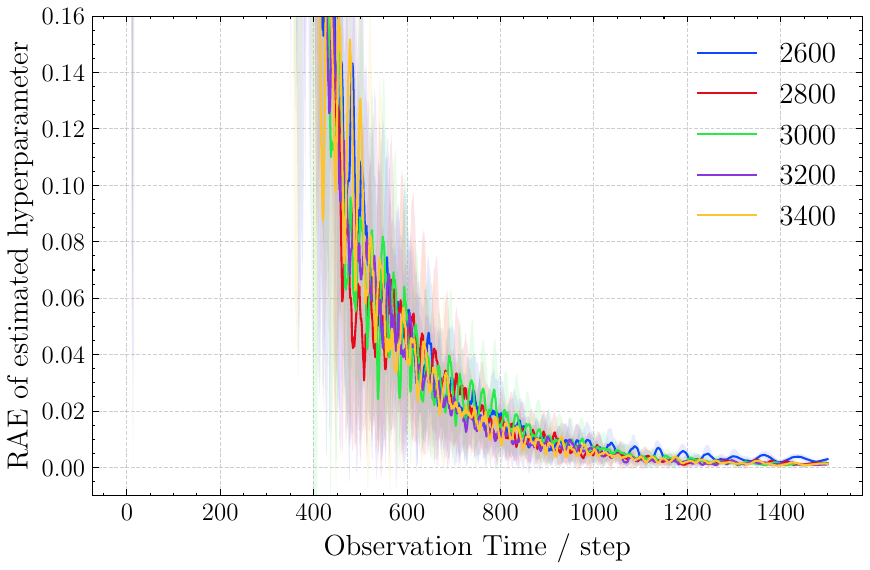}}
    \subfloat[\label{fig:hp_estimation_traces_SIS:d2}]{\includegraphics[width=0.22\textwidth]{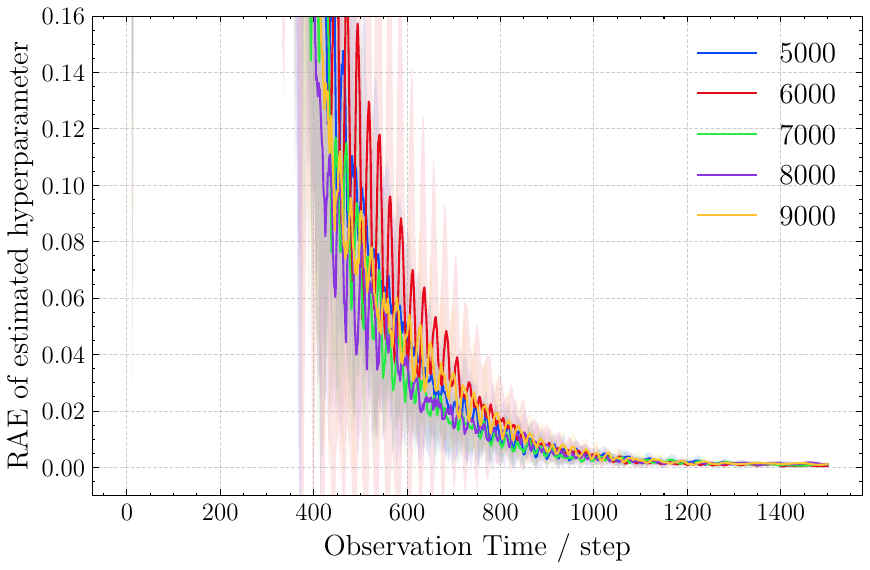}}
    \caption{The RAE curves of hyperparameter estimation process for different network population sizes for SIS model experiments. The top 4 figures (a, b, c, d) show the whole RAE curves of estimation process for 4 groups of experiments with different network population sizes of ([400, 600, 800, 1000, 1200, 1400], [1600, 1800, 2000, 2200, 2400], [2600, 2800, 3000, 3200, 3400], [5000, 6000, 7000, 8000, 9000]), respectively. The bottom 4 figures (e, f, g, h) are the zoom-in illustration of the top 4 figures (a, b, c, d), respectively.}

    \label{fig:hp_estimation_traces_SIS}
\end{figure}

\subsection{Data Assimilation in Spike Neuronal Networks}\label{sec:SNN}

In brain neuronal networks, the synaptic conductances of neurons can be modeled as random variables governed by hyperparameters. The leaky integrate-and-fire (LIF) \cite{burkitt2006review, stein1965theoretical} model serves as the canonical framework for describing the spiking neuronal network model and is widely used to capture the dynamical behavior of large-scale neuronal networks. A variety of neuroimaging signals can be employed as observations for hyperparameter estimation, including BOLD \cite{deco2012ongoing, logothetis2004interpreting}, fMRI \cite{deco2012ongoing, logothetis2008we, heeger2002does}, MEG \cite{david2004evaluation, da2013eeg}, EEG \cite{buzsaki2012origin} and LFP \cite{peyrache2012spatiotemporal, buzsaki2012origin, linden2011modeling} and so on.

The Digital Twin Brain (DTB) model \cite{lu2024imitating, lu2024simulation} is a computational framework that emulates brain dynamics using spiking neuronal network composed of LIF neurons, which is governed by the following dynamics:
\begin{equation*}
    \left\{\begin{matrix}
        \begin{aligned}
        C_i \frac{\df V_i(t)}{\df t} =& -g_L (V_i(t) - V_L) + I_{syn,i}(t) + I_{bg,i}(t) + I_{ext,i}(t), \quad V_i(t) < V_{th}, \\
        I_{syn,i}(t) =& \sum_{u} I_{u,i}(t) = \sum_{u} (g_{u,i}+g_{u,i}^{\mathrm{sin}}(t)) (V_u - V_i(t))J_{u,i}(t), \\
        g_{u,i} \sim& p(g \mid \rvh_{\mathrm{dyn}}^\star), \\
        g_{u,i}^{\mathrm{sin}}(t) =& g_0 \sin(\pi \omega t), \\
        \tau_{bg} \df I_{bg,i}(t) =& (\mu_{bg} -I_{bg,i}(t)) \df t + \sqrt{2\tau_{bg}} \sigma_{bg} \df W(t), \\
        \frac{\df J_{u,i}(t)}{\df t} =& -\frac{J_{u,i}(t)}{\tau_{i}^{u}} + \sum_{k,j} w_{ij}^{u} \delta(t - t_{k}^{j}), \\
        V_i(t) =& V_{rest}, \quad t\in (t_{k}^{i}, t_{k}^{i}+T_{ref}].  
        \end{aligned}
    \end{matrix}\right.
\end{equation*}
In the neuronal network dynamical system described above, $C_i$ represents the neuron membrane capacitance of node $i$, $g_L$ denotes the leakage conductance, $V_L(t)$ is the leakage voltage, and $I_{syn,i}(t)$ refers to the total synaptic current of neuron $i$ at time $t$, defined as the sum of currents from four synapse types. Specifically, $I_{u,i}(t)$ represents the synaptic current for synapse type $u$ of neuron $i$. Here we considered four synapse types: AMPA, NMDA, $\mathrm{GABA_{A}}$ and $\mathrm{{GABA}_B}$. $g_{u,i}$ denotes the conductance of synapse type $u$ for neuron $i$. $V_u$ is the reversal potential of synapse type $u$. The parameters $\mu_{bg}$ and $\sigma_{bg}$ represents the mean and standard deviation of the background current, respectively, and $\tau_{bg}$ denotes the time-scale constant of the background current while $W(t)$ is the standard Wiener process. And $\tau_{i}^{u}$ is the time-scale constant of synapse type $u$ for neuron $i$. $w_{ij}^{u}$ is the connection weights from neuron $j$ to neuron $i$ for synapse type $u$, which are drawn from the uniform distribution $\mathrm{Uniform}(0,1)$. $\delta(\cdot)$ is the Dirac-delta function, and $t_{k}^{j}$ is the spike time of neuron $j$ at the $k$-th spike. The background current $I_{bg,i}(t)$ is modeled as an independent Ornstein-Uhlenbeck (OU) process, which is defined as
\begin{equation*}
    \tau_{b g} d I_{b g, i}=\left(\mu_{b g}-I_{b g, i}\right) \mathrm{dt}+\sqrt{2 \tau_{b g}} \sigma_{b g} d W_{t}.
\end{equation*}
The external stimulus $I_{ext,i}(t)$ is applied exclusively to excitatory neurons. When the membrane potential $V_i$ reaches the threshold $V_{th}$, a spike is generated for neuron $i$, and the membrane potential is reset to the resting potential $V_{rest}$ for a refractory period $T_{ref}$. The synaptic current is then updated based on the spikes of other neurons. The dynamics of the neuron model continue to evolve based on these updates, allowing for the simulation of complex neural interactions. 

The synaptic conductance $g_{u,i}$ governs the strength of current transmission between internal neurons, which reflects the structural and functional properties of a certain brain region. Therefore, our objective is to estimate this conductance. Within the hierarchical Bayesian model in this paper, we model the conductance from NMDA channel, $g_{\mathrm{NMDA},i}$, as random variables following a Gamma distribution. For notational simplicity, here we still denote the conductance $\{g_{\mathrm{NMDA},i}\}$ as $\{g_{u,i}\}$ when no ambiguity arises:
\begin{equation*}
    g_{u,i} \sim \mathrm{Gamma}(\alpha, \beta), \text{ with } \beta = \alpha / \rvh_{\mathrm{dyn}},
\end{equation*}
where we set $\alpha = 5$ based on empirical considerations. Then, we can estimate the conductance $g_{u,i}$ through the hyperparameter $\rvh_{\mathrm{dyn}}^\star$ by observations. Here, we choose the LFP (Local Field Potential) as the observation, which is defined as the average membrane voltage across all neurons, that is
\begin{equation*}
    \rvy(t) = \frac{1}{n} \sum\limits_{i=1}^n V_i(t), \text{ for } t= 1,\dots, T.
\end{equation*} 

In our network setup, the excitatory and inhibitory neurons are in a ratio of $4:1$. Each neuron has connections of in-degree $D = 10$, with excitatory and inhibitory connections also maintaining a ratio of $4:1$. 

Based on the constructed spike neural network, we preset the hyperparameter to be estimated as $\rvh^\star = \rvh_{\mathrm{dyn}}^\star = 4.86 \times 10^{-6}$. The conductance from NMDA channel, $g_{u,i}$, are drawn from a Gamma distribution with shape parameter $\alpha = 5$ and inverse scale parameter $\beta = 5/ \rvh^\star$.

And the settings for other arguments in the neuron model are shown in Table. \ref{tab:default_parameters}.
\begin{table}[!ht]
  \centering
  \caption{{\bf Default parameters in the SNN model}}
  \label{tab:default_parameters}
  \begin{tabular}{ccccccccccc}
    \toprule
    {\bf Symbol} & {\bf Value}   \\
    \toprule
    $C_{i}$ & $1 \mathrm{\mu f} $  \\
    \midrule
    $g_{L}$ & $0.03 \mathrm{mS}$ \\
    \midrule
    $V_{L}$ & $-75 \mathrm{mV}$ \\
    \midrule
    $V_{th}$ & $-50 \mathrm{mV}$ \\
    \midrule
    $V_{rest}$ & $-65 \mathrm{mV}$ \\
    \midrule
    $\left [ V_{\mathrm{AMPA}, i}, V_{\mathrm{NMDA},i}, V_{\mathrm{GABA_A},i}, V_{\mathrm{GABA_B},i} \right ]$ & $\left [ 0, 0, -70, -100 \right ] \mathrm{mV}$ \\
    \midrule
    $\left [ g_{\mathrm{AMPA}, i}, g_{\mathrm{GABA_{A}},i}, g_{\mathrm{GABA_{B}},i} \right ]$ & $\left [ 1.6 \times 10^{-5}, 3.672 \times 10^{-5}, 7.56 \times 10^{-6} \right ]$ \\
    \midrule
    $\left [ \tau_{\mathrm{AMPA}, i}, \tau_{\mathrm{NMDA},i}, \tau_{\mathrm{GABA_A},i}, \tau_{\mathrm{GABA_B},i} \right ]$ & $\left [ 2, 40, 10, 50 \right ]$ \\
    \midrule
    $w_{i,j}^{u}$ &  $\sim \mathrm{Uniform}(0,1)$ \\
    \midrule
    $g_0$ & $4.86 \times 10^{-6}$ \\
    \midrule 
    $\omega$ & $0.02$ \\
    \midrule
    $T_{ref}$ & $5 \mathrm{ms}$ \\
    \midrule
    $\mu_{bg}$ & $0.71 \mathrm{nA}$ \\
    \midrule
    $\sigma_{bg}$ & $0.05 \mathrm{nA}$ \\
    \midrule
    $\tau_{bg}$ & $10 \mathrm{ms}$ \\
    \midrule
    $I_{ext,i}$ & $0 \mathrm{mA}$ \\
    \bottomrule
  \end{tabular}
\end{table}

To evaluate how our theoretical results support assimilation tasks in spiking neural networks, we follow a procedure analogous to that described in Section \ref{sec:SIS}. 

Specifically, we first verify that the observation used for assimilation is injective with respect to the hyperparameter, thereby ensuring that the deterministic objective function is identifiable. To this end, we constructed a spiking neuron network with 1000 nodes, and perturbed the preset hyperparameter $\rvh^\star$ by adding Gaussian noise with different noise amplitude, and then sample the conductances according to the distribution determined by each perturbed hyperparameter. For each perturbation level, we repeatedly generating the neural network 100 times with different random seeds to reduce stochastic effects and obtain the corresponding observation sequence under various hyperparameters. Then, we compute the RRMSE, as defined in \eqref{eq:rrmse}, between the observation sequence under perturbed hyperparameter and the observation sequence under preset true hyperparameter with respect to duration $T$, which quantifies the deviation between the observation sequence under perturbed hyperparameter and the observation sequence under preset true hyperparameter. The results are shown in Fig. \ref{fig:rmse_observation_deviation_SNN}, which indicates that the deviation is nonzero whenever the perturbed hyperparameter differs from the true hyperparameter. This demonstrates that the observation sequence is injective with respect to the hyperparameter, thereby confirming that the estimation problem is identifiable.

\begin{figure}[!t]
    \centering
    \includegraphics[width=0.5\textwidth]{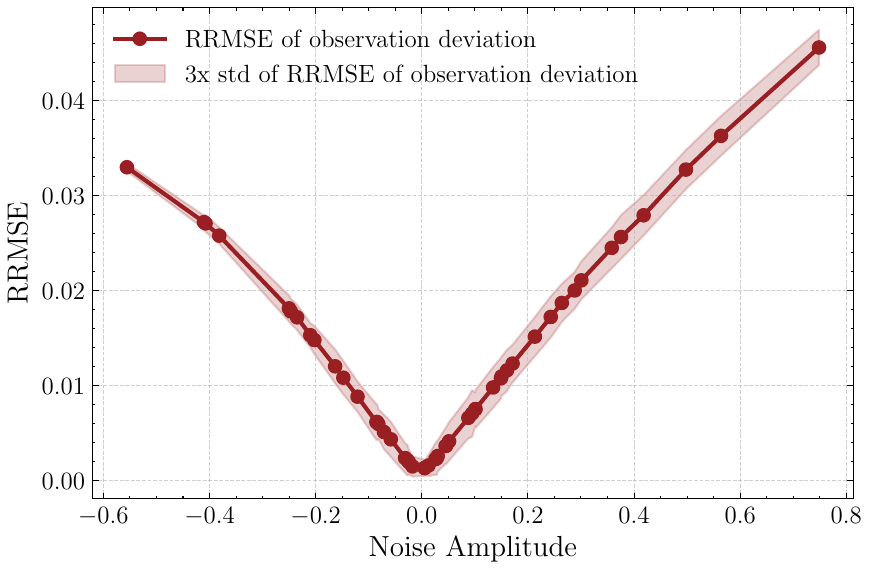}
    \caption{The RRMSE of the observation $\rmY(\tilde{\rvh})$ with perturbed hyperparameter $\tilde{\rvh}$ and the true observation $\rmY^\star$ with the preset true hyperparameter $\rvh^\star$ for different perturbation noise amplitude, for SNN model experiments.}
    \label{fig:rmse_observation_deviation_SNN}
\end{figure}

Analogous to the analysis in Section \ref{sec:SIS}, we constructed SNN models with 14 different network population sizes: 400, 600, 800, 1000, 1200, 1600, 2000, 2400, 2800, 3200, 4000, 5000, 8000 and 10000. Again, the sampling of sizes is denser for smaller networks to closely observe the initial, rapid phase of convergence, and sparser for larger networks where the estimator saturates.

Under the preset true hyperparameter $\rvh^\star$, we simulate the spiking neural network and record the LFP at a temporal resolution of $1$ ms. This yields the observation sequence $\rmY^\star = \left\{\rvy_1^\star, \dots, \rvy_T^\star\right\}$ over a duration of  $T=4000$ ms. Following the procedure outlined in Section \ref{sec:SIS}, we estimate the hyperparameter $\rvh^\star$ sequentially at each time step $t$ using the EnKF algorithm. This yields a sequence of hyperparameter estimates $\left\{\hat{\rvh}_n(1), \hat{\rvh}_n(2), \dots, \hat{\rvh}_n(t)\right\}$. We evaluate the accuracy of the estimates by computing the RAE, as defined in \eqref{eq:rae_curve}. To reduce the stochastic effects, we repeat each assimilation task 100 times with different random seeds. The results are shown in Fig. \ref{fig:hp_estimation_traces_SNN}. The results indicate that the estimates converge after approximately 2500 steps and that the final RAE decreases as the number of neurons increases, thereby validating our theoretical predictions in a practical setting.

\begin{figure}[!t]
    \centering
    \subfloat[\label{fig:hp_estimation_traces_SNN:a1}]{\includegraphics[width=0.3\textwidth]{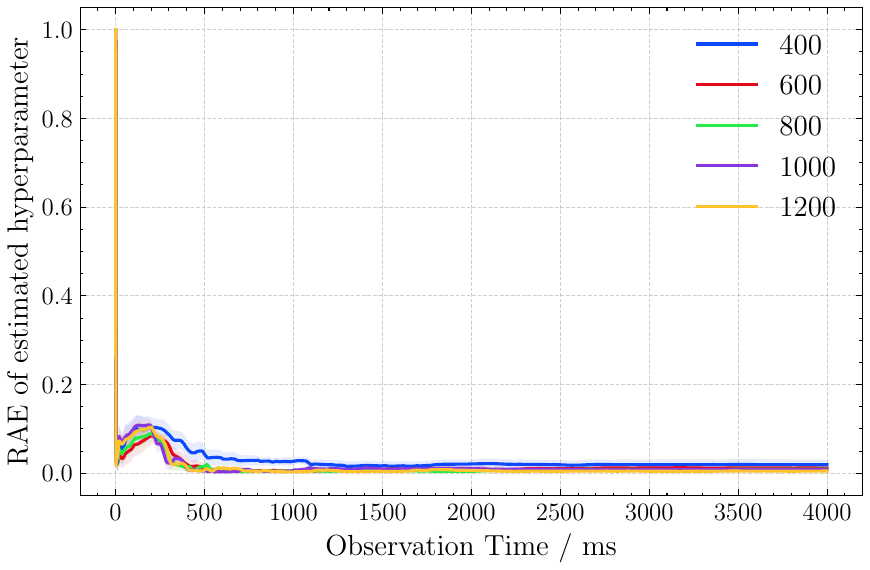}}
    \subfloat[\label{fig:hp_estimation_traces_SNN:b1}]{\includegraphics[width=0.3\textwidth]{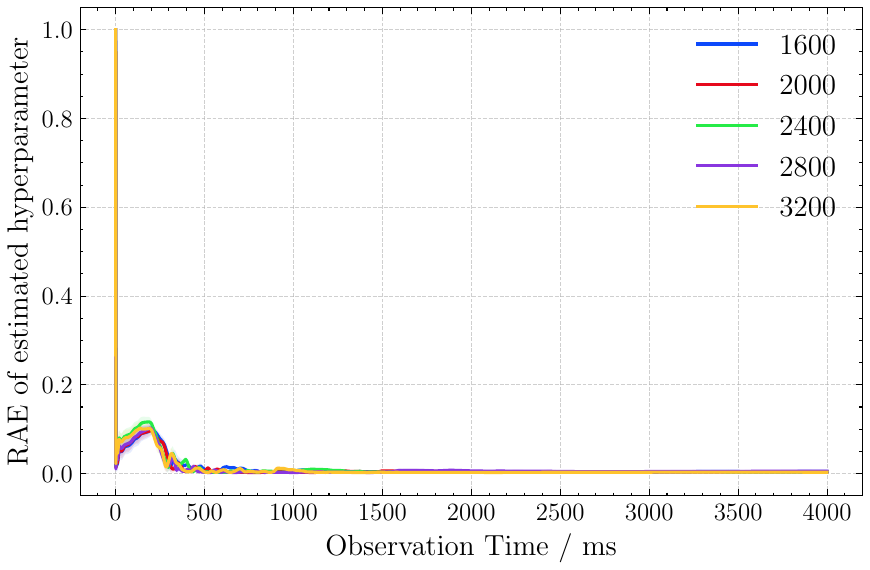}}
    \subfloat[\label{fig:hp_estimation_traces_SNN:c1}]{\includegraphics[width=0.3\textwidth]{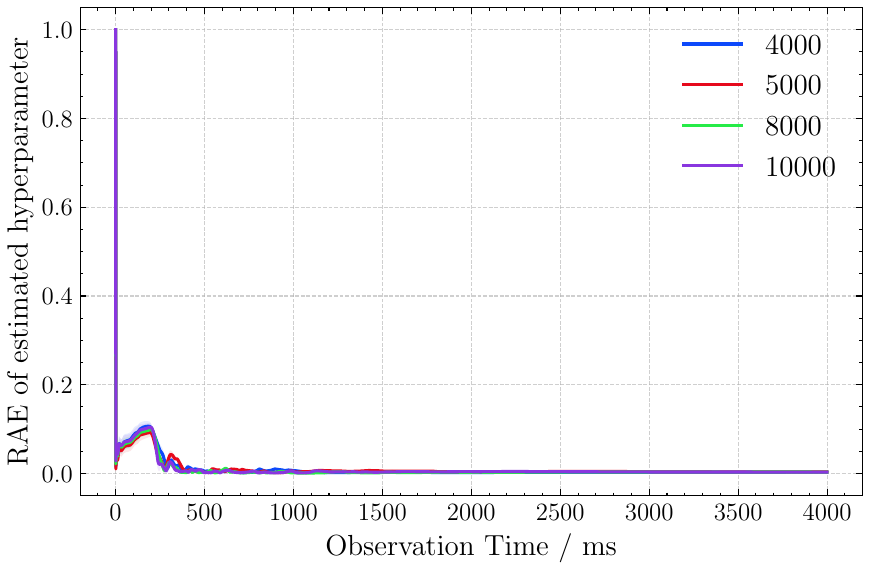}}\\
    \subfloat[\label{fig:hp_estimation_traces_SNN:a2}]{\includegraphics[width=0.3\textwidth]{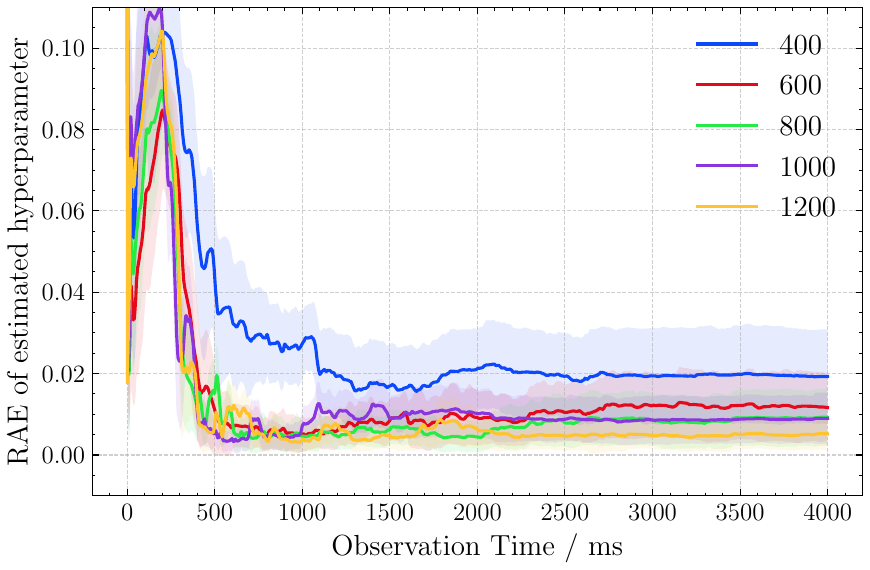}}
    \subfloat[\label{fig:hp_estimation_traces_SNN:b2}]{\includegraphics[width=0.3\textwidth]{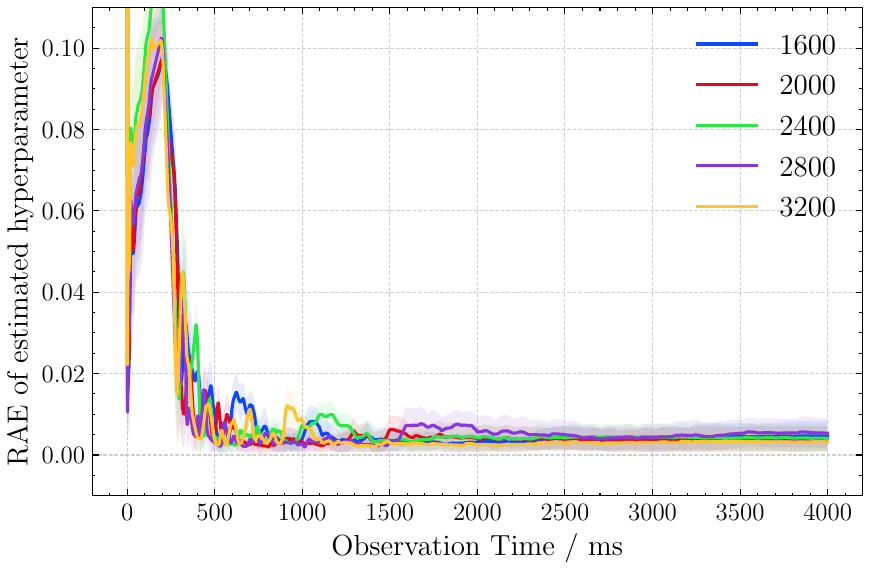}}
    \subfloat[\label{fig:hp_estimation_traces_SNN:c2}]{\includegraphics[width=0.3\textwidth]{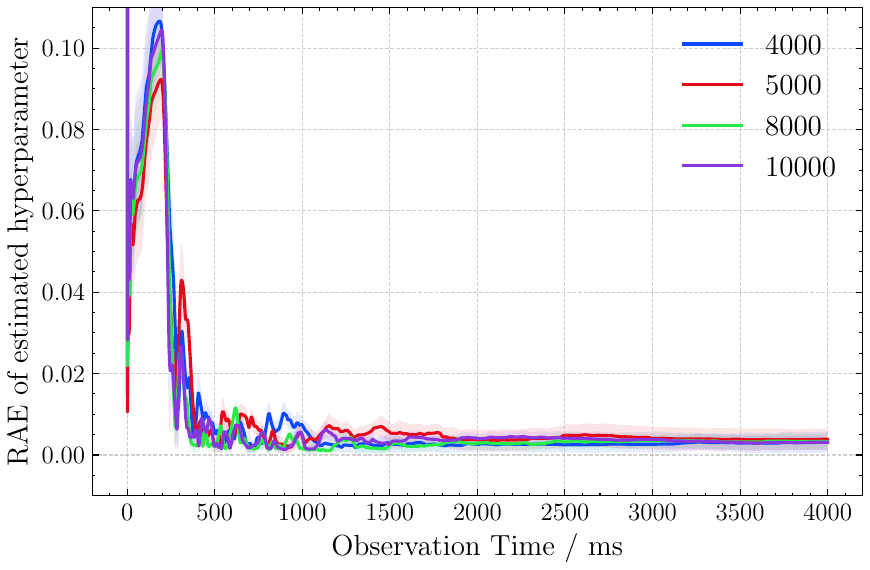}}
    \caption{The RAE curves of hyperparameter estimation process for different network population sizes for SNN experiments. The top 3 figures (a, b, c) show the whole RAE curves of estimation process for 3 groups of experiments with different network population sizes of ([400, 600, 800, 1000, 1200], [1600, 2000, 2400, 2800, 3200], [4000, 5000, 8000, 10000]), respectively. The bottom 3 figures (d, e, f) are the zoom-in illustration of the top 3 figures (a, b, c), respectively.}

    \label{fig:hp_estimation_traces_SNN}
\end{figure}

To provide a more interpretable measure of convergence, we compute the RAE of the mean estimate over the final 100 steps, as defined in \eqref{eq:rae_final}. We show these RAE values of SNN models with different network population sizes in Fig. \ref{fig:hp_estimation_size_SNN}. Our theoretical results state that the estimated hyperparameter $\hat{\rvh}$ will converge to the true hyperparameter $\rvh^\star$ with probability increasing to 1 as the network population size increasing. In Fig. \ref{fig:hp_estimation_size_SNN}, the RAE decreases with the network population size increasing, which supports that our theoretical results hold for assimilation tasks in spiking neural networks.

\begin{figure}[!t]
    \centering
    \includegraphics[width=0.5\textwidth]{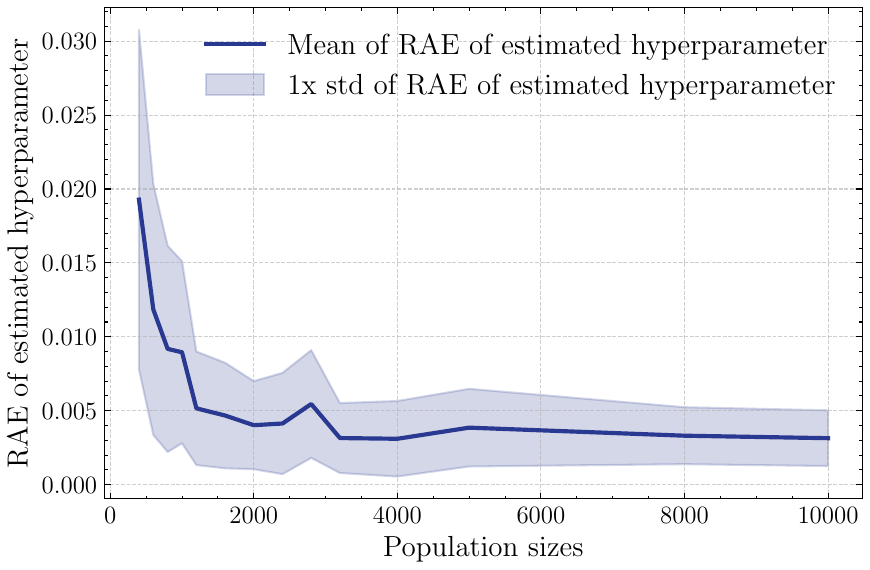}
    \caption{Hyperparameter estimation results for different network population sizes for SNN model experiments. The RAE is computed for the last 100 steps of the estimation process representing the final converged performance.}

    \label{fig:hp_estimation_size_SNN}
\end{figure}

\section{Conclusion}\label{sec:conclusion}
In this work, we have addressed a fundamental challenge in the statistical modeling of large-scale complex network dynamical systems: the reliability of parameter estimation in the face of high dimensionality and data scarcity. While the hierarchical Bayesian framework offers a powerful approach for modeling inhomogeneous systems by estimating hyperparameters that govern parameter distributions, its theoretical validity as the network population size increases has remained an open question.

Our central contribution is the development of a rigorous theoretical framework that establishes the consistency of hyperparameter estimation with respect to the growing network population size. We have formally proven that as the network population size increases, the estimates of hyperparameters, which govern both the system's dynamics and its initial conditions, converge to their true values, even with a limited duration of observational data. By leveraging a measure transport perspective and focusing on common mean-type observations, we established this guarantee for a general class of optimization-based estimators.

Crucially, this consistency result is not confined to simplified systems with independent nodes; we have extended it to encompass systems with weak dependencies, which more accurately reflect the interconnected nature of most real-world phenomena. Our theoretical claims were substantiated through numerical experiments on representative models from epidemiology (the SIS model) and neuroscience (a spiking neuronal network model). In both cases, the empirical estimation error was shown to decrease as the network population size increased, precisely as our theory predicts.

This research fills a significant void in the statistical theory of complex network dynamical systems. It provides the necessary justification for applying hierarchical inference and data assimilation techniques to large-scale, inhomogeneous problems, ensuring that the resulting models are not only computationally tractable but also statistically robust. The findings offer confidence that as we build larger and more complex models of phenomena like brain activity or global pandemics, our inference methods will remain reliable and will not be compromised by the scale of the system itself. By providing this theoretical guarantee, our work paves the way for more ambitious and realistic modeling of complex systems across science and engineering.

\section*{Funding}
    The authors were jointly supported by the STCSM (No. 23JC1400800), the JiHua Laboratory S\&T Program (No. X250881UG250), the Lingang Laboratory, Grant (No. LGL-1987), and the key projects of "Double First-Class" initiative of Fudan University.

\begin{appendix}\label{sec:appendix}

\subsection{Stability and Concentration Tools}
\label{app:preliminaryTools}

\begin{proof}[Proof of Lemma \ref{le:identifiable}]
For $\eps>0$, let
$A_\eps=\{\rvh\in\gH:d(\rvh,\rvh^\star)\ge\eps\}$. If $A_\eps$ is
nonempty, it is compact because $\gH$ is compact. Continuity therefore gives
some $\tilde\rvh\in A_\eps$ such that
\begin{equation*}
\inf_{\rvh\in A_\eps}L^\star(\rvh)=L^\star(\tilde\rvh).
\end{equation*}
Since $\rvh^\star\notin A_\eps$ and is the unique minimizer,
$L^\star(\tilde\rvh)>L^\star(\rvh^\star)$. This is strict
identifiability.
\end{proof}

\begin{proof}[Proof of Lemma \ref{le:SIMarkov}]
Set
\begin{equation*}
\Delta_n=\sup_{\rvh\in\gH}|L_n(\rvh)-L^\star(\rvh)|.
\end{equation*}
For any positive $\eta(\eps)$ not exceeding the strict-identifiability gap,
the event $d(\hat\rvh_n,\rvh^\star)\ge\eps$ implies
\begin{equation*}
L^\star(\hat\rvh_n)-L^\star(\rvh^\star)\ge\eta(\eps).
\end{equation*}
The minimizing property of $\hat\rvh_n$ gives
\begin{equation*}
\begin{aligned}
L^\star(\hat\rvh_n)-L^\star(\rvh^\star)
&\le L^\star(\hat\rvh_n)-L_n(\hat\rvh_n)
 +L_n(\rvh^\star)-L^\star(\rvh^\star)\\
&\le2\Delta_n.
\end{aligned}
\end{equation*}
Consequently,
\begin{equation*}
\{d(\hat\rvh_n,\rvh^\star)\ge\eps\}
\subseteq\{\Delta_n\ge\eta(\eps)/2\}.
\end{equation*}
Markov's inequality proves \eqref{eq:SIMarkov}.
\end{proof}

The remainder of Appendix \ref{app:preliminaryTools} collects the concentration tools
used in the proofs of Theorems \ref{thm:ConsistencyIID} and
\ref{thm:ConsistencyWD}; see also \cite{vershynin2018high}.

\begin{defn}[Sub-Gaussian random variable]\label{sub-GaussianV}
A centered random variable $X$ is sub-Gaussian with variance proxy $v^2$ if
\begin{equation*}
\E e^{uX}\le e^{u^2v^2/2},\qquad u\in\R.
\end{equation*}
\end{defn}

\begin{props}[Sub-Gaussian sum rule]\label{prop:subGaussianSum}
If $X_1,\ldots,X_m$ are sub-Gaussian with variance proxies
$v_1^2,\ldots,v_m^2$, then $\sum_jX_j$ is sub-Gaussian with variance proxy
$(\sum_jv_j)^2$. If the variables are independent, the sharper proxy
$\sum_jv_j^2$ holds.
\end{props}

\begin{defn}[Sub-Gaussian process]\label{sub-GaussianP}
Let $(\sT,d)$ be a metric space. A centered process
$\{Z_t:t\in\sT\}$ is sub-Gaussian if
\begin{equation*}
\E e^{u(Z_t-Z_s)}\le e^{u^2d(t,s)^2/2},
\qquad s,t\in\sT,\quad u\in\R.
\end{equation*}
\end{defn}

\begin{defn}[Separable process]\label{separableP}
A process $\{Z_t:t\in\sT\}$ is separable if there is a countable
$\sT_0\subseteq\sT$ such that, almost surely, every $Z_t$ is a limit of
$Z_s$ along a sequence $s\in\sT_0$ converging to $t$.
\end{defn}

\setcounter{remark}{3}
\begin{remark}
If $\sT$ is separable and $t\mapsto Z_t$ has continuous sample paths, then
$\{Z_t:t\in\sT\}$ is a separable process. The processes used in the main
proofs have this property because their sample paths are Lipschitz on the
compact set $\gH$.
\end{remark}
\setcounter{remark}{5}

\begin{defn}[$\eps$-net and covering number]\label{defn:coveringNumber}
A set $\mathcal N\subseteq\sT$ is an $\eps$-net of $(\sT,d)$ if every
$t\in\sT$ has some $s\in\mathcal N$ with $d(t,s)\le\eps$. The covering
number is
\begin{equation*}
N(\sT,d,\eps)=\inf\{|\mathcal N|:\mathcal N
\text{ is an $\eps$-net of }(\sT,d)\}.
\end{equation*}
\end{defn}

\setcounter{lemma}{3}
\begin{lemma}[Dudley's theorem \cite{dudley1991uniform,dudley1978central}]
\label{le:Dudley}
Let $\{Z_t:t\in\sT\}$ be a separable sub-Gaussian process on $(\sT,d)$.
Then
\begin{equation*}
\E\sup_{t\in\sT}Z_t
\le12\int_0^\infty\sqrt{\log N(\sT,d,\eps)}\,\mathrm d\eps.
\end{equation*}
The same bound applied to $Z$ and $-Z$ controls the expected absolute
supremum, up to a factor of two.
\end{lemma}

\begin{lemma}[Bound for the entropy integral]\label{le:BoundedEntropy}
Let $\sT\subset\R^h$ be bounded under the Euclidean metric $d$, with
$D=\diam(\sT)$. Then
\begin{equation*}
\int_0^\infty\sqrt{\log N(\sT,d,\eps)}\,\mathrm d\eps
\le2\sqrt{3h}\,D.
\end{equation*}
The conclusion is unchanged after a fixed rescaling of the metric.
\end{lemma}

\begin{proof}
The Euclidean volume bound gives
$N(\sT,d,\eps)\le(3D/\eps)^h$ for $0<\eps<D$, whereas the covering number
is one for $\eps\ge D$. Since $\log x\le x$ for $x\ge1$,
\begin{equation*}
\begin{aligned}
\int_0^\infty\sqrt{\log N(\sT,d,\eps)}\,\mathrm d\eps
&\le\int_0^D\sqrt{h\log(3D/\eps)}\,\mathrm d\eps\\
&\le\int_0^D\sqrt{3hD/\eps}\,\mathrm d\eps
=2\sqrt{3h}\,D.
\end{aligned}
\end{equation*}
\end{proof}

\setcounter{lemma}{8}

\begin{defn}[$\beta$-mixing coefficient]\label{defn:betaMixing}
Let $\cA$ and $\cB$ be two $\sigma$-fields on the same probability space
$(\Omega,\cF,\gP)$. Let $\gP_{\cA\otimes\cB}$ denote the joint law induced
by the diagonal map $\omega\mapsto(\omega,\omega)$, and let
$\gP_\cA\otimes\gP_\cB$ be the product of the marginal restrictions. The
coefficient of absolute regularity is
\begin{equation*}
\beta(\cA,\cB)=
\sup_{C\in\cA\otimes\cB}
|\gP_{\cA\otimes\cB}(C)
-(\gP_\cA\otimes\gP_\cB)(C)|.
\end{equation*}
For a process $\{X_j\}_{j\in\sN}$, its uniform coefficient at lag $k$ is
\begin{equation*}
\beta_k=\sup_{j\in\sN}
\beta\left(\sigma(X_j),\sigma\{X_{j+s}:s\ge k\}\right),
\qquad k\ge1.
\end{equation*}
\end{defn}

\subsection{Regularity of the Deterministic Flow and Nodewise Representation}
\label{app:flowTransport}

For $s=(s_1,\ldots,s_n)\in\gS^n$ and
$\theta=(\theta_1,\ldots,\theta_n)\in\Theta$, define the block maximum norms
\begin{equation*}
\|s\|_n\triangleq\max_{1\le i\le n}\|s_i\|_2,
\qquad
\|\theta\|_{\theta,n}\triangleq\max_{1\le i\le n}\|\theta_i\|_2,
\end{equation*}
and let $\|\cdot\|_{\mathrm{op},n}$ denote the operator norm induced by
$\|\cdot\|_n$. In particular, each coordinate projection
$\operatorname{pr}_i:(\gS^n,\|\cdot\|_n)\to(\gS,\|\cdot\|_2)$ is
nonexpansive.

\begin{assum}[Deterministic-flow regularity]\label{assum:flowRegularity}
Write the continuous dynamics in vector form as
\begin{equation*}
\frac{\mathrm d}{\mathrm dt}\rmS_n(t)
=\mathbf g_n(t,\rmS_n(t),\theta).
\end{equation*}
For each fixed $n$, the full initial-state and latent-parameter laws have
densities $p_{0,n}(\cdot\mid\rvh_{\mathrm{prob}})$ and
$q_n(\cdot\mid\rvh_{\mathrm{dyn}})$ on compact supports; in the independent
parameter specification of \eqref{eq:dynSysGenODE},
$q_n=q^n$. The two vectors are independent, every initial node-state marginal
equals $p_0(\cdot\mid\rvh_{\mathrm{prob}})$, and every node-parameter marginal
equals $q(\cdot\mid\rvh_{\mathrm{dyn}})$. No equality is required among the
finite-time node-state marginals. There is a hyperparameter-independent base random variable
$\rmX_n$ generating
$\rmS_n(0;\rmX_n,\rvh_{\mathrm{prob}})$ and
$\theta(\rmX_n;\rvh_{\mathrm{dyn}})$ such that, uniformly in $n$,
\begin{equation}\label{eq:baseGeneratorLipschitz}
\begin{aligned}
&\|\rmS_n(0;\rmX_n,\rvh_{\mathrm{prob}})
-\rmS_n(0;\rmX_n,\rvg_{\mathrm{prob}})\|_n\\
&\quad+
\|\theta(\rmX_n;\rvh_{\mathrm{dyn}})
-\theta(\rmX_n;\rvg_{\mathrm{dyn}})\|_{\theta,n}
\le L_Xd(\rvh,\rvg)
\end{aligned}
\end{equation}
almost surely.

The vector field is measurable in time, $C^1$ and locally Lipschitz in the
state on an open trajectory tube containing the required forward and
backward solutions. On this tube and over $[0,T]$,
\begin{equation*}
\|\mathbf g_n(t,s,\theta)\|_n\le B_T(1+\|s\|_n),
\qquad
\|D_s\mathbf g_n(t,s,\theta)\|_{\mathrm{op},n}\le L_T,
\end{equation*}
and, at the same base random variable,
\begin{equation}\label{eq:vectorFieldParameterBound}
\|\mathbf g_n(t,s,\theta(\rmX_n;\rvh_{\mathrm{dyn}}))
-\mathbf g_n(t,s,\theta(\rmX_n;\rvg_{\mathrm{dyn}}))\|_n
\le L_{\mathrm{par}}d(\rvh,\rvg).
\end{equation}
The constants are independent of $n$. A dimension-free sufficient condition
for the derivative bound in the block maximum norm is
\begin{equation*}
\sup_{n,t,i}\sum_j
\|D_{s_j}g_i(t,s,\theta)\|_{\mathrm{op}}\le L_T.
\end{equation*}
\end{assum}

\begin{props}[ODE flow regularity]\label{prop:flowRegularity}
Under Assumption \ref{assum:flowRegularity}, the operator
$G_{u\leftarrow s}^{\theta}$ is well defined for $0\le s,u\le T$ and is a
$C^1$ diffeomorphism between the corresponding trajectory domains. In
particular, $G_t=G_{t\leftarrow0}^{\theta}$ as in the first relation of \eqref{eq:dynTrans}. Moreover,
\begin{equation}\label{eq:flowGrowth}
\|G_{u\leftarrow s}^{\theta}(x)\|_n
\le(\|x\|_n+1)e^{B_T|u-s|}-1,
\end{equation}
and the common-base coupling obeys
\begin{equation}\label{eq:flowHyperparameterLipschitz}
\sup_{0\le t\le T}
\|\rmS_n(t;\rvh)-\rmS_n(t;\rvg)\|_n
\le
\left\{e^{L_TT}L_X+
L_{\mathrm{par}}\frac{e^{L_TT}-1}{L_T}\right\}d(\rvh,\rvg),
\end{equation}
where the fraction is interpreted as $T$ when $L_T=0$. The flow Jacobian is
\begin{equation}\label{eq:flowJacobian}
\det D_sG_{t\leftarrow0}^{\theta}(s)
=\exp\left\{\int_0^t
\operatorname{tr}D_s\mathbf g_n
(u,G_{u\leftarrow0}^{\theta}(s),\theta)\,\mathrm du\right\},
\end{equation}
and is finite and nonzero. Consequently,
\begin{equation}\label{eq:flowDensity}
\begin{aligned}
p_{t,n}(s,\theta\mid\rvh)
&=p_{0,n}\left(G_{0\leftarrow t}^{\theta}(s)
\mid\rvh_{\mathrm{prob}}\right)
\left|\det D_sG_{0\leftarrow t}^{\theta}(s)\right|\\
&\quad\times q_n(\theta\mid\rvh_{\mathrm{dyn}}).
\end{aligned}
\end{equation}
\end{props}

\begin{proof}
Fix $n$ and $\theta$. The Carath\'eodory conditions give a unique local
solution. If $\rmS_n(u)=G_{u\leftarrow s}^{\theta}(x)$, the linear-growth
bound gives
\begin{equation*}
1+\|\rmS_n(u)\|_n
\le1+\|x\|_n+B_T\int_s^u\{1+\|\rmS_n(v)\|_n\}\,\mathrm dv.
\end{equation*}
Gronwall's inequality proves \eqref{eq:flowGrowth}, precluding finite-horizon
norm blow-up in either time direction. The same argument applied to two
solutions generated by a common base random variable, using
\eqref{eq:baseGeneratorLipschitz} and
\eqref{eq:vectorFieldParameterBound}, proves
\eqref{eq:flowHyperparameterLipschitz}.

The backward solution is the inverse of the forward solution. Differentiable
dependence on the initial state therefore makes the finite-time flow a
$C^1$ diffeomorphism. Its Jacobian $J(t)$ satisfies
\begin{equation*}
\dot J(t)=D_s\mathbf g_n(t,G_{t\leftarrow0}^{\theta}(s),\theta)J(t),
\qquad J(0)=I.
\end{equation*}
Liouville's formula yields \eqref{eq:flowJacobian}. Conditional on $\theta$,
the ordinary change-of-variables theorem applied to the diffeomorphism
$s_0\mapsto G_{t\leftarrow0}^{\theta}(s_0)$ gives
\eqref{eq:flowDensity}; projection and integration over the remaining state
coordinates and $\theta$ give the node marginal density in
\eqref{eq:nodeMarginalDensity}.
\end{proof}

\begin{props}[Nodewise single-base representation of time-indexed simulator responses]
\label{prop:nodewiseBaseRepresentation}
Let $\gH$ be compact and put $\mathcal W=C(\gH;\R^k)$, equipped with the
uniform metric. For fixed $n$ and each observation time
$t\in\{0,1,\ldots,T\}$, define the hyperparameter-indexed node response
$\mathscr{S}_{t,n,i}(\rvh)=\rvs_i(t;\rvh)$, $i=1,\ldots,n$, induced by the joint
simulator in \eqref{eq:Rosenblatt}. Under Assumption
\ref{assum:flowRegularity}, Proposition \ref{prop:flowRegularity}, and the
compact-support conditions in Assumption \ref{assum:dynSysCompactMeasure},
these are $\mathcal W$-valued random elements whose sample paths are uniformly
bounded and uniformly Lipschitz in $\rvh$. On an extension of the probability
space, there are jointly distributed random variables
$\rvx_i(t)\sim\mathrm{Uniform}(0,1)$ and Borel maps
$\tilde\psi_{t,n,i}:[0,1]\times\gH\to\R^k$ for which the nodewise identity
\eqref{eq:independentNodeProcessRepresentation} holds simultaneously over the
finite observation-time and node array. The same base variable $\rvx_i(t)$
represents the complete response process over $\rvh$, while different
observation times may use different base variables. The reconstructed finite
array has the same joint law as the array generated by
\eqref{eq:Rosenblatt}, including its temporal dependence.

There are constants $B_{\mathrm S},C_T<\infty$, independent of $n$, $i$,
and the observation time, such that, for almost every base value $x$ and all
$\rvh,\rvg\in\gH$,
\begin{equation*}
\sup_{\rvh\in\gH}
\|\tilde\psi_{t,n,i}(x,\rvh)\|_2\le B_{\mathrm S},
\qquad
\|\tilde\psi_{t,n,i}(x,\rvh)
-\tilde\psi_{t,n,i}(x,\rvg)\|_2
\le C_Td(\rvh,\rvg).
\end{equation*}
For every fixed $t$ and disjoint node sets $A,B$, the construction satisfies
the dependence-preservation identity
\begin{equation}\label{eq:nodeBaseBetaPreservation}
\begin{aligned}
&\beta\!\left(\sigma\{\rvx_i(t):i\in A\},
              \sigma\{\rvx_j(t):j\in B\}\right)\\
&\quad=\beta\!\left(\sigma\{\mathscr{S}_{t,n,i}:i\in A\},
                    \sigma\{\mathscr{S}_{t,n,j}:j\in B\}\right).
\end{aligned}
\end{equation}
Consequently, fixed-time node independence is preserved in the independent
case, whereas the corresponding fixed-time $\beta$-mixing coefficients are
preserved in the weakly-dependent case. For independent candidate and
recorded copies, with $\mathscr{S}_{t,n,i}^\star$ generated by $\rvx_i^\star(t)$, the
paired base blocks satisfy
\begin{equation}\label{eq:pairedNodeBaseBetaBound}
\begin{aligned}
&\beta\!\left(
\sigma\{(\rvx_i(t),\rvx_i^\star(t)):i\in A\},
\sigma\{(\rvx_j(t),\rvx_j^\star(t)):j\in B\}\right)\\
&\quad\le
\beta\!\left(\sigma\{\mathscr{S}_{t,n,i}:i\in A\},
              \sigma\{\mathscr{S}_{t,n,j}:j\in B\}\right)
+\beta\!\left(\sigma\{\mathscr{S}_{t,n,i}^\star:i\in A\},
               \sigma\{\mathscr{S}_{t,n,j}^\star:j\in B\}\right).
\end{aligned}
\end{equation}
The two terms on its right are equal when
the recorded response array is an independent copy of the candidate array.
Applying the same deterministic maps to an independent copy of the entire
base array gives the recorded realization.
\end{props}

\begin{proof}[Proof of Proposition \ref{prop:nodewiseBaseRepresentation}]
The compact metric space $\gH$ and Euclidean target make $\mathcal W$ a
Polish space and hence a standard Borel space. For each $(t,n,i)$, choose a
Borel isomorphism $\mathfrak b_{t,n,i}$ from $\mathcal W$ to a Borel subset
of $[0,1]$, put
$Y_{t,n,i}=\mathfrak b_{t,n,i}(\mathscr{S}_{t,n,i})$, and denote its distribution
function by $\mathsf F_{t,n,i}$. For $y\in\R$, define the
left-hand limit of this CDF by
$\mathsf F_{t,n,i}(y-)\triangleq
\lim_{z\uparrow y}\mathsf F_{t,n,i}(z)
=\gP(Y_{t,n,i}<y)$; in particular,
$\mathsf F_{t,n,i}(Y_{t,n,i}-)$ denotes this limit evaluated at
$y=Y_{t,n,i}$. Let
$R_{t,n,i}\sim\mathrm{Uniform}(0,1)$ be independent over all $(t,i)$ and
independent of the entire response array. The randomized distributional
transform
\begin{equation*}
\rvx_i(t)=\mathsf F_{t,n,i}(Y_{t,n,i}-)
+R_{t,n,i}\{\mathsf F_{t,n,i}(Y_{t,n,i})
-\mathsf F_{t,n,i}(Y_{t,n,i}-)\}
\end{equation*}
is $\mathrm{Uniform}(0,1)$ and satisfies
$\mathsf F_{t,n,i}^{-1}\{\rvx_i(t)\}=Y_{t,n,i}$ almost surely, where
$\mathsf F_{t,n,i}^{-1}$ is the generalized quantile. Extending
$\mathfrak b_{t,n,i}^{-1}$ arbitrarily outside its image and defining
\begin{equation*}
\tilde\psi_{t,n,i}(x,\rvh)
=\left[\mathfrak b_{t,n,i}^{-1}
\circ\mathsf F_{t,n,i}^{-1}(x)\right](\rvh)
\end{equation*}
proves \eqref{eq:independentNodeProcessRepresentation} simultaneously for
the finite array of observation times and nodes. Because the reconstruction
is coordinatewise exact, the reconstructed responses retain the joint law of
the original response array.

For disjoint $A,B$, adjoining the mutually independent local auxiliary
variables $R_{t,n,i}$ does not change the total-variation definition of the
absolute-regularity coefficient. Since each base block is a measurable
function of its response and local auxiliary variable, data processing gives
one direction in \eqref{eq:nodeBaseBetaPreservation}. Conversely,
\eqref{eq:independentNodeProcessRepresentation} makes every response block a
measurable function of its base variable, and data processing gives the
reverse direction. This proves equality. For independent candidate and
recorded arrays, the joint law across the two runs is a product. The
subadditivity of total variation for product measures therefore gives
\eqref{eq:pairedNodeBaseBetaBound}. Since the two arrays have the same law,
the two terms on its right-hand side are equal. In particular, fixed-time
independence is preserved by the paired base blocks.

Proposition \ref{prop:flowRegularity}, compact initial and parameter supports,
and \eqref{eq:flowHyperparameterLipschitz} place every response law on a
common bounded subset of $\mathcal W$ whose elements have a common Lipschitz
constant. This subset is closed. The maps $\tilde\psi_{t,n,i}$ may therefore
be modified on a Lebesgue-null set so that their images lie in it, proving the
inherited bounds. Applying the construction to an independent copy of the
whole response array, with the same deterministic maps, gives the recorded
realization.

This construction represents the complete hyperparameter-indexed response
at each observation time, rather than merely its pointwise marginal at one
value of $\rvh$. In the weak-dependence branch the node-specific bases remain
dependent with the same coefficients in
\eqref{eq:nodeBaseBetaPreservation}; independence is introduced only through
Lemma \ref{le:finiteSequenceApprox}.
\end{proof}

\subsection{Identifiability from Local Growth}
\label{app:identifiabilityGrowth}

\begin{proof}[Proof of Lemma \ref{le:localGrowthIdentifiability}]
The two inequalities imply
$L^\star(\rvh)>L^\star(\rvh^\star)$ for every
$\rvh\ne\rvh^\star$, so the minimizer is unique. If
$0<\eps<r_0$, every point with
$d(\rvh,\rvh^\star)\ge\eps$ lies either in the local annulus, where the
excess objective is at least $c\eps^\alpha$, or in the outer region, where
it is at least $\Gamma$. If $\eps\ge r_0$, only the outer bound is needed.
Taking the infimum over $\rvh$ and then over $n\ge n_0$ proves
\eqref{eq:localGrowthEta}.
\end{proof}

\subsection{Observable Conditions for the Independent-Node Radius Bound}
\label{app:iidRadiusConditions}

This subsection gives the complete derivation of checkable sufficient
conditions for
\eqref{eq:localGrowthCondition}--\eqref{eq:strictOuterGap}. Define the
deterministic candidate-observation mean and its variance contribution by
\begin{equation*}
\begin{aligned}
\overline{\rvy}_{t,n}(\rvh)&\triangleq \E\rvy_t(\rvh),
&\overline{\rmY}_n(\rvh)&\triangleq
\bigl(\overline{\rvy}_{1,n}(\rvh),\ldots,
\overline{\rvy}_{T,n}(\rvh)\bigr),\\
V_n(\rvh)&\triangleq \frac1T\sum_{t=1}^T
\E\bigl\|\rvy_t(\rvh)-\overline{\rvy}_{t,n}(\rvh)\bigr\|^2.
\end{aligned}
\end{equation*}
For $z=(z_1,\ldots,z_T)$, put
$\|z\|_T^2=T^{-1}\sum_{t=1}^T\|z_t\|^2$. The candidate and recorded arrays
are independent, their means agree at $\rvh^\star$, and the recorded
measurement noise is centered. Consequently, its variance cancels when the
two deterministic objectives are subtracted, and
\begin{equation}\label{eq:iidPopulationExcessDecomposition}
G_n(\rvh)\triangleq L^\star(\rvh)-L^\star(\rvh^\star)
=\bigl\|\overline{\rmY}_n(\rvh)
-\overline{\rmY}_n(\rvh^\star)\bigr\|_T^2
+V_n(\rvh)-V_n(\rvh^\star).
\end{equation}

The following sufficient regime is the observable counterpart of the two
geometric conditions. Fix $r_0>0$ and let
$U=\{\rvh\in\gH:d(\rvh,\rvh^\star)<r_0\}$. Suppose, uniformly for
$n\ge n_0$, that:
\begin{enumerate}
\item $\gH$ is compact, $U$ is locally star-shaped about $\rvh^\star$, and
every segment $[\rvh^\star,\rvh]\subset U$ lies in an open set on which
$G_n$ is twice differentiable. On $U$ the metric is locally equivalent to
the Euclidean metric, with
\begin{equation*}
c_d\|\rvh-\rvh^\star\|_2
\le d(\rvh,\rvh^\star)
\le C_d\|\rvh-\rvh^\star\|_2,
\qquad 0<c_d\le C_d<\infty.
\end{equation*}
\item The maps $\overline{\rmY}_n$ converge uniformly on $\gH$ to a
continuous map $\overline{\rmY}_\infty$ that identifies $\rvh^\star$.
Writing
\begin{align}
\varepsilon_{m,n}&\triangleq
\sup_{\rvh\in\gH}\bigl\|\overline{\rmY}_n(\rvh)
-\overline{\rmY}_\infty(\rvh)\bigr\|_T,\label{eq:iidMeanUniformError}\\
s_\infty(r_0)&\triangleq
\inf_{d(\rvh,\rvh^\star)\ge r_0}
\bigl\|\overline{\rmY}_\infty(\rvh)
-\overline{\rmY}_\infty(\rvh^\star)\bigr\|_T,\notag
\end{align}
assume, after increasing $n_0$ if necessary, that
$2\varepsilon_{m,n}\le s_\infty(r_0)/2$ and, for some
$\kappa\in[0,1)$,
\begin{equation}\label{eq:iidOuterVarianceControl}
\sup_{d(\rvh,\rvh^\star)\ge r_0}
\bigl[V_n(\rvh^\star)-V_n(\rvh)\bigr]_+
\le \frac{\kappa s_\infty(r_0)^2}{4},
\qquad [x]_+\triangleq\max\{x,0\}.
\end{equation}
\item Let $r_\bullet=\sum_{t=1}^T r_t$,
$D_T=\operatorname{diag}(T^{-1}\mI_{r_1},\ldots,
T^{-1}\mI_{r_T})$, and
$J_n(\rvh)=\nabla_{\rvh}\overline{\rmY}_n(\rvh)$. There are constants
$\lambda_J>0$, $\delta_J<\sqrt{\lambda_J}$, $M_2<\infty$, and
$R_Y<\infty$ such that, for $J_n^\star=J_n(\rvh^\star)$ and every
$\rvg\in U$,
\begin{align}
(J_n^\star)^\intercal D_TJ_n^\star&\succeq\lambda_J\mI,
\label{eq:iidMeanJacobianRank}\\
\bigl\|D_T^{1/2}\{J_n(\rvg)-J_n^\star\}\bigr\|_{\mathrm{op}}
&\le\delta_J,\qquad
\max_{1\le a\le r_\bullet}
\bigl\|\nabla_{\rvh}^2\overline{\rmY}_{n,a}(\rvg)\bigr\|_{\mathrm{op}}
\le M_2,\notag\\
\bigl\|\overline{\rmY}_n(\rvg)
-\overline{\rmY}_n(\rvh^\star)\bigr\|_T&\le R_Y.
\label{eq:iidMeanLocalRange}
\end{align}
Define
\begin{equation}\label{eq:iidMeanCurvatureConstant}
\lambda_{\mathrm{mean}}
\triangleq 2(\sqrt{\lambda_J}-\delta_J)^2
-2M_2\sqrt{\frac{r_\bullet}{T}}R_Y.
\end{equation}
Finally, suppose $V_n$ is twice continuously differentiable on $U$ and
\begin{equation}\label{eq:iidVarianceLocalControl}
\nabla_{\rvh}V_n(\rvh^\star)=0,\qquad
\nabla_{\rvh}^2V_n(\rvg)\succeq-\lambda_{V,n}\mI,\qquad
\lambda_G\triangleq
\inf_{n\ge n_0}(\lambda_{\mathrm{mean}}-\lambda_{V,n})>0.
\end{equation}
\end{enumerate}

We first derive the outer gap. Compactness, continuity, and identification
give $s_\infty(r_0)>0$ whenever the outer set is nonempty. For every point in
that set, \eqref{eq:iidMeanUniformError} gives
\begin{equation}\label{eq:iidFiniteOuterSeparation}
\bigl\|\overline{\rmY}_n(\rvh)
-\overline{\rmY}_n(\rvh^\star)\bigr\|_T
\ge s_\infty(r_0)-2\varepsilon_{m,n}
\ge\frac{s_\infty(r_0)}2.
\end{equation}
Combining \eqref{eq:iidPopulationExcessDecomposition},
\eqref{eq:iidOuterVarianceControl}, and
\eqref{eq:iidFiniteOuterSeparation} yields
\begin{equation*}
G_n(\rvh)
\ge\frac{(1-\kappa)s_\infty(r_0)^2}{4}
\triangleq\Gamma_{r_0}>0,
\qquad d(\rvh,\rvh^\star)\ge r_0.
\end{equation*}

For the local region, set
$G_n^{\mathrm{mean}}(\rvh)=
\|\overline{\rmY}_n(\rvh)-
\overline{\rmY}_n(\rvh^\star)\|_T^2$. Direct differentiation gives
\begin{equation*}
\begin{aligned}
\nabla_{\rvh}G_n^{\mathrm{mean}}(\rvh)
&=2J_n(\rvh)^\intercal D_T
\{\overline{\rmY}_n(\rvh)-
\overline{\rmY}_n(\rvh^\star)\},\\
\nabla_{\rvh}^2G_n^{\mathrm{mean}}(\rvh)
&=2J_n(\rvh)^\intercal D_TJ_n(\rvh)
+2\sum_{a=1}^{r_\bullet}
\bigl[D_T\{\overline{\rmY}_n(\rvh)-
\overline{\rmY}_n(\rvh^\star)\}\bigr]_a
\nabla_{\rvh}^2\overline{\rmY}_{n,a}(\rvh).
\end{aligned}
\end{equation*}
The singular-value perturbation bound and
\eqref{eq:iidMeanJacobianRank}--\eqref{eq:iidMeanLocalRange} imply
\begin{equation*}
\begin{aligned}
\lambda_{\min}\{J_n(\rvg)^\intercal D_TJ_n(\rvg)\}
&\ge(\sqrt{\lambda_J}-\delta_J)^2,\\
\left\|2\sum_{a=1}^{r_\bullet}
\bigl[D_T\{\overline{\rmY}_n(\rvg)-
\overline{\rmY}_n(\rvh^\star)\}\bigr]_a
\nabla_{\rvh}^2\overline{\rmY}_{n,a}(\rvg)\right\|_{\mathrm{op}}
&\le2M_2\sqrt{\frac{r_\bullet}{T}}R_Y.
\end{aligned}
\end{equation*}
Hence \eqref{eq:iidMeanCurvatureConstant},
\eqref{eq:iidVarianceLocalControl}, and
\eqref{eq:iidPopulationExcessDecomposition} give
$\nabla_{\rvh}^2G_n(\rvg)\succeq\lambda_G\mI$ throughout $U$.
Moreover, $G_n(\rvh^\star)=0$ and
$\nabla_{\rvh}G_n(\rvh^\star)=0$. For
$v=\rvh-\rvh^\star$, Taylor's integral formula along the admissible segment
therefore gives
\begin{equation*}
\begin{aligned}
G_n(\rvh)
&=\int_0^1(1-s)v^\intercal
\nabla_{\rvh}^2G_n(\rvh^\star+sv)v\,\mathrm ds\\
&\ge\frac{\lambda_G}{2}\|v\|_2^2
\ge\frac{\lambda_G}{2C_d^2}d(\rvh,\rvh^\star)^2,
\qquad 0<d(\rvh,\rvh^\star)<r_0.
\end{aligned}
\end{equation*}
Thus \eqref{eq:localGrowthCondition}--\eqref{eq:strictOuterGap} hold with
\begin{equation*}
\alpha=2,\qquad c=\frac{\lambda_G}{2C_d^2},\qquad
\Gamma=\Gamma_{r_0}=\frac{(1-\kappa)s_\infty(r_0)^2}{4}.
\end{equation*}

For independent nodes, the variance controls above are themselves
checkable. Let
\begin{equation*}
\xi_{t,n,i}(\rvh)
\triangleq\varphi_{t,i}(\rvx_i(t),\rvh)
-\E\varphi_{t,i}(\rvx_i(t),\rvh).
\end{equation*}
Independence across $i$ gives
\begin{equation*}
V_n(\rvh)=\frac1{Tn^2}\sum_{t=1}^T\sum_{i=1}^n
\E\|\xi_{t,n,i}(\rvh)\|^2.
\end{equation*}
If the first two hyperparameter derivatives of the node features are
uniformly bounded and differentiation may be interchanged with expectation,
then, for constants independent of $n$,
\begin{equation*}
\sup_{\rvh\in U}
\left\{\|\nabla_{\rvh}V_n(\rvh)\|_2
+\|\nabla_{\rvh}^2V_n(\rvh)\|_{\mathrm{op}}\right\}
\le\frac{C_{V,2}}n,
\qquad
\sup_{\rvh\in\gH}V_n(\rvh)\le\frac{C_V}{n}.
\end{equation*}
Therefore one may take $\lambda_{V,n}=C_{V,2}/n$, and the outer negative
variance drift is at most $C_V/n$. These terms cannot cancel a fixed positive
mean curvature or outer separation for all sufficiently large $n$. The
stationarity condition $\nabla_{\rvh}V_n(\rvh^\star)=0$ is separate: it holds,
for example, when the finite-$n$ candidate-observation variance is locally
constant or symmetric about $\rvh^\star$. If it fails, $\rvh^\star$ need not
be the exact finite-$n$ population minimizer, and the radius statement must be
recentered before a deterministic centering bias is added.

\subsection{Proof of the Cyclic Node-Disjoint Decomposition}
\label{app:cyclicDecomposition}

\begin{proof}[Proof of Lemma \ref{le:cyclicDecomposition}]
For odd $n=2m+1$, put
$\langle a\rangle_n=1+\{(a-1)\bmod n\}$ and define
\begin{equation}\label{eq:cyclicMatching}
\mathcal M_j=
\left\{\{\langle j-s\rangle_n,\langle j+1+s\rangle_n\}:
s=0,\ldots,m-1\right\},\qquad j=1,\ldots,n.
\end{equation}
Each $\mathcal M_j$ is a matching. For every unordered edge $\{a,b\}$, the
congruence $2j+1\equiv a+b\pmod n$ has a unique solution because $n$ is
odd, after which $s$ is uniquely determined. Hence the $n$ matchings in
\eqref{eq:cyclicMatching} partition all unordered edges. For even $n$, apply
the odd construction to $K_{n+1}$ and delete the dummy vertex and its
incident edges.

For $e=\{a,b\}$, set $H_e=H_{a,b}+H_{b,a}$. H\"older's inequality and
\eqref{eq:cyclicKernelProxy} give $H_e$ a variance proxy at most
$4\nu^2d(\rvh,\rvg)^2$. Within a fixed matching, the $H_e$ are independent
because their endpoint blocks are disjoint, so a matching sum has proxy at
most $2n\nu^2d(\rvh,\rvg)^2$. Combining at most $n+1$ matching sums by the
general sub-Gaussian sum rule gives a proxy of order
$n^3\nu^2d(\rvh,\rvg)^2$. The independent diagonal kernels contribute only
order $n\nu^2d(\rvh,\rvg)^2$ and are absorbed into the universal constant
$C_S$ in \eqref{eq:cyclicFullProxy}.
\end{proof}

\begin{figure}[H]
\centering
\includegraphics[width=0.95\textwidth]{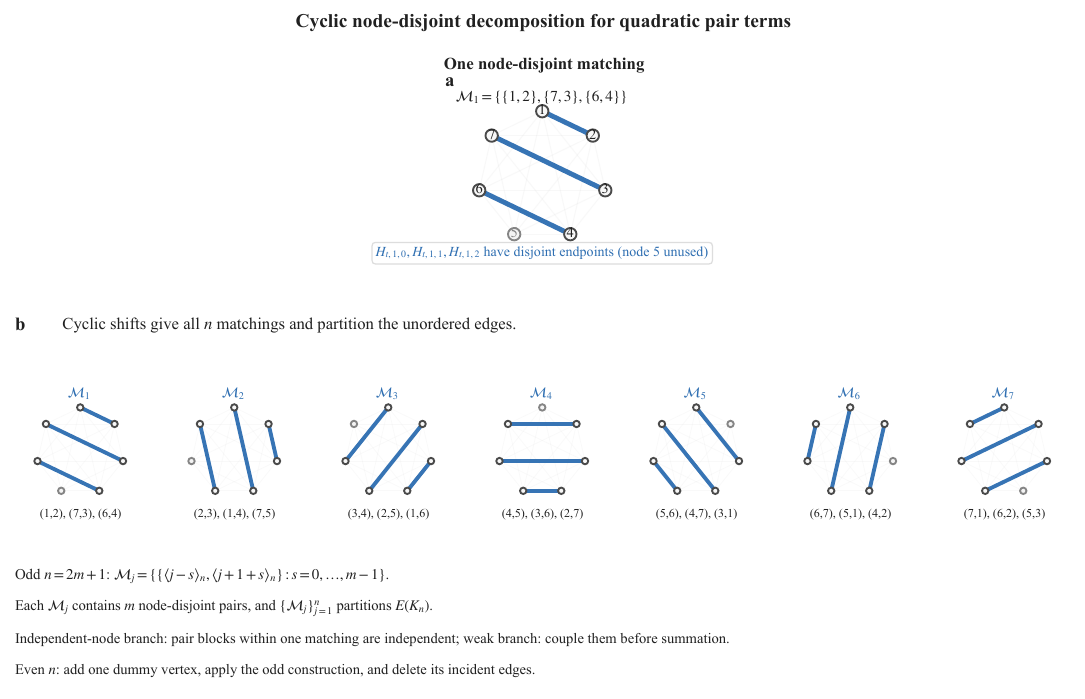}
\caption{Cyclic node-disjoint decomposition of the off-diagonal quadratic
kernel terms for $n=7$. One matching contains three node-disjoint pairs, and
the seven cyclic shifts partition all unordered edges. For even $n$, a dummy
vertex is added and its incident edges are deleted after the odd-order
construction.}
\label{fig:matchingDecompositionManuscript}
\end{figure}

\subsection{Verification of the Uniform-in-Time Mixing Assumption}
\label{app:betaVerification}

Let $\rvx_j^{\mathrm{in}}$ collect the candidate and recorded initial-state
and latent-parameter variables of node $j$; for stochastic dynamics it also
contains both realizations' innovation paths up to time $T$. Define
\begin{equation}\label{eq:inputBlockBeta}
\beta_n^{\mathrm{in}}(K,q)=\sup
\beta\left(
\sigma\{\rvx_a^{\mathrm{in}}:a\in A\},
\sigma\{\rvx_b^{\mathrm{in}}:b\in B\}\right),
\end{equation}
where the supremum is over nonempty $A,B\subseteq\{1,\ldots,n\}$ with
$|A|\le K$ and
$\min_{a\in A,b\in B}|a-b|\ge q$.

\begin{lemma}[Finite-time ancestor criterion]\label{le:ancestorMixing}
Suppose that every node set $C$ has a finite-horizon ancestor set
$\Gamma_n(C)$ such that, for every $t\in[0,T]$, the state-pair block
$\{(\rvs_i(t;\rvh),\rvs_i^\star(t)):\rvh\in\gH,\ i\in C\}$ is measurable
with respect to
$\sigma\{\rvx_j^{\mathrm{in}}:j\in\Gamma_n(C)\}$. If
$|\Gamma_n(A)|\le K_m$ whenever $|A|\le m$, and
\begin{equation}\label{eq:ancestorSeparation}
\min_{u\in\Gamma_n(A),v\in\Gamma_n(B)}|u-v|\ge q_m(k)
\end{equation}
whenever the node sets $A$ and $B$ are separated by at least $k$, then
\begin{equation}\label{eq:ancestorMixingBound}
\sup_{t\in[0,T]}\beta_{t,n}(m,k)
\le\beta_n^{\mathrm{in}}\{K_m,q_m(k)\}.
\end{equation}
In particular, if $K_m$ is bounded, $q_m(k)\ge c_qk$, and
$\beta_n^{\mathrm{in}}(K,q)\le CK^\xi q^{-\lambda}$ uniformly in $n$, then
Assumption \ref{assum:betaMixing} holds.
\end{lemma}

\begin{proof}
For each fixed $t$, the two output sigma fields in
\eqref{eq:timeBlockBeta} are subfields of their corresponding
ancestor-input sigma fields. Absolute regularity decreases when either sigma
field is replaced by a subfield. Applying this data-processing property and
\eqref{eq:ancestorSeparation}, then taking the supremum over
$t\in[0,T]$, proves \eqref{eq:ancestorMixingBound}.
\end{proof}

\begin{props}[Gaussian-kernel graph verification]
\label{prop:gaussianKernelMixing}
Consider an ordered random-connection graph containing the nearest-neighbor
path and having independent additional edges with
\begin{equation}\label{eq:gaussianEdgeProbability}
\gP_{\mathcal G,n}\{i\sim j\}
\le\rho\exp\{-c|i-j|^2\},\qquad i\ne j,
\end{equation}
as in Gaussian-kernel spatial connection models
\cite{wilsher2020soft,lang2018spatial}. Let
$W_n^{\max}=\max\{|i-j|:i\sim j\}$. Then
\begin{equation}\label{eq:gaussianMaxEdge}
\gP_{\mathcal G,n}\{W_n^{\max}>w\}
\le\rho n\sum_{r>w}e^{-cr^2}
\le\frac{\rho n}{2cw}e^{-cw^2}.
\end{equation}
If the dynamics have exact finite-horizon graph radius $R_T$ and
\eqref{eq:inputBlockBeta} has polynomial bound
$C_0K^{\xi_0}q^{-\lambda_0}$, then a Gaussian kernel truncated at the fixed
range $w_0$ satisfies
\begin{equation}\label{eq:gaussianTruncatedBeta}
\sup_{t\in[0,T]}\beta_{t,n}(m,k)
\le C_0\{m(2R_Tw_0+1)\}^{\xi_0}
(k-2R_Tw_0)^{-\lambda_0}
\le C_\beta m^{\xi_0}k^{-\lambda_0}
\end{equation}
for $k>2R_Tw_0$, after increasing $C_\beta$ for the remaining lags. Hence
Assumption \ref{assum:betaMixing} holds uniformly in $n$.

For the untruncated Gaussian kernel, if a radius-$R$ localization has
joint total-variation error at most $\tau_m(R)$, then, whenever $k>2Rw$,
\begin{equation}\label{eq:gaussianUntruncatedBeta}
\sup_{t\in[0,T]}\beta_{t,n}^{\mathcal G_n}(m,k)
\le3\tau_m(R)+
C_0\{m(2Rw+1)\}^{\xi_0}(k-2Rw)^{-\lambda_0}
\end{equation}
outside an event of probability at most
$\rho n e^{-cw^2}/(2cw)$. Thus sequences satisfying
$cw_n^2-\log n\to\infty$, $k_n/(R_nw_n)\to\infty$, and
$\tau_m(R_n)\to0$ yield the quenched growing-separation statement.
\end{props}

\begin{proof}
Summing \eqref{eq:gaussianEdgeProbability} over edge spans and nodes, and
using
\begin{equation*}
\sum_{r>w}e^{-cr^2}
\le\int_w^\infty e^{-cx^2}\,\mathrm dx
\le\frac{e^{-cw^2}}{2cw},
\end{equation*}
proves \eqref{eq:gaussianMaxEdge}. On
$\{W_n^{\max}\le w\}$, every graph-radius-$R$ ancestor of node $i$ lies in
$[i-Rw,i+Rw]$. Two blocks separated by $k$ therefore have localized
ancestor sets separated by at least $k-2Rw$, and each localized ancestor
block contains at most $m(2Rw+1)$ primitive nodes. Lemma
\ref{le:ancestorMixing} proves \eqref{eq:gaussianTruncatedBeta}; the
total-variation triangle inequality adds $3\tau_m(R)$ and proves
\eqref{eq:gaussianUntruncatedBeta}.
\end{proof}

\subsection{Limitation on the Sharpness of Objective-Based Bounds}
\label{app:sharpnessLimitation}

The fixed-accuracy results use Markov's inequality after bounding the expected
uniform objective deviation. In the independent-node setting, if the node
residuals are bounded, a Hoeffding-type bounded-difference inequality instead
gives
\begin{equation}\label{eq:hoeffdingObjective}
\gP\{\Delta_n-\E\Delta_n\ge u\}
\le\exp(-c_{\mathrm H}nu^2).
\end{equation}
For a fixed accuracy $\eps$, the identifiability threshold
$\eta(\eps)/2$ is fixed, so \eqref{eq:hoeffdingObjective} improves the
confidence dependence of Theorem \ref{thm:ConsistencyIID}. A corresponding
mixing-adjusted concentration inequality gives the analogous improvement for
Theorem \ref{thm:ConsistencyWD}. For the nonasymptotic hyperparameter bounds in
Corollaries \ref{coro:iidRadius} and \ref{coro:wdRadius}, however, the
threshold is itself of the same order as $\E\Delta_n$. Under growth of order
$\alpha$, inverting an objective deviation of order $n^{-1/2}$ still gives
$n^{-1/(2\alpha)}$; quadratic growth therefore gives $n^{-1/4}$. Replacing
Markov's inequality by \eqref{eq:hoeffdingObjective} changes the confidence
factor but not this exponent.

A Gaussian example isolates the source of the loss. Let
\begin{equation*}
X_i(h)=h+\sigma Z_i,\qquad
Y_i^\star=h^\star+\sigma Z_i^\star,
\end{equation*}
where the two families are mutually independent and i.i.d. standard normal,
and use the same $Z_i$ for every candidate $h$ in
$[h^\star-R,h^\star+R]$. With
$D_n=\sigma(\overline Z_n-\overline Z_n^\star)$,
\begin{equation*}
L_n(h)=(h-h^\star+D_n)^2,\qquad
L^\star(h)=(h-h^\star)^2+\frac{2\sigma^2}{n}.
\end{equation*}
The exact estimator is
$\hat h_n-h^\star=\Pi_{[-R,R]}(-D_n)$ and hence has the root-$n$ tail
\begin{equation*}
\gP\{|\hat h_n-h^\star|\ge\eps\}
\le2\exp\left(-\frac{n\eps^2}{4\sigma^2}\right).
\end{equation*}
By contrast,
\begin{equation*}
\sup_{|h-h^\star|\le R}|L_n(h)-L^\star(h)|
=2R|D_n|+\left|D_n^2-\frac{2\sigma^2}{n}\right|
\end{equation*}
has expectation of order $n^{-1/2}$. Combining this fixed-neighborhood
objective scale with quadratic growth yields the generic $n^{-1/4}$ radius,
even if a Hoeffding-type tail replaces Markov's inequality. Thus, the loss of
sharpness in the nonasymptotic hyperparameter bound is caused by the fixed-neighborhood
objective comparison, not by Markov's inequality.

\subsection{Nonlinear Least Squares Estimation}\label{app:leastSquares}
Although the ANN can be viewed as a dynamical system with respect to the layer forward process, in practice, most training methods of ANN only concern the last layer output, therefore no time perspective observations are available in this dynamical system construction. However, the training process of ANN can be modeled as a single node system with the whole forward process transformation as the state operator, and each time the node resamples an input sample and perform this transformation on it, and the observation is state transformed from the input sample. In the case of training process of neural networks, the dynamical system estimation task via observations degenerates to classical Least Squares Estimation (LSE) or Maximum Likelihood Estimation (MLE). 

As described in Section \ref{sec:problem_formulation}, the least squares estimation problem can be formulated as a single-node network with the \textit{resample and transform} dynamics in \eqref{eq:dynSysLS}. Thus, we can prove the consistency of least squares estimation within our framework by interpreting the sample size in classical least-squares estimation as the observation duration $T$ in our dynamical-system view. Appendix \ref{app:leastSquares} formulates and proves consistency for the i.i.d. least-squares problem. The corresponding weak-dependence result follows the proof strategy of Section \ref{sec:EstimationWD} and is omitted for brevity. To integrate the least-squares estimation problem into Problem \ref{defn:ProbGen}, notice that the hyperparameter to be estimated is $\rvh^\star = \rvh_{\mathrm{dyn}}^\star \in \gH$ while the probabilistic hyperparameter space $\gH_{\mathrm{prob}}$ is null space. The whole observation till time $t$ is $\rmY^\star = (\rvy_1^\star, \ldots, \rvy_T^\star)$, where $\rvy_t^\star = \rvs(t) = f(\rvx(t); \rvh^\star)$. Thus, Assumption \ref{assum:obsOperator} is naturally satisfied. Therefore, we formulate least-squares estimation within Problem \ref{defn:ProbGen}. The i.i.d. formulation is as follows:

\begin{problem}[\textbf{Least Square Estimation Problem under Independent and identically Distributed Condition}]\label{defn:ProbLS}
    Consider the \textit{resample and transform} dynamics \eqref{eq:dynSysLS}, the hyperparameter space is simply $\gH = \gH_{\mathrm{dyn}}$ while the probabilistic hyperparameter space in \ref{defn:ProbGen} is null space. The observation is defined by 
    \begin{equation*}
        \rvy_t(\rvh) = \rvs(t) \triangleq f(\rvx(t), \rvh), \quad t = 1,2,\dots, T,
    \end{equation*}
    and the true observation is generated with the true hyperparameter $\rvh^\star$ with noise, that is,
    \begin{equation*}
        \rvy_t^\star = \rvs^\star(t) \triangleq f(\rvx(t), \rvh^\star) + \omega_{(t)}, \quad t = 1,2,\dots, T,
    \end{equation*}
    where $f:\gX \times \gH \to \gY \subset \R^r$ is an transformation operator, $\rvx(t) \in \gX \subset \R^k$ are i.i.d. random vectors, the parameter $\rvh^\star \in \gH \subset \R^h$ and the noise $\omega_T = (\omega_{(1)}, \omega_{(2)}, \dots, \omega_{(T)})$, where $\{\omega_{(t)}\}_{t=1,\dots,T}$ are i.i.d. Gaussian random vector with zero mean and variance of $\sigma^2 \mI_r$. Since there is only one single node in this system, the observation operator naturally satisfies the Assumption \ref{assum:obsOperator}.

    We denote $\rmX_T = (\rvx(1), \dots, \rvx(T))$ and the loss function as $\|\rvy_t(\rvh) - \rvy_t^\star\|_2^2 = \|f(\rvx(t), \rvh) - f(\rvx(t), \rvh^\star) - \omega_{(t)}\|_2^2 \triangleq l(\rvx(t), \rvh)$.

    Then, the empirical objective function $L_1$ merging the observation thus takes the form of 
    \begin{equation*}
        L_1(\rvh) = \frac{1}{T} \sum\limits_{t=1}^T l(\rvx(t),\rvh).
    \end{equation*}
    
    And the deterministic objective function is the expectation of the empirical objective function, that is
    \begin{equation*}
        L_1^\star(\rvh) = \E [L_1(\rvh)] = \E [l\left(\rvx_1, \rvh\right)].
    \end{equation*}
\end{problem}

The empirical objective function $L_1(\rvh)$ carries the subscript $1$ because the system contains only one node. We omit this node subscript throughout the remainder of Appendix \ref{app:leastSquares}, writing $L(\rvh)$ and $\hat{\rvh}$ for simplicity; the number of i.i.d. observations is consistently denoted by $T$.

\begin{remark}
    In the problem formulation \ref{defn:ProbLS}, when the operator $f$ is nonlinear, the problem is a Nonlinear Least Square (NLS) estimation problem. When the operator $f$ is linear and the dimension of hyperparameter space $\gH$ is the same as the dimension of the sample space $\gX$, that is, $h = k$, then $f(\rvx(t), \rvh) = \rvx(t)^\intercal \rvh$ and the problem is a Linear Least Square (LLS) estimation problem. 
\end{remark}

We first record the Rademacher symmetry argument used both above and for
Problem \ref{defn:ProbLS}.

\begin{lemma}[Rademacher symmetry]\label{le:Symmetry}
Let $X_1,\ldots,X_m$ be independent random elements and let
$q_i(\rvh)=q_i(X_i,\rvh)$ be integrable real-valued processes indexed by
$\rvh\in\gH$. If $q_i'$ are independent copies and
$\epsilon_1,\ldots,\epsilon_m$ are independent Rademacher variables, then
\begin{align}
\E\sup_{\rvh\in\gH}
\left|\sum_{i=1}^m\{q_i(\rvh)-\E q_i(\rvh)\}\right|
&\le \E\sup_{\rvh\in\gH}
\left|\sum_{i=1}^m\{q_i(\rvh)-q_i'(\rvh)\}\right|\notag\\
&\le2\E\sup_{\rvh\in\gH}
\left|\sum_{i=1}^m\epsilon_iq_i(\rvh)\right|.
\label{eq:generalRademacherSymmetry}
\end{align}
The same conclusion holds without absolute values when only a one-sided
supremum is required.
\end{lemma}

\begin{proof}
Conditioning on $X_1,\ldots,X_m$ and applying Jensen's inequality gives
the first inequality in \eqref{eq:generalRademacherSymmetry}. For each $i$,
the process $q_i-q_i'$ has the same law after the pair $(q_i,q_i')$ is
swapped. Independence of the pairs therefore gives
\begin{equation*}
\left\{q_i-q_i':1\le i\le m\right\}
\overset{\mathrm d}=\left\{\epsilon_i(q_i-q_i'):1\le i\le m\right\}.
\end{equation*}
The triangle inequality for the supremum, followed by equality in law of
$(q_i',-\epsilon_i)$ and $(q_i,\epsilon_i)$, proves the second inequality.
The one-sided statement follows by the same calculation.
\end{proof}

We now state a consistency theorem for the least-squares estimation problem in Problem \ref{defn:ProbLS}.
\begin{thm}[Consistency]\label{thm:ConsistencyLS}
    Consider Problem \ref{defn:ProbLS} with $T$ independent and identically distributed observations, and suppose that the algorithm $\gA$ returns a minimizer of the empirical objective function as required by Assumption \ref{assum:minimizer}. Assume that
    \begin{enumerate}
        \item $L^\star$ is uniquely minimized at $\rvh^\star$;
        \item $l(\rvx,\rvh)$ is continuous in $\rvh\in\gH$ for every $\rvx$;
        \item for i.i.d. Rademacher variables $\eps_1,\ldots,\eps_T$ independent of the data, the process $\{\eps_t l(\rvx(t),\rvh)\}_{\rvh\in\gH}$ has sub-Gaussian increments with variance proxy $M^2d(\rvh,\rvg)^2$, uniformly in $t$, for some $M>0$;
        \item the anchor loss has a finite second moment, $\E[l(\rvx(1),\rvh^\star)^2]\le B_\star^2<\infty$;
        \item $\gH\subset\R^h$ is compact.
    \end{enumerate}
    Then, for every $\eps>0$,
    \begin{equation*}
        \gP\left(d(\hat{\rvh},\rvh^\star)\ge\eps\right)
        \le \frac{C}{\eta(\eps)\sqrt{T}},
    \end{equation*}
    where
    \begin{equation*}
        C=96\sqrt{3h}\diam(\gH)M+2B_\star,
        \qquad
        \eta(\eps)=\frac{1}{2}\left(
        \inf_{d(\rvh,\rvh^\star)\ge\eps}L^\star(\rvh)
        -L^\star(\rvh^\star)\right)>0.
    \end{equation*}
    In particular, both $C$ and $\eta(\eps)$ are independent of the observation duration $T$.
\end{thm}

\begin{proof}
    Continuity and compactness, together with the unique-minimizer condition, imply by Lemma \ref{le:identifiable} that $\eta(\eps)>0$. Set
    \begin{equation*}
        \Delta_T\triangleq\sup_{\rvh\in\gH}|L(\rvh)-L^\star(\rvh)|.
    \end{equation*}
    The deterministic comparison used in the proof of Lemma \ref{le:SIMarkov} gives
    \begin{equation}\label{eq:markov2}
        \gP\left(d(\hat{\rvh},\rvh^\star)\ge\eps\right)
        \le \frac{1}{\eta(\eps)}\E[\Delta_T].
    \end{equation}

    To bound the right-hand side of \eqref{eq:markov2}, define the symmetrized process
    \begin{equation*}
        Z_T(\rvh)\triangleq\frac{1}{T}\sum_{t=1}^T
        \eps_t l(\rvx(t),\rvh).
    \end{equation*}
    Lemma \ref{le:Symmetry} yields
    \begin{equation*}
        \E[\Delta_T]
        \le 2\E\left[\sup_{\rvh\in\gH}|Z_T(\rvh)|\right].
    \end{equation*}

    Anchor the process at the true hyperparameter by setting
    \begin{equation*}
        W_T(\rvh)\triangleq Z_T(\rvh)-Z_T(\rvh^\star),
        \qquad W_T(\rvh^\star)=0.
    \end{equation*}
    For any $\rvh,\rvg\in\gH$, independence and the assumed sub-Gaussian increment bound imply
    \begin{equation*}
        \E\exp\left\{\lambda\bigl(W_T(\rvh)-W_T(\rvg)\bigr)\right\}
        \le \exp\left\{\frac{\lambda^2M^2d(\rvh,\rvg)^2}{2T}\right\}.
    \end{equation*}
    Thus, $W_T$ is a separable sub-Gaussian process on $\gH$ with metric
    $\widetilde d(\rvh,\rvg)=Md(\rvh,\rvg)/\sqrt{T}$. Dudley's inequality and Lemma \ref{le:BoundedEntropy} give
    \begin{equation*}
        \E\left[\sup_{\rvh\in\gH}W_T(\rvh)\right]
        \le\frac{24\sqrt{3h}\diam(\gH)M}{\sqrt{T}}.
    \end{equation*}
    The same bound applies to $-W_T$. Since both one-sided suprema are nonnegative at the anchor $\rvh^\star$,
    \begin{equation*}
        \E\left[\sup_{\rvh\in\gH}|W_T(\rvh)|\right]
        \le\frac{48\sqrt{3h}\diam(\gH)M}{\sqrt{T}}.
    \end{equation*}

    The anchor itself is controlled separately by the stated moment condition:
    \begin{equation*}
        \E|Z_T(\rvh^\star)|
        \le\sqrt{\E[Z_T(\rvh^\star)^2]}
        =\sqrt{\frac{\E[l(\rvx(1),\rvh^\star)^2]}{T}}
        \le\frac{B_\star}{\sqrt{T}}.
    \end{equation*}
    Therefore,
    \begin{equation*}
        \E[\Delta_T]
        \le\frac{96\sqrt{3h}\diam(\gH)M+2B_\star}{\sqrt{T}}.
    \end{equation*}
    Substitution into \eqref{eq:markov2} proves the claim.
\end{proof}

When the observation operator $f$ in Problem \ref{defn:ProbLS} is nonlinear, then problem turns to a Nonlinear Least Square (NLS) estimation problem. The consistency of NLS estimation problem can be derived from the above consistency Theorem \ref{thm:ConsistencyLS} with the following corollary.

\begin{coro}[Consistency for Nonlinear Least Square Estimation]\label{coro:NLS}
    Consider the nonlinear least-squares problem in Problem \ref{defn:ProbLS}. Assume that
    \begin{enumerate}
        \item for every $\rvh\ne\rvh^\star$,
        $\E\|f(\rvx(1),\rvh)-f(\rvx(1),\rvh^\star)\|^2>0$;
        \item $f$ is bounded and uniformly Lipschitz continuous in $\rvh$ with a common constant $M>0$;
        \item $\gH$ is compact.
    \end{enumerate}
    Then, for every $\eps>0$,
    \begin{equation*}
        \gP\left(d(\hat{\rvh}, \rvh^\star) \ge \eps \right) \le \frac{C}{\eta(\eps)\sqrt{T}}.
    \end{equation*}
    where $C$ is a constant, and we can choose $\eta(\eps) = \frac{1}{2}\left(\inf_{d(\rvh, \rvh^\star) \ge \eps} L^\star(\rvh) - L^\star(\rvh^\star)\right)$, which is independent of the sample size.
\end{coro}

\begin{proof}
    We verify the identifiability, increment, and anchor-moment conditions of Theorem \ref{thm:ConsistencyLS}.

    The empirical objective function of Nonlinear Least Square is 
    \begin{equation*}
        L(\rvh)  = \frac{1}{T} \sum\limits_{t=1}^T \|f(\rvx(t), \rvh) - f(\rvx(t), \rvh^\star) - \omega_{(t)}\|^2.
    \end{equation*}
    
    Denote $l(\rvx(t), \rvh) = \|f(\rvx(t), \rvh) - f(\rvx(t), \rvh^\star) - \omega_{(t)}\|^2$. The deterministic objective function 
    \begin{equation*}
        L^\star(\rvh) = \E[L(\rvh)] = \E[l(\rvx(t), \rvh)] = \E[\|f(\rvx(t), \rvh) - f(\rvx(t), \rvh^\star)\|^2] + r \sigma^2.
    \end{equation*}
    The first condition of the corollary implies that $L^\star(\rvh)>L^\star(\rvh^\star)$ for every $\rvh\ne\rvh^\star$, so the population objective is uniquely minimized at $\rvh^\star$.

    Let $\eps_1, \dots, \eps_T$ be i.i.d. Rademacher random variables independent of $\rmX_T$. We aim to verify that $\{\eps_t l(\rvx(t), \rvh)\}_{\rvh \in \gH}$ is a sub-Gaussian process. Note that 
    \begin{equation*}
        \begin{aligned}
            \eps_t l(\rvx(t), \rvh) - \eps_t l(\rvx(t), \rvg) =& \eps_t \langle f(\rvx(t), \rvh) - f(\rvx(t), \rvg), f(\rvx(t), \rvh) + f(\rvx(t), \rvg) - 2f(\rvx(t), \rvh^\star) - 2 \omega_{(t)}\rangle \\
            =& \Xi_1(\rvx(t), \rvh, \rvg) + \Xi_2(\rvx(t), \rvh, \rvg).
        \end{aligned}
    \end{equation*}
    where we denote 
    \begin{equation*}
\begin{aligned}
        \Xi_1(\rvx(t), \rvh, \rvg) \triangleq& \eps_t \langle f(\rvx(t), \rvh) - f(\rvx(t), \rvg), f(\rvx(t), \rvh) + f(\rvx(t), \rvg) - 2f(\rvx(t), \rvh^\star)\rangle, \\
        \Xi_2(\rvx(t), \rvh, \rvg) \triangleq& -2\eps_t \langle f(\rvx(t), \rvh) - f(\rvx(t), \rvg), \omega_{(t)}\rangle.
    \end{aligned}
\end{equation*}

    For the first term $\Xi_1(\rvx(t), \rvh, \rvg)$, suppose $f(\rvx(t), \rvh)$ is bounded and uniformly Lipschitz continuous with respect to $\rvh$ with some common constant $M>0$, then by Cauchy-Schwarz inequality, we have
    \begin{equation*}
        |\Xi_1(\rvx(t), \rvh, \rvg)| \le \|f(\rvx(t), \rvh) - f(\rvx(t), \rvg)\|\cdot \|f(\rvx(t), \rvh) + f(\rvx(t), \rvg) - 2f(\rvx(t), \rvh^\star)\| \le 4\|f\|_\infty M d(\rvh, \rvg).
    \end{equation*}
    Thus, Hoeffding's inequality yields that $\Xi_1(\rvx(t), \rvh, \rvg)$ is sub-Gaussian with variance proxy of $(4\|f\|_\infty M d(\rvh, \rvg))^2$.

    As for the second term $\Xi_2(\rvx(t), \rvh, \rvg)$, notice that the gaussian random variable $\omega_{(t)}$ and the Rademacher random variable $\eps_t$ are both symmetric, thus the law of $\Xi_2(\rvx(t), \rvh, \rvg)$ is same as the law of $\Xi_2^\prime(\rvx(t), \rvh, \rvg) \triangleq 2\langle f(\rvx(t), \rvh) - f(\rvx(t), \rvg), \omega_{(t)}\rangle$. Therefore, we can calculate the expectation via the properties of gaussian random variable.
    \begin{equation*}
        \E\left[\ef^{\lambda \Xi_2(\rvx(t), \rvh, \rvg)}\right] = \E\left[\ef^{\lambda \Xi_2^\prime(\rvx(t), \rvh, \rvg)}\right] = \E\left[\E\left[\ef^{\lambda \Xi_2^\prime(\rvx(t), \rvh, \rvg)}\mid \rvx(t)\right]\right].
    \end{equation*}
    Since $\omega_{(t)}$ is gaussian with mean $0$ and variance $\sigma^2\mI_r$, direct calculation yields
    \begin{equation*}
        \E\left[\ef^{\lambda \Xi_2^\prime(\rvx(t), \rvh, \rvg)}\mid \rvx(t)\right] = \ef^{4\lambda^2\sigma^2\|f(\rvx(t), \rvh) - f(\rvx(t), \rvg)\|^2/2} \le \ef^{4\lambda^2\sigma^2(Md(\rvh,\rvg))^2/2}.
    \end{equation*}
    That implies
    \begin{equation*}
        \E\left[\ef^{\lambda \Xi_2(\rvx(t), \rvh, \rvg)}\right] \le \ef^{\lambda^2(2\sigma M d(\rvh,\rvg))^2/2},
    \end{equation*}
    which means $\Xi_2(\rvx(t), \rvh, \rvg)$ is sub-Gaussian with variance proxy of $(2\sigma M d(\rvh,\rvg))^2$.

    Combining the two estimations, we have that $\eps_t l(\rvx(t), \rvh) - \eps_t l(\rvx(t), \rvg)$ is sub-Gaussian with variance proxy of $[2(2\|f\|_\infty +\sigma) M d(\rvh,\rvg)]^2$.

    This result means that $\{\eps_t l(\rvx(t), \rvh)\}_{\rvh \in \gH}$ is a sub-Gaussian process with variance proxy of $[2(2\|f\|_\infty +\sigma) M]^2 d(\rvh,\rvg)^2$. Moreover, at the anchor,
    \begin{equation*}
        l(\rvx(t),\rvh^\star)=\|\omega_{(t)}\|^2,
        \qquad
        \E[l(\rvx(t),\rvh^\star)^2]=\sigma^4r(r+2)<\infty,
    \end{equation*}
    because $\omega_{(t)}\sim\mathcal N(0,\sigma^2\mI_r)$. Thus, the required anchor moment is finite.

    Then, Theorem \ref{thm:ConsistencyLS} can be applied to the nonlinear least squares estimation problem to complete the proof.
\end{proof}

For Linear Least Square Estimation, it is a special case of Nonlinear Least Square Estimation where the observation operator $f$ is linear, that is, $f(\rvx(t), \rvh) = \rvx(t)^\intercal \rvh$. The consistency of Linear Least Square Estimation can be derived from the above corollary \ref{coro:NLS} with the following corollary.

\begin{coro}[Consistency for Linear Least Square Estimation]\label{coro:LLS}
    Consider the linear least-squares problem in Problem \ref{defn:ProbLS}, where $f(\rvx(t),\rvh)=\rvx(t)^\intercal\rvh$. Assume that
    \begin{enumerate}
        \item $\E[\rvx(t)\rvx(t)^\intercal]$ is positive definite;
        \item $\|\rvx(t)\|\le M$ almost surely for some $M>0$ and every $t=1,\ldots,T$;
        \item $\gH$ is compact and contains $\rvh^\star$.
    \end{enumerate}
    Define the constrained linear least-squares estimator by
    \begin{equation}\label{eq:LLSestimator}
        \hat{\rvh}\in\argmin_{\rvh\in\gH}
        \frac{1}{T}\sum_{t=1}^T
        \bigl(\rvx(t)^\intercal\rvh-\rvy_t^\star\bigr)^2.
    \end{equation}
    Then, for every $\eps>0$,
    \begin{equation*}
        \gP\left(d(\hat{\rvh}, \rvh^\star) \ge \eps \right) \le \frac{C}{\eta(\eps)\sqrt{T}}.
    \end{equation*}
    where $C$ is a constant, and we can choose $\eta(\eps) = \frac{1}{2}\left(\inf_{d(\rvh, \rvh^\star) \ge \eps} L^\star(\rvh) - L^\star(\rvh^\star)\right)$, which is independent of $T$. If $T\ge h$, $\operatorname{rank}(\rmX_T)=h$, and the unconstrained ordinary least-squares solution belongs to $\gH$, then the estimator in \eqref{eq:LLSestimator} also has the familiar representation
    \begin{equation*}
        \hat{\rvh}=\left(\rmX_T\rmX_T^\intercal\right)^{-1}
        \rmX_T\rmY(\rvh^\star).
    \end{equation*}
\end{coro}

\begin{proof}
    The empirical objective function for the linear least-squares problem is
    \begin{equation*}
        L(\rvh)  = \frac{1}{T} \sum\limits_{t=1}^T (\rvx(t)^\intercal(\rvh - \rvh^\star) - \omega_{(t)})^2.
    \end{equation*}
    
    Denote $l(\rvx(t), \rvh) = (\rvx(t)^\intercal(\rvh - \rvh^\star) - \omega_{(t)})^2$. The deterministic objective function is
    \begin{equation*}
        L^\star(\rvh) = \E[L(\rvh)] = \E[l(\rvx(t), \rvh)] = (\rvh - \rvh^\star)^\intercal \E[\rvx(t) \rvx(t)^\intercal] (\rvh - \rvh^\star) + \sigma^2.
    \end{equation*}
    Since the non-central second moment $\E[\rvx(t) \rvx(t)^\intercal]$ is positive definite, $L^\star(\rvh)$ is uniquely minimized by $\rvh^\star$, which means that the problem is identifiable. This condition plays the same role as the condition 1. in corollary \ref{coro:NLS}.

    For the second condition in corollary \ref{coro:NLS}, the boundness of $\rvx(t)$ guarantees the uniformly Lipschitz continuity of $f(\rvx(t), \rvh) = \rvx(t)^\intercal \rvh$ with respect to $\rvh$, and the boundness of $f$ is guaranteed by the boundness of $\rvx(t)$ and the compactness of $\gH$. Therefore, the second condition in corollary \ref{coro:NLS} is satisfied. 

    Finally, compactness of $\gH$ ensures that the constrained minimizer in \eqref{eq:LLSestimator} exists and satisfies Assumption \ref{assum:minimizer}. Corollary \ref{coro:NLS} therefore proves the result. The displayed inverse formula is asserted only under the additional finite-sample full-row-rank and feasibility conditions stated in the corollary.

\end{proof}

\end{appendix}
\bibliographystyle{IEEEtran}
\bibliography{sample}

\section{Biography Section}

\begin{IEEEbiographynophoto}{Yi Yu}
Yi Yu received the B.S. degree in Mathematics from Wuhan University, China, in 2021. He is currently pursuing
the Ph.D. degree with the School of Mathematical Sciences, Fudan University, Shanghai, China. His current research interests include information geometry and machine learning theory.
\end{IEEEbiographynophoto}
\vspace{11pt}

\begin{IEEEbiographynophoto}{Yubo Hou}
Yubo Hou received the B.S. degree in mathematics from Wuhan University, China, in 2022. She is currently pursuing the Ph.D. degree with the School of Mathematical Sciences, Fudan University, Shanghai, China. Her current research interests include multimodal data assimilation and encoding/decoding of Digital Twin Brain (DTB).
\end{IEEEbiographynophoto}
\vspace{11pt}

\begin{IEEEbiographynophoto}{Yinchong Wang}
Yinchong Wang received the B.S. degree in Mathematics from Shandong University, China, in 2020, and the Ph.D. degree in Applied Mathematics from Fudan University, China, in 2025. He is currently employed at China Mobile Internet Company.
\end{IEEEbiographynophoto}
\vspace{11pt}

\begin{IEEEbiographynophoto}{Nan Zhang}
Nan Zhang obtained a B.S. in Mathematics and a B.A. in Economics from Peking University, Beijing, China, in 2010. He earned the Ph.D. in Statistics from Texas A\&M University and held a postdoctoral research position for one year before joining Fudan. He is currently Associate Professor affiliated with the Institute of Science and Technology for Brain-Inspired Intelligence and the School of Data Science, Fudan University. His research lies at the intersection of statistics and artificial intelligence, with a focus on nonparametric methods, statistical machine learning, and functional data analysis.
\end{IEEEbiographynophoto}
\vspace{11pt}

\begin{IEEEbiographynophoto}{Jianfeng Feng}
Jianfeng Feng (Senior Member, IEEE) received the
BS, MS, and PhD degrees from the Department of
Probability and Statistics, Peking University, China.
He is the chair professor with the Shanghai National
Centre for Mathematic Sciences and the Dean with
the Brain-Inspired AI Institute, Fudan University. He
leads the DTB project. He has been developing new
mathematical, statistical, and computational theories
and methods to meet the challenges raised in neuroscience and mental health research.
\end{IEEEbiographynophoto}
\vspace{11pt}

\begin{IEEEbiographynophoto}{Wenlian Lu}
Wenlian Lu (Senior Member, IEEE) received the
B.S. degree in Mathematics and the Ph.D. degree
in Applied Mathematics from Fudan University,
Shanghai, China, in 2000 and 2005, respectively.
He was a Postdoctoral Fellow with the Max
Planck Institute for Mathematics in Science, Leipzig,
Germany, from 2005 to 2007. He was a MarieCurie International Incoming Research Fellow with
the Department of Computer Sciences, University
of Warwick, Coventry, U.K., from 2012 to 2014.
He is currently a Professor with the School of
Mathematical Sciences and the Institute for Science and Technology of BrainInspired AI, Fudan University. His current research interests include neural
networks, cybersecurity dynamics, computational systems biology, nonlinear
dynamical systems, and complex systems.
Prof. Lu has served as an Associate Editor for IEEE TRANSACTIONS ON
NEURAL NETWORKS AND LEARNING SYSTEMS from 2013 to 2019 and
Neurocomputing from 2010 to 2015.
\end{IEEEbiographynophoto}
\vfill

\end{document}